\documentclass[11pt]{article}
\usepackage{latexsym,amsmath,amsthm,amssymb,graphicx,mathrsfs,amsfonts,yfonts,calrsfs,eufrak,bm,stackrel}
\input epsf
\usepackage{tikz}
\usepackage{relsize}
\usepackage{mathtools}
\usepackage{titlesec}
\usepackage{tikz}
\usetikzlibrary{calc,shapes.geometric,patterns,positioning,decorations.markings}
\usepackage{verbatim}
\usetikzlibrary{backgrounds}
\usetikzlibrary{arrows, shapes, shapes.geometric}

\tikzset{
  circlenode/.style={circle,draw=black!100,very thick,inner sep=1pt, minimum size=3mm},
  squarenode/.style={draw=black!100, very thick,inner sep=1.5pt, minimum size=6mm},
  trianglenode/.style={        draw,
        shape border rotate=00,
        regular polygon,
        regular polygon sides=3,
        draw=black!100,very thick,inner sep=1.0pt, minimum size=4mm},
 bigblue/.style={circle, draw=blue!80,fill=blue!40,thick, inner sep=1.5pt, minimum size=5mm},
  bigred/.style={circle, draw=red!80,fill=red!40,thick, inner sep=1.5pt, minimum size=5mm},
  bluevertex/.style={circle, draw=blue!100,fill=blue!100,thick, inner sep=0pt, minimum size=2mm},
  redvertex/.style={circle, draw=red!100,fill=red!100,thick, inner sep=0pt, minimum size=2mm},
  blackvertex/.style={circle, draw=black!100,fill=black!100,thick, inner sep=0pt, minimum size=1mm},
    state/.style={circle, draw=black!100,fill=white!100,thick, inner sep=1pt, minimum size=5mm},  
  haltstate/.style={circle, draw=black!100,fill=white!100,double,thick, inner sep=1pt, minimum size=5mm},  
}

\titleformat{\paragraph}
{\normalfont\normalsize\bfseries}{\theparagraph}{1em}{}
\titlespacing*{\paragraph}
{0pt}{3.25ex plus 1ex minus .2ex}{1.5ex plus .2ex}
\DeclareFontEncoding{LS1}{}{}
\DeclareFontSubstitution{LS1}{stix}{m}{n}
\DeclareSymbolFont{symbols3}{LS1}{stixbb}{m}{n}
\DeclareMathSymbol{\bigslopedvee}{\mathbin}{symbols3}{"A7}

        {\qquad\hspace*{\fill}$\square$\end{trivlist}}

        {\qquad\hspace*{\fill}\end{trivlist}}

\newcounter{subtheor}[section]
\newcounter{sublem}[section]
\newcounter{subcor}[section]

\newcounter{subprop}[section]
\newcounter{subconj}[section]

\newcounter{fig}

\newcounter{clai}
\newcounter{exa}
\newcounter{theorem}[section]
\renewcommand{\thetheorem}{\arabic{section}.\arabic{theorem}}
\newcounter{nona}[theorem]
\newcounter{nonanona}[theorem]
\newcounter{rem}[section]

\renewcommand{\thenona}{\Alph{nona}}
\renewcommand{\thenonanona}{\alph{nonanona}}
\renewcommand{\thenonanona}{\arabic{section}.\arabic{nonanona}}

\newenvironment{nonamenoname}{\begin{trivlist}\item[]\refstepcounter{nonanona}%
        {\bf (\thenonanona)\ \ \ }\nobreak\noindent\sl\ignorespaces}{%
        \ifvmode\smallskip\fi\end{trivlist}}

\newenvironment{theorem}{\begin{trivlist}\item[]\refstepcounter{theorem}%
        {\bf\thetheorem\ Theorem}\par\nobreak\noindent\sl\ignorespaces}{%
        \ifvmode\smallskip\fi\end{trivlist}}

\newenvironment{noname}{\begin{trivlist}\item[]\refstepcounter{nona}%
        {\bf (\thenona)\ \ \ }\nobreak\noindent\sl\ignorespaces}{%
        \ifvmode\smallskip\fi\end{trivlist}}

\newenvironment{theoremplus}[1]{\begin{trivlist}\item[]%
        \refstepcounter{theorem}{\bf\thetheorem\ Theorem} {\rm (\,#1\,)}%
        \par\nobreak\noindent\sl\ignorespaces}{%
        \ifvmode\smallskip\fi\end{trivlist}}

\newenvironment{lemma}{\begin{trivlist}\item[]\refstepcounter{theorem}%
        {\bf\thetheorem\ Lemma}\par\nobreak\noindent\sl\ignorespaces}{%
        \ifvmode\smallskip\fi\end{trivlist}}

\newenvironment{corollary}{\begin{trivlist}\item[]\refstepcounter{theorem}%
        {\bf\thetheorem\ Corollary}\par\nobreak\noindent\sl\ignorespaces}{%
        \ifvmode\smallskip\fi\end{trivlist}}

\newenvironment{conjecture}{\begin{trivlist}\item[]%
        \refstepcounter{theorem}{\bf\thetheorem\ Conjecture}\par%
        \nobreak\noindent\sl\ignorespaces}{%
        \ifvmode\smallskip\fi\end{trivlist}}
\newenvironment{conjectureplus}[1]{\begin{trivlist}\item[]%
        \refstepcounter{theorem}{\bf\thetheorem\ Conjecture} %
        {\rm(\,#1\,)}\par\nobreak\noindent\sl\ignorespaces}{%
        \ifvmode\smallskip\fi\end{trivlist}}

\newenvironment{observation}{\begin{trivlist}\item[]\refstepcounter{theorem}%
        {\bf\thetheorem\ Observation}\par\nobreak\noindent\ignorespaces}{%
        \ifvmode\smallskip\fi\end{trivlist}}

        {\endlist\unskip}

        {\endtrivlist}
\newcommand{\wideitem}[1]{\leavevmode\hangindent\mathindent\noindent%
        \hbox to \mathindent{#1\hfil}\ignorespaces}

\renewcommand{\emptyset}{\mathchar"001F}
\newcommand{\eop}{\rule{1.4ex}{1.4ex}}

\newcommand{\C}{{\cal C}}

\newcommand{\sms}{\vspace*{.2in}}
\newcommand{\ssms}{\vspace*{.1in}}

\title{\bf Short Rainbow Circuits in Regular Matroids}

\author{Sean McGuinness \\ Dept. of Mathematics\\ Thompson Rivers University\\
McGill Road, Kamloops BC\\ V2C5N3 Canada\\ email: smcguinness@tru.ca}

\date{}

\begin{document}

\maketitle

\begin{abstract}
DeVos et al conjectured that if $M$ is a simple, regular matroid and $c$ is a colouring of the elements of $M$ with $r(M)+1$ colours, where each colour class has at least two elements, then $M$ contains a rainbow circuit of size at most 
$\lceil \frac {r(M)+1}2 \rceil.$  We prove this conjecture by showing that for all such regular matroids there are four rainbow circuits $C_i,\ i = 1,2,3,4$ for which $\sum_i |C_i| \le 2r(M) +4$ and for which no element of $M$ belongs to more than two of the circuits.

\bigskip

\end{abstract}

\section{Introduction}

For any positive integer $k$, $[k]$ shall denote the set $\{ 1, \dots ,k \}$ and for any integers $a < b,$ $[a,b]$ will denote the set $\{ a, a+1, \dots ,b \}.$  For any graph $G$, we let $\nu (G)$ and $\varepsilon (G)$ denote the number of vertices and edges, respectively in $G.$

A {\bf colouring} for a graph $G$ is defined to be a function $c:E(G) \rightarrow S$ for some set $S$.  We refer to the pair $(G,c)$ as a {\bf coloured graph}. If $|S| =k,$ then we refer to $c$ as a $\mathbf{k}${\bf - colouring}.  For a colouring $c$ of a graph $G$ and a subset $A \subseteq E(G)$, we define $c(A) := \{ c(e) \ \big| \ e\in A \}$ and we define $c(G) := c(E(G));$ that is, the set of distinct colours appearing in $G.$   For a colouring $c: E(G) \rightarrow S$ and for all $s\in S$ we let $c^{-1}(s) = \{ e\in E(G) \ \big| \ c(e) = s \}.$  We refer to $c^{-1}(s)$ as the {\bf colour class} for colour $s$.  For a colouring $c$ of $G$, we say that a subset $A \subseteq E(G)$ is a {\bf rainbow-subset} if $c$ restricted to $A$ is injective.  A subgraph $H$ of $G$ is a {\bf rainbow subgraph} if $E(H)$ is a rainbow subset of $E(G).$  A vertex $v$ is said to be a {\bf rainbow vertex} if the set of edges incident with $v$ form a rainbow subset. 

The problem dealt with in this paper has its origins in the 
well-known Caccetta-H\"{a}ggkvist conjecture \cite{CacHag}:

\begin{conjectureplus}{Cacetta, H\"{a}ggkvist}
Every simple directed graph on $n$ vertices with minimum out-degree $r$ contains a directed cycle of length
at most $\lceil \frac nr \rceil$.
\end{conjectureplus}

The authors in \cite{AhaDevHol} observed that if for the directed graph $G$ in the above conjecture,  one assigns $n$ colours to the edges whereby the out-directed edges at each vertex receive the same colour, then the conjecture asserts that there
exists a rainbow cycle of length at most  $\lceil \frac nr \rceil$.  As such, they posed the following more general conjecture:

\begin{conjectureplus}{Aharoni, Devos, Holzman}
Let $G$ be a graph and let $c$ be a $\nu(G)$-colouring of $G$ where each colour class has size at least $r$.  Then the coloured graph $(G,c)$ contains a rainbow cycle of length at most $\lceil \frac {\nu(G)}r \rceil.$
\label{conj1}
\end{conjectureplus}

In \cite{DevDreFunMazGuoHuyMohMon}, the authors verified the above conjecture in the case where $r=2.$  In \cite{HomSpi}, it was shown that Conjecture \ref{conj1} holds if each colour class has $\Omega (r)$ edges and in \cite{HomHuy} the authors show that Conjecture \ref{conj1} holds up to an additive constant.  That is, for all $r \ge 2$  there is a constant $\alpha_r$ such Conjecture \ref{conj1} is true if one replaces $\lceil \frac {\nu(G)}2 \rceil$ with $\lceil \frac {\nu(G)}2 \rceil + \alpha_r.$

In \cite{DevDreFunMazGuoHuyMohMon}, the authors considered an extension of rainbow colourings to matroids.   A {\bf colouring} of a matroid $M$ is defined to be a function $c:E(M) \rightarrow S$ for some set $S$.  We call the pair $(M,c)$ a {\bf coloured matroid}.  If $S = [k]$, then we refer to $c$ as a $\mathbf{k}${\bf - colouring} of $M.$  As with graphs, for each $s \in S,$ we define $c^{-1}(s) = \{ e \in E(M) \ \big| \ c(e) =s \}$ to be the colour class for the colour $s.$  For a subset $A\subset E(M),$ we define $c(A) := \{ c(e)\ \big| \ e\in A \}$ and we define $c(M) := c(E(M)),$ the set of distinct colours appearing in $M.$  We say that an element $e$ in a coloured matroid is {\bf colour-singular} if $e$ is the only element in the colour class containing it.
A colouring of matroid (or graph) is said to be $\mathbf{k}${\bf -bounded} (resp. $\mathbf{k}${\bf -uniform}) if each colour class has at most (resp. exactly) $k$ elements.  If a colour class $X$ contains two elements and $e \in X$, we shall denote by $e'$ the other element in $X.$

For a coloured matroid $(M,c)$, a subset $A \subseteq E(M)$ is said to be a {\bf rainbow subset} if $c$ restricted to $A$ is injective.   In particular, a {\bf rainbow circuit} is a circuit whose elements all have different colours.  We say that a rainbow circuit $C$ in $(M,c)$ is {\bf short} if $|C| \le \left\lfloor \frac {r(M)+2}2 \right\rfloor$.
A coloured matroid $(M,c)$ is said to be {\bf circuit - achromatic} if it contains no rainbow circuits.  A subset $A \subseteq E(M)$ is said to be {\bf monochromatic} if all the elements in $A$ have the same colour.

  We shall make frequent use of the following observation whose easy proof is left to the reader.

\begin{observation}
Suppose $(M,c)$ is a simple, circuit-achromatic coloured matroid where $\varepsilon(M) = 2r(M) -1$ and $c$ is $2$-bounded.  Then $(M,c)$ has exactly one colour-singular element.
\label{obs-onesingular}
\end{observation} 

Extending Conjecture \ref{conj1}  for the case $r=2$ to matroids, the authors in \cite{DevDreFunMazGuoHuyMohMon} posed the following conjecture:
\begin{conjecture}
Let $M$ be a simple, regular matroid and let $c$ be a $(r(M)+1)$-colouring of $M$.  If $(M,c)$ has no colour-singular elements, then $M$ has a short rainbow circuit.
\label{conj2}
\end{conjecture}
In \cite{DevDreFunMazGuoHuyMohMon}, the conjecture is shown to be true for graphic and cographic matroids.  In this paper, we prove the above conjecture by proving the existence of four rainbow circuits the sum of whose cardinalities is at most $2r(M) +4.$

A {\bf cycle} in a matroid is a disjoint union of circuits. 
For a coloured matroid $(M,c)$, a pair of circuits (resp. cycles) $\{ C_1, C_2 \}$ is said to be a  {\bf rainbow circuit pair  (RCP)} (resp. {\bf rainbow cycle pair (RCyP)}) if $C_1$ and $C_2$ are disjoint rainbow circuits (resp. cycles).  
Generally, a disjoint pair of rainbow subsets $\{ X_1, X_2\}$ is said to be a {\bf rainbow pair}.
A RCP (resp. RCyP) $\{ C_1, C_2 \}$ is said to be {\bf short} if $|C_1| + |C_2| \le r(M) +2.$  In shorthand, we write SRCP (resp. SRCyP) for a short circuit (resp, cycle) pair.
 
A family of circuits $\{ C_1, C_2, C_3, C_4 \}$ is said to be a {\bf short rainbow circuit  $4$-tuple} ({\bf SRC} $\mathbf{4}${\bf-tuple})  if the circuits $C_i,\ i \in [4]$ are rainbow circuits for which $\sum_i|C_i| \le 2r(M) + 4$ and for which no element of $M$ belongs to more than two of the circuits.  We can define the same for cycles where a SRCy $4$-tuple denotes a short rainbow cycle $4$-tuple.
Note that if $\{ C_1, C_2 \}$ is a SRCP (resp. SRCyP), then $\{ C_1, C_1, C_2, C_2 \}$ is a SRC $4$-tuple (resp. SRCy $4$-tuple). 
%
By averaging, one sees that at least one of the circuits in an SRC $4$-tuple is short.  We prove Conjecture \ref{conj2} by proving the following:

\begin{theorem}
Let $M$ be a simple regular matroid and let $c$ be a $2$-uniform $(r(M)+1)$-colouring of $M$.  Then $M$ has a SRC $4$-tuple.
\label{the-main}
\end{theorem}


\subsection{Notation}

For definitions and notation pertaining to matroids, we shall follow \cite{Oxl}. 
For a set $X$ and elements $x_1, \dots ,x_k$ we shall often write $X + x_1 + \cdots + x_k$ (resp. $X - x_1 - \cdots - x-k$) in place of $X \cup \{ x_1, \dots ,x_k \}$ (resp. $X \backslash \{ x_1, \dots ,x_k \}$).
For a graph $G$, $M(G)$ will denote the {\bf cycle matroid} for $G$, and $M^*(G)$ will denote the {\bf cocycle matroid}.  The sets $E(G)$ and $V(G)$ will denote the sets of edges and vertices of $G$, respectively.
For a set $X$, either $X \subseteq V(G)$ or $X \subseteq E(G),$ $G[X]$ will denote the subgraph of $G$ induced by $X$.

For a graph $G$, let $\kappa(G)$ denote the number of components of $G.$ 
For a vertex $v \in V(G)$, $d_G(v)$ will denote the degree of the vertex $v.$  For a subset $U \subseteq V(G),$ we let $d_G(U) = \sum_{v \in U}d_G(v).$  The set $N_G(v)$ will denote the set of neighbours of $v$ and $E_G(v)$ will denote the set of edges incident with $v.$
For an edge $e \in E(G)$ we shall write $e = uv$ to mean that $u$ and $v$ are the endvertices of $e$.  For an edge $e\in E(G),$ we let $G\backslash e$ (resp. $G/e$) denote the graph obtained from $G$ by deleting (resp. contracting) $e.$
For a subgraph $H$ of $G$, we often write $|H|$ in place of $|E(H)|$ or $\varepsilon(H)$ which we refer to as the {\bf size} of $H.$  We denote the subgraph obtained from $H$ by adding $v$ and all edges between $v$ and vertices in $H$ by $H +v.$  For an edge $e \not\in E(H),$ we let $H + e$ denote the subgraph obtained by adding $e$ to $H$, including the endvertices of $e.$ 

Suppose $(G,c)$ is a coloured graph.  For vertices $u,v \in V(G),$ we let $dist_G^c(u,v)$ be the length of a shortest rainbow path between $u$ and $v$ in $G$; if no such path exists, we define $dist_G^c(u,v) := \infty.$  
For a coloured graph $(G,c)$ (resp. coloured matroid $(M,c)$) and a minor $H$ of $G$ (resp. $N$ of $M$), we let $c\big| H$ (resp. $c\big| N$) denote the colouring of $H$ (resp. $N$) which is the restriction of $c$ to $E(H)$ (resp. $E(N)$).
%
 
 Let $P = v_1e_2v_2 e_2 \cdots e_k v_{k+1}$ be a path where $v_1, \dots ,v_{k+1}$ are vertices and $e_1, \dots ,e_k$ are edges.  If we extend $P$ to a path $P' =  v_1e_2v_2 e_2 \cdots e_k v_{k+1}e v$ or  $P' = v e  v_1e_2v_2 e_2 \cdots e_k v_{k+1}$ by adding an edge $e$ and a vertex $v$, then we shall simply write $P' = P + e + v.$ 
 
  For a matroid $M$ and a set $S$ where $E(M) \cap S = \emptyset$ we say that a matroid $N$ is a {\bf regular} (resp. {\bf graphic, cographic}) {\bf extension} of $M$ by $S$ if $M$ is regular (resp. graphic, cographic), $E(N) = E(M) + S$ and $N\backslash S = M.$  Normally, we shall denote such an extension by $M +S.$
  
  Let $N = M+e$ be a regular extension of $M$ by an element $e$.  A circuit $C$ of $N$ is said to be an $\mathbf{e}${\bf-rainbow} circuit for $(M,c)$ if $e\in C$ and $C-e$ is a rainbow subset of $E(M).$  We shall often make use of the following elementary observation whose proof we omit.

\begin{observation}
Let $(M,c)$ be a simple, regular, circuit-achromatic coloured matroid where $\varepsilon(M) \ge r(M)+1.$  Then for any regular extension $N = M+e$ by an element $e$, the there is an $e$-rainbow circuit $C$ for $(M,c)$ where $|C| \le r(M).$ \label{obs-erainbowcircuit}
\end{observation}   
  
For a subset $X \subseteq E(M)$ of a matroid $M$, we  say that an element $e$ is a {\bf chord} of $X$ if $e\in \mathrm{cl}(X) - X.$  For all chords $e$ of $X$, we refer to the circuits of $X+e$ containing $e$ as $\mathbf{X,e}${\bf -chordal circuits}. Note that for a chord $e$ of a circuit $C$ in a binary matroid $M$, there are exactly two $C,e$ - chordal circuits.
  
%


\section{A theorem for circuit-achromatic coloured matroids}\label{sec-roleachromatic}
 
 Understanding the structure of circuit-achromatic coloured matroids is a necessary aspect of the proof of the main theorem.  The primary goal of this section is to provide a theorem for circuit-achromatic regular matroids which shows that such matroids are convenient in the sense that one can find certain pairs of disjoint rainbow circuits in certain regular extensions of these matroids.
 To begin with, we shall describe why understanding circuit-achromatic matroids is important.
 
In the proof of Theorem \ref{the-main},  we use a fairly standard approach for regular matroids which exploits 
the well-known decomposition theorem for regular matroids of Seymour \cite{Sey}.  That is, for every regular matroid $M$, $M$ is either graphic, cographic, isomorphic to $R_{10}$ or $M$ is a $1$-, $2$-, or $3$-sum of two regular matroids.  This results in a decomposition of a regular matroid $M$ into $1$-, $2$- and $3$-sums of matroids which are either graphic, cographic or isomorphic to $R_{10}.$

Suppose $(M,c)$ is a coloured matroid as described in Theorem \ref{the-main}.  The next lemma shows that if $M = M_1 \oplus_2 M_2$ or $M = M_1 \oplus_3 M_2$, then to find a SRC $4$-tuple for $(M,c)$ we may assume that for some $i\in [2],$ $M_i \backslash E(M_{3-i})$ is circuit-achromatic.  

\begin{lemma}
Let $(M,c)$ be a simple coloured binary matroid.  Suppose there are matroids $M_1$ and $M_2$ such that for $i=1,2,$ $M_i \backslash E(M_{3-i})$ contains a rainbow circuit and either
\begin{itemize}
\item[i)]  $M = M_1 \oplus_2 M_2$ and $\varepsilon(M) > r(M)+3$ or
\item[ii)] $M = M_1 \oplus_3 M_2$ and for $i = 1,2,$ $\varepsilon (M \big| E(M_i)) > r(M_i) +1. $ 
\end{itemize}
Then $M$ has a SRCP.
\label{lem-shortcircuitin23sum}
\end{lemma}

\begin{proof}
For $i = 1,2,$ let $M_i' = M_i \backslash E(M_{3-i})$ and let $C_i$ be a rainbow circuit in $M_i'.$
Suppose that i) holds.   Then $r(M_1) + r(M_2) = r(M) +1.$  If for some $i$, $|C_i| \le r(M_i)$, then $|C_1| + |C_2| \le r(M_1) + r(M_2) +1 = r(M) +2$ and $\{ C_1,C_2 \}$ is a SRCP for $(M,c),$ implying that $\{ C_1, C_1, C_2, C_2 \}$ is SRC $4$-tuple.  Thus we may assume that for $i=1,2$, $|C_i| = r(M_i) +1.$  Then for $i = 1,2,$ $C_i$ is a spanning circuit for $M_i$.  Since $\varepsilon(M) > r(M) +3$, there exists $i\in [2]$ such that $\varepsilon (M_i) > r(M_i) +1$ and we may assume that $\varepsilon (M_1) > r(M_1) +1.$  Then there exists $e\in E(M_1') - C_1$  and a $C_1,e$-chordal circuit $C_1'$ where $C_1'$ is a rainbow circuit and $|C_1'| \le r(M_1).$  Now $\{ C_1', C_2\}$ is seen to be a SRCP (and hence $(M,c)$ has a SRC $4$-tuple).

Suppose that ii) holds.  Then $r(M_1) + r(M_2) = r(M) +2$ and for $i=1,2,$ $\varepsilon(M_i') > r(M_i) +1.$   We claim that for $i=1,2,$ $M_i'$ has a rainbow circuit $C_i'$ where $|C_i'|\le r(M_i).$  If $|C_i| \le r(M_i)$, then we define $C_i' := C_i.$
Suppose instead that $|C_i| > r(M_i).$ Then $|C_i| = r(M_i) +1$ and  
$C_i$ spans $M_i'.$  Given that $\varepsilon(M_i') > r(M_i) +1,$ there exists $e \in E(M_i') - C_i$ and a $C_i,e$-chordal rainbow circuit $D$ for which $|D| \le r(M_i).$  In this case, we define $C_i':= D.$  This proves our claim.
We have that $|C_1'| + |C_2'| \le r(M_1) + r(M_2) = r(M) +2$ and $\{ C_1',C_2' \}$ is a SRCP (and hence $(M,c)$ has a SRC $4$-tuple).
\end{proof}

The main obstacle to proof of Theorem \ref{the-main} is the case where $M$ is a $3$-sum $M = M_1 \oplus_3 M_2$ where $M_1' = M_1 \backslash E(M_2)$ has no rainbow circuits. 
The general strategy is to modify $M_2' = M_2 \backslash E(M_1)$ by adding a subset $S' \subseteq E(M_1) \cap E(M_2)$ to $M_2'$ and colouring these elements so as to obtain a coloured matroid $(M_2'', c_2'')$ where $c_2''$ is a $2$-uniform colouring.  We will argue inductively, assuming that when $M_2''$ is simple,  $(M_2'', c_2'')$ contains a SRC $4$-tuple.  We will then extend this to a SRC $4$-tuple of $(M,c).$  To do this, we need to find certain pairs of disjoint rainbow circuits in $M_1.$  The fact that $M_1'$ has no rainbow circuits enables one to do this.  

Let $(M,c)$ be a simple, regular coloured matroid and let
$N=M+T$ be a regular extension of $M$ by $T.$ 
A pairwise-disjoint collection of circuits $\C$ (resp. cycles) from $N$ is said to be a $\mathbf{T}${\bf-short rainbow circuit collection} (resp. $\mathbf{T}${\bf-short rainbow cycle collection}) for $(M,c)$ if for all $\{ C, C' \} \subset {\binom {\C}2}$, $|C| + |C'| \le r(M)+2$ and for all $C\in \C,$ $|C\cap T| =1$ and $C-T$ is a rainbow subset of $E(M).$
In the case where $\C$ is a pair/triple of circuits (resp. cycles), we simply write $\C$ is a $\mathbf{T}${\bf - SRCP} / $\mathbf{T}${\bf - SRCT} (resp. $\mathbf{T}${\bf - SRCyP} / $\mathbf{T}${\bf - SRCyT}) for $(M,c).$  Note that the existence of a $T$-SRCyP (resp. $T$-SRCyT) implies the existence of a $T$-SRCyP (resp. $T$-SRCyT) for $(M,c).$ 

A disjoint pair of circuits $\{ C_1, C_2 \}$ in $N$ is said to be a {\bf near-}$\mathbf{T}${\bf-short rainbow circuit pair} ({\bf near-}$\mathbf{T}${\bf-SRCP}) for $(M,c)$ if for $i= 1,2$, $|C_i\cap T| =1,$ $C_i - T$ is a rainbow subset and $|C_1| + |C_2| \le r(M) +3.$
We have the corresponding definition for cycles in place of circuits, where {\bf near-}$\mathbf{T}${\bf-SRCyP} denotes a near-$T$-short rainbow cycle pair.   The existence of a near-$T$-SRCyP implies the existence of a near $T$-SRCP.
 
Suppose $T = \{ x \}$ and $\C = \{ C_1, C_2 \}$ is a pair of disjoint circuits (resp. cycles)  of $N$ such that $C_1 \cap C_2 = \{ x \},$ $|C_1| + |C_2| \le r(M)+3$ and for $i=1,2,$ $C_i - x$ is a rainbow subset of $E(M).$  In this case, we say that $\C$ is $\mathbf{x}${\bf-semi-short rainbow circuit pair} ($\mathbf{x}${\bf-semi-SRCP}) for $(M,c).$  We have the corresponding definition for cycles where $\mathbf{x}${\bf-semi-SRCyP} denotes an $x$-near-semi-short rainbow cycle pair. 

The next theorem will play a central role in the proof of the main theorem and a large part of this paper is devoted to its proof.    

\begin{theorem}
Let $M$ be a rank-$n,$ simple regular matroid where $2n-2 \le \varepsilon(M) \le 2n-1$ and let $c$ be a $2$-bounded $n$-colouring of $M$ where $(M,c)$ is circuit-achromatic
Let $N = M + T$ be a regular extension of $M$ where $T$ is a co-independent $3$-circuit in $N.$  
If $\varepsilon(M) = 2n-1,$ then let $e$ be the unique colour-singular element in $(M,c).$
\begin{itemize}
\item[i)]  Suppose $\varepsilon(M) = 2n-1.$  Then either 
\begin{itemize}
\item[i.1)] there exists a $T$-SRCT for $(M,c)$ 
or
\item[i.2)] for all $x\in T,$ there exists a $T$-SRCP $\{ C_1, C_2 \}$ for $(M-e,c)$ where $x\not\in C_1 \cup C_2.$
\end{itemize}
\item[ii)] Suppose $\varepsilon(M) = 2n-1$ and $M+x$ is a regular extension of $M$ by an element $x.$  If $x$ and $e$ are non-parallel, then there exists an $x$-semi-SRCP for $(M,c).$
\item[iii)] Suppose $\varepsilon(M) = 2n-2.$ Then there exists a $T$-SRCP for $(M,c).$
\end{itemize}  
\label{the-NoRainbowCircuit}
\end{theorem}

An important tool used in the proof of the above theorem is a structure theorem for circuit-achromatic binary matroids which will be given in the next section.


\section{Colourings of Binary Matroids}\label{sec-colourings}

As outlined in the previous section, the proof of Theorem \ref{the-main} depends largely on Theorem \ref{the-NoRainbowCircuit}.   In this section, we shall describe a structure theorem for circuit-achromatic binary matroids which will be instrumental in the proof of this theorem. 

Let $M$ be a rank-$n$ binary matroid.  Let $c: E(M) \rightarrow [k]$ be a $k$-colouring of $M$ having colour classes $X_1, \dots ,X_k$.  The colouring $c$ is said to be {\bf stratified} if there exists a permutation $i_1, \dots ,i_k$ of $[k]$ such that
for $j= 1, \dots ,k,$ $S_j = X_{i_1} \cup \cdots \cup X_{i_j}$ is closed (i.e. $S_j = \mathrm{cl}_M(S_j)$) and $1 \le r(S_1) < r(S_2) < \cdots < r(S_k) =n.$  We call $X_{i_1}, X_{i_2}, \dots ,X_{i_k}$ a {\bf stratification} of $c.$   We claim that if $c$ is a stratified colouring, then $(M,c)$ is circuit-achromatic.  For suppose $X_{i_1}, X_{i_2}, \dots ,X_{i_k}$ is a stratification of $c.$  Let $C$ be a circuit of $M$ and let $j^*$ be the greatest integer $j$ for which $C\cap X_{i_j} \ne \emptyset.$  Then clearly $|C \cap X_{i_{j^*}}| \ge 2$ and hence $C$ is not a rainbow circuit.
In \cite{BerSch} it was shown that the converse holds if $c$ is an $r(M)$ - colouring of a binary matroid $M$ for which $(M,c)$ is circuit-achromatic.

\begin{theoremplus}{\cite{BerSch}}
Let $M$ be a binary matroid and let $c$ be an $r(M)$-colouring of $M$.   Then $(M,c)$ is circuit-achromatic if and only if $c$ is stratified.
\label{the1}
\end{theoremplus}

The above theorem has the following useful corollaries.

\begin{corollary}
Let be a rank-$n$ binary matroid $M$ having an $n$-colouring $c$ with colour classes $X_1. \dots ,X_n$.  If $(M,c)$ is circuit - achromatic, then 
\begin{itemize}
\item[i)] for some $i$, the elements of $X_i$ form a parallel class of $M$, and
\item[ii)] for some $i$, $X_i$ is a cocircuit of $M.$
\end{itemize}
\label{cor-cor1}
\end{corollary}

\begin{proof}
By Theorem \ref{the-main}, $c$ is stratified and hence there exists a permutation $i_1, \dots ,i_n$ of $[n]$ such that
for $j= 1, \dots ,n,$ $S_j = X_{i_1} \cup \cdots \cup X_{i_j}$ is closed  and $1 \le r(S_1) < r(S_2) < \cdots < r(S_n) =n.$ It follows that for all $j\in [n],$ $r(S_j) = j$. In particular,
$r(X_{i_1}) = r(S_1) =1$ hence $X_{i_1}$ is a parallel class.  Furthermore, since $r(S_{n-1}) = n-1$ and $S_{n-1}$ is closed, it follows that
$X_{i_n}$ is a cocircuit. 
\end{proof}

The next corollary is an immediate consequence of the above.

\begin{corollary}
If $M$ is a simple matroid with an $r(M)$-colouring $c$ for which $(M,c)$ has no colour-singular elements, then $(M,c)$ contains a rainbow circuit.
\label{cor-cor2}
\end{corollary}

 
 \section{Proof for graphic and cographic matroids}\label{sec-proofGraphicCographic}
Our first task is to prove Theorem \ref{the-main} for graphic and cograph matroids.  In this section, we shall prove the following theorem:

\begin{theorem}
Let $(M,c)$ be simple coloured matroid where $c$ is a $2$-uniform $(r(M)+1)$-colouring
and $M$ is either graphic or cographic.  Then $M$ has a SRC $4$-tuple.
\label{the-GraphicCographicMain}
\end{theorem}

The proof of this theorem builds on some of the ideas in the proof of \cite[Theorem 5]{DevDreFunMazGuoHuyMohMon}.
%
 
 Let $M$ be a matroid.   We define a {\bf theta-subset} in $M$ to be a subset $\Theta \subseteq E(M)$ for which $\Theta$ contains exactly three circuits $C_1, C_2, C_3$ where for all permutations $i,j,k$ of $[3]$ $C_i \triangle C_j = C_k.$  A disjoint pair $( \Theta, C )$ where $\Theta$ is a rainbow theta-subset and $C$ is a rainbow circuit is said to be a {\bf rainbow theta-circuit pair (RThCP)}. 
  
For a RThCP $(\Theta, C)$) we define $\psi(\Theta,C):= \left \lceil \frac {|C|}2 \right\rceil + |\Theta |$.  
A {\bf short rainbow theta-circuit pair (SRThCP)} is a RThCP $(\Theta,C)$ where $\psi(\Theta,C) \le r(M) +3.$

\subsection{A lemma for graphic matroids}\label{sec-proofTheGraphicCographic}

We shall first prove the following lemma:

\begin{lemma}
Let $M$ be a simple graphic matroid and let $c$ be a $2$-uniform $(r(M)+1)$-colouring for $M.$
Then either 
\begin{itemize}
\item[i)] $(M,c)$ has a SRCP or 
\item[ii)] $(M,c)$ has a RThCP $(\Theta,C)$ where $\psi(\Theta,C) \le r(M)+2$, or
\item[iii)] $(M,c)$ has a SRThCP $(\Theta, C)$ where $\Theta$ has at most two chords and for all chords $e$ of $\Theta$ and all $\Theta,e$-chordal circuits $D_e$, $\Theta - D_e$ is independent. Furthermore, either
\begin{itemize}
\item[iii.1)] $C$ contains one chord of $\Theta,$ $|C|$ is odd, and $2|\Theta| + |C| \le 2r(M) +5,$ or 
\item[iii.2)] $C$ contains two chords of $\Theta,$  $|C|$ is even and $2|\Theta| + |C| \le 2r(M) + 6.$
\end{itemize}
\end{itemize}
\label{lem-shortpairsgraphic}
\end{lemma} 

\begin{proof}
Let $M = M(G).$  The proof is by induction on $\nu (G)$.  It suffices to prove the lemma when $G$ is connected. Since $M$ is simple, ${\binom {\nu(G)}2} \ge \varepsilon(G) = \varepsilon(M) = 2(r(M) +1) = 2\nu(G)$ and thus $\nu(G) \ge 5.$
 When $\nu(G) =5$, we have $G \simeq K_5$ and it can be shown that $(M,c)$ has a SRCP.  Suppose that $n>5$ and the lemma holds for graphs with fewer than $n$ vertices.
Let $(G,c)$ be the coloured graph associated with $(M,c)$ where $\nu(G) = n.$
 Suppose first that $G$ contains a rainbow vertex $v.$  Let $G' = G-v$.  Then $|c(G')| = n= \nu(G') +1.$  
 Thus there is a component $K$ of $G'$ for which $|c(K)| \ge \nu(K) +1.$  Let $T$ be a rainbow spanning tree of $K$ (which has $\nu(K) -1$ edges) and let $\{ f_1,f_2 \} \subset E(K) - E(T)$ where $c \left( \{ f_1, f_2 \} \right) \cap c(T) = \emptyset$; such edges exist since $|c(K) - c(T)| \ge 2.$    For $i = 1,2,$ there is a unique (rainbow) cycle $C_i$ in $T+f_i.$
If $|V(C_1) \cap V(C_2)| \le 1,$ then $C_1,C_2$ are edge-disjoint rainbow cycles for which $|C_1| + |C_2| \le n+1$ and hence $\{ C_1,C_2 \}$ is a SRCP and i) holds.  If $|V(C_1) \cap V(C_2)| \ge 2,$ then the subgraph of $T + \{ f_1 , f_2 \}$ induced by $E(C_1) \cup E(C_2)$ contains a rainbow theta subgraph.  In this case, $G$ contains a rainbow subgraph $\overline{\Theta}$ and we may assume that among such subgraphs, $\varepsilon(\overline{\Theta})$ is minimum.  Let $\Theta = E(\overline{\Theta})$ and let $D_i,\ i = 1,2,3$ be the circuits of $\Theta.$ Given that $|c(G - \Theta)| = n,$ it follows that $G- \Theta$ contains a rainbow cycle, say $\overline{C}$.  Let $C = E(\overline{C}).$  
Suppose $e$ is a chord of $\Theta$ and $D_e$ is a $\Theta,e$-chordal circuit. Suppose that for some $i$, $D_i \subseteq \Theta - D_e.$   If $D_e$ is not a rainbow circuit, then $\Theta' = \Theta \triangle D_e$ is seen to be a rainbow theta-subset of $M$ where $\varepsilon(\Theta') < \varepsilon(\Theta),$ contradicting the minimality of $\Theta.$  Thus $D_e$ is a rainbow circuit and
$\{ D_e,D_i \}$ is seen to be a SRCP and ii) holds.  Because of this,
 we may assume that for all chords $e$ of $\Theta$ and for all
$\Theta,e$-chordal circuits $D_e$, $\Theta - D_e$ is independent.
By this assumption and the minimality of $\Theta$, it can be shown that at most two elements of $M$ are chords of $\Theta.$  In particular, at most two elements of $C$ are chords of $\Theta.$ 
If no elements of $C$ are chords of $\Theta,$  then $|V(\overline{C}) \cap (V(G) - V(\overline{\Theta})| \ge \left\lceil \frac {|C|}2 \right\rceil$ and hence $\nu(\overline{\Theta}) + \left\lceil \frac {|C|}2 \right\rceil \le n$ implying that
$|\Theta| + \left\lceil \frac {|C|}2 \right\rceil \le n+1= r(M) +2$ and $\psi(\Theta,C) \le r(M) +2.$  In this case, ii) holds. Suppose one element of $C$ is a chord of $\Theta.$  Then $|V(\overline{C}) \cap (V(G) - V(\Theta))| \ge \left\lceil \frac {|C|-1}2 \right\rceil$ and hence
$|\Theta| +  \left\lceil \frac {|C|-1}2 \right\rceil \le n+1 = r(M) + 2.$  If $|C|$ is even then $\left\lceil \frac {|C|-1}2 \right\rceil =  \left\lceil \frac {|C|}2 \right\rceil$ in which case $|\Theta| +  \left\lceil \frac {|C|}2 \right\rceil \le r(M) +2$ and $\psi(\Theta,C) \le r(M) +2$ and ii) holds again.
If $|C|$ is odd, then $\left\lceil \frac {|C|-1}2 \right\rceil =  \frac {|C| -1}2$ and $2|\Theta| + |C| \le 2(n + 1) + 1 = 2r(M) + 5.$ In this case, iii.1) holds.  If two elements of $C$ are chords of $\Theta,$ then $|V(\overline{C}) \cap (V(G) - V(\Theta)| \ge \left\lceil \frac {|C| -2}2 \right\rceil$ and hence
$|\Theta| + \left\lceil \frac {|C|}2 \right\rceil \le n+2 = r(M) + 3.$  As such, we have that $2|\Theta| + |C| \le 2(n+2) = 2r(M) + 6$ and iii.2) holds. 

From the above, we may assume that $G$ has no rainbow vertices.  In particular, this implies that for all $v$, $d_G(v) \ge 2.$  Suppose that for some vertex $v,$ $d_G(v) = 2.$  Let $E_G(v) = \{ e, f \}$.  Since $v$ is not a rainbow vertex, $c(e) = c(f).$  Let $G' = G/e\backslash f$ and $c' = c\big| E(G').$
Since $G$ is simple, $G'$ is seen to be simple and $|c(G')| = 2\varepsilon (G').$  By assumption, the lemma holds for $M' = M(G')$ and using this, one can show that the lemma holds for $M$ as well. 
Thus we may assume that for all $v\in V(G),$ $d_G(v) \ge 3.$  Suppose there is a vertex $v$ for which $d_G(v) = 3.$  Let $E_G(v) = \{ e,f,g \}.$  Since $v$ is not a rainbow vertex,  we may assume that $c(e) = c(f).$  Let $G' = G/e\backslash f$ and $c' = c\big| E(G').$  Then $G'$ is simple and $|c'(G')| = 2\varepsilon(G')$.  By assumption, the theorem holds for $M' = M(G').$  Suppose that i) holds for $M'.$  Then $G'$ contains a SRCP say $\{ C_1', C_2'\}.$   If neither of these cycles contains $g$, then $\{C_1',C_2'\}$ is a SRCP for $(G,c).$   Suppose that one the cycles contains $g,$ say $C_1'.$  Then the cycle $C_1$ in $G$ where $E(C_1) = E(C_1') + e$ is a rainbow cycle.  Letting $C_2 = C_2',$ we see that $\{ C_1, C_2 \}$ is a SRCP for $(G,c).$   Suppose that ii) holds $M'.$ Then $(M',c')$ has RThCP $(\Theta',C')$ where $\psi(\Theta,C') \le r(M') + 2 = r(M) +1.$ 
We may assume that $g\in C'\cup \Theta'$, for otherwise, $(\Theta', C')$ is a RThCP for $M$, and ii) holds for $M.$  Suppose $g\in C'.$  Then there is a rainbow cycle $C$ in $G$ where $E(C) = E(C') +e.$  Then $(\Theta', C)$ is seen to be a RThCP in $M$
where  $\psi(\Theta',C) \le \psi(\Theta',C') + 1 \le r(M) +2$ and ii) holds for $M.$ On the other hand, suppose $g\in \Theta'.$  Then there is a rainbow theta-subgraph $\Theta$ of $G$ where $E(\Theta ) = E(\Theta') + e.$  Then $(\Theta,C')$ is seen to be a RThCP in $M$
where $\psi(\Theta,C') = \psi(\Theta',C') + 1 \le r(M) +2$ and ii) holds for $M.$  Lastly, it is straightforward to show that if iii) holds for $M'$, then either i), ii), or iii) holds for $M.$

From the above, we may assume that for all vertices $v$, $d_G(v) \ge 4.$  Given that $2\varepsilon (G) = 4n,$ its average degree is $4$ and thus $G$ must be $4$-regular.  Furthermore, since $G$ is simple and has no rainbow vertices, each vertex $v$ has exactly two incident edges of the same colour.  We shall direct the edges of $G$ as follows:  for all vertices $v$, if $e,f \in E_G(v)$ and $c(e) = c(f),$ then both $e$ and $f$ will be directed away from $v.$  The resulting directed graph $\overrightarrow{G}$ is $2$-regular; that is, the in-degree and out-degree of each vertex equals two.  Using some of the ideas from \cite{AloMcdMol}, we shall show that i) holds.  We first observe that every directed cycle of $\overrightarrow{G}$ corresponds to a rainbow cycle in $G.$  Suppose $\overrightarrow{G}$ contains four edge-disjoint directed cycles, say $C_i, i\in [4]$.  Given that $\sum_i \varepsilon(C_i) \le \varepsilon = 2n,$ either $\varepsilon(C_1) + \varepsilon(C_2) \le n,$ or  $\varepsilon(C_3) + \varepsilon(C_4) \le n,$ implying that either $\{ C_1, C_2 \}$ or $\{ C_3,C_4 \}$ is a SRCP.  Thus we may assume that $\overrightarrow{G}$ has at most $3$ edge-disjoint directed cycles.

For each vertex, one can find two edge-disjoint directed cycles which contain the vertex.  Among all vertices and such pairs of directed cycles, let $v$ be a vertex and $C_1, C_2$ be edge-disjoint directed cycles such that $|C_1| + |C_2|$ is minimum.   Let $\overrightarrow{H}$ be the directed subgraph induced by $C_1 \cup C_2.$  We first observe that for all vertices $u\in V(\overrightarrow{H})$ where $d_{H}(u) =4,$ $\overrightarrow{H}-u$ contains no directed cycle;  if $\overrightarrow{H}-u$ did contain a directed cycle $\overrightarrow{C}$ then $\overrightarrow{H}' = \overrightarrow{H} - E(\overrightarrow{C})$ is Eulerian (in-degree = out-degree for all vertices) and given that $u$ has degree $4$ in $\overrightarrow{H}'$, it follows that $\overrightarrow{H}'$ has two edge-disjoint directed cycles $C_1', C_2'$ containing $u$, contradicting the minimality of $C_1, C_2.$   
Since $\overrightarrow{G}$ has at most three edge-disjoint directed cycles, it follow that $V(\overrightarrow{H}) = V(\overrightarrow{G}).$  If $V(C_1) \cap V(C_2) = \{ v \},$ then $|C_1| + |C_2| \le n+1 = r(M) +2$ and $\{ C_1,C_2 \}$ is a SRCP. Thus we may assume that there exists $u\in V(C_1) \cap V(C_2) - v.$   We claim that we may assume that there is no edge $e \in E(G) - E(H)$ which is a chord of $C_1$ or $C_2$.  For suppose $e$ is a chord of $C_1.$  Let $C_{1j},\ j = 1,2$ be the $C_1,e$-chordal cycles. 
Suppose first that $u \in V(C_{11})$ and $v\in V(C_{12}).$  Seeing as one of $C_{11}$ or $C_{12}$ is a directed cycle in $\overrightarrow{G},$ either $u$ together with $C_{11}, C_2$ or $v$ together with $C_{12}, C_2$ contradicts the minimality of $v$ and $C_1, C_2.$  On the other hand, suppose that $\{ u,v \} \subset V(C_{11}).$
If $C_{11}$ is a directed cycle in $\overrightarrow{G},$ then $v$ together with $C_{11}, C_{2}$ contradicts the minimality of $v$ and $C_1, C_2.$  Thus $C_{12}$ is a directed circuit.  If $V(C_{12}) \cap V(C_2) = \emptyset,$ then $|C_{12}| + |C_2| \le n$ and $\{ C_{12}, C_2 \}$ is is a SRCP.  If there exists
$v' \in V(C_{12}) \cap V(C_2),$ then $v'$ together with $C_{12}, C_2$ contradicts the minimality of $v, C_2, C_2.$  Thus we may assume that no edge of $E(G) - E(H)$ is a chord of $C_1$ or $C_2.$   

Given that $\overrightarrow{H} -v$ has no directed cycles, we can describe $\overrightarrow{H}$ as follows:  For $i=1,2$, $C_i$ decomposes into directed paths $P_{i1}, P_{i2}. \dots ,P_{ik}$ where for $j=1, \dots ,k$ $P_{ij}$ is directed from $u_i$ to $u_{i+1},$ where $u_1 = u_{k+1} = v$ and $V(C_1) \cap V(C_2) = \{ u_1, \dots ,u_k \}.$ Since $G$ is simple, for all $i$, $V(P_{i1} \cup P_{i2}) - \{ u_i, u_{i+1} \} \ne \emptyset.$  Since $\overrightarrow{G}$ has at most $3$ edge-disjoint directed cycles, all the vertices of $V(G) - \{ u_1, \dots ,u_k \}$ belong to one directed cycle $\overrightarrow{C}.$  Thus for some $i_1 < i_2,$ and 
$j_1, j_2 \in \{ 1, 2 \}$ there is an edge $e\in E(\overrightarrow{C})$ directed from $v_{i_1}$ in $P_{i_1j_1} - u_{i_1} - u_{i_1+1}$ to $v_{i_2}$ in $P_{i_2j_2} - u_{i_2} - u_{i_2+1}.$  By assumption, we have that $j_1 \ne j_2.$  We may change $\overrightarrow{C}_1$ and $\overrightarrow{C}_2$ into directed cycles $\overrightarrow{C}_1'$ and $\overrightarrow{C}_2'$ by swapping the paths $P_{i_11}$ and $P_{i_12}.$  It now follows that $e$ is a chord of $\overrightarrow{C}_1'$ or $\overrightarrow{C}_2'$.  Using $C_i',\ i =1,2$ in place of $C_1, C_2$, we can argue as before.  This completes the proof.
\end{proof}

\subsubsection{The proof of Theorem \ref{the-GraphicCographicMain} for graphic matroids}\label{sec-mainthegraphic}

Let $(M,c)$ be a simple coloured matroid where $c$ is a $2$-uniform $(r(M)+1)$-colouring
 and $M$ is graphic.  Let $M = M(G)$ where $\nu(G) = n.$  Suppose $(M,c)$ has a SRCP $\{ C_1,C_2\}$.  Then $|C_1| + |C_2| \le r(M) +2$ and $\{ C_1,C_1,C_2,C_2\}$ is seen to be a SRC $4$-tuple.  Thus we may assume that $(M,c)$ has no SRCP.   Suppose $(M,c)$ has a RThCP $(\Theta, C)$ where $\psi(\Theta,C) \le r(M) +2.$  Then
$2|\Theta| + |C| \le 2r(M) + 4.$  Let $C_i, \ i = 1,2,3$ be the circuits of $\Theta.$  Then $\sum_i |C_i| = 2|\Theta|$ and hence $\{ C_1,C_2,C_3,C \}$ is seen to be a SRC $4$-tuple.  Thus we may assume that no such RThCP exists.   
Then Lemma \ref{lem-shortpairsgraphic} iii) holds, and there exists a SRThCP $(\Theta,C)$ where $\Theta$ has at most two chords and for all chords $e$ of $\Theta$ and all $\Theta,e$-chordal circuits $D_e,$ $\Theta - D_e$ is independent. 

 Let $C_i,\ i = 1,2,3$ be the circuits of $\Theta.$  Suppose that  Lemma \ref{lem-shortpairsgraphic} iii.1) holds.
 Let $e_1$ the edge of $C$ which is a chord of $\Theta.$   Then for some $i$, $e$ is a chord of $C_i$ and we may assume this is true for $i=1.$  Let $C_{1j},\ j = 1,2$ be the $C_1,e$-chordal cycles.
  At least one of these is a rainbow circuit, say $C_{11}.$  Now $C_{11}, C_2,C_3,C$ are rainbow circuits for which no element of $M$ belongs to more than two of the circuits.  Since $|C_{11}| + |C_{12}| = |C_1| + 2,$ it follows that
\begin{align*}|C_{11}| + |C_2| + |C_3| + |C| &= |C_1| +2 - |C_{12}| + |C_2| + |C_3| + |C|\\ &\le 2|\Theta| + |C| + (2-|C_{12}|)\\ &\le 2r(M) + 5 -1 = 2r(M) + 4. \end{align*} Thus $\{ C_{11}, C_2, C_3, C\}$ is a SRC $4$-tuple.  
Suppose instead that Lemma \ref{lem-shortpairsgraphic} iii.2) holds.  Let $e_i,\ i = 1,2$ be the elements of $C$ which are chords of $\Theta.$  Suppose first that there exists a $2$-subset $\{ i_1, i_2 \} \subset [3]$ such that for $j= 1,2,$ $e_j$ is a chord of $C_{i_j}.$
For convenience, we may assume $i_j = j,\ j = 1,2.$ For $j=1,2$, there is an $C_j,e_j$-chordal circuit $D_j$ which is a rainbow circuit.  We see that
\begin{align*}
|D_1| + |D_2| + |C_3| + |C| &\le |C_1|-1 + |C_2|-1 + |C_3| + |C|\\ &= 2|\Theta| + |C| - 2\le 2r(M) +4.
\end{align*}
Thus $\{ D_1,D_2, C_3, C\}$ is seen to be a SRC $4$-tuple.  Suppose instead that for some $i\in [3],$ $e_1$ and $e_2$ are chords of $C_i.$  We may assume that $e_i,\ i = 1,2$ are chords of $C_1.$ 
For $i = 1,2$, there is an $C_1,e_i$-chordal circuit $D_i$ which is a rainbow circuit.  Since for $i= 1,2$, $\Theta - D_i$ contains no circuits, it can be shown that at least one of $\{ D_1, C_2,C_3,C\},$ $\{D_2, C_2,C_3,C\}$ or $\{ D_1 \triangle D_2, C_2, C_3,C\}$ is a SRC $4$-tuple.
This completes the proof.
%
%

\subsection{The proof of Theorem \ref{the-GraphicCographicMain} when $M$ is cographic}\label{sec-theproofcographic}

Let $(M,c)$ be a simple coloured matroid where $c$ is a $2$-uniform $(r(M)+1)$-colouring
 and $M$ is cographic.
Let $M = M^*(G)$ where $G$ is a connected graph with $n$ vertices and having no 2-cocycles.  
Then $\varepsilon(M) = 2(r(M)+1) = 2(\varepsilon(G) -n +2).$  Since $\varepsilon(G) = \varepsilon(M),$ it follows that $\varepsilon(M) = \varepsilon(G) = 2(n-2)$ and $r(M) = n-3.$
Let $(G,c)$ be the coloured graph inherited from $(M,c).$   Then $c$ is a $2$-uniform $(n-2)$-colouring. 
It will suffice show that $(M,c)$ has a SRCP; that is, $(G,c)$ has two edge-disjoint rainbow cocycles $C_1$ and $C_2$ such that $|C_1| + |C_2| \le n-1.$  We shall proceed by induction on $n.$   The smallest graph $G$ in question has $6$ vertices and the theorem is seen to be true in this case.
Assume that for graphs with fewer than $n$ vertices there is a SRCP.
Suppose $G$ contains a monochromatic digon $\{ e,f \}$.  Let $G' = G/e\backslash f$ and $c' = c\big| E(G').$  Then $G'$ is connected and has no $2$-cocycles (since a cocycle of $G'$ is seen to be a cocycle of $G$).  By assumption, $G'$ has disjoint rainbow cocycles $C_1, C_2$ where $|C_1| + |C_2| \le n-2$ and these are also seen to be cocycles in $G.$  Therefore we may assume that $G'$ has no monochromatic digons.  Since $|c(G)| = n-2,$ there are two rainbow vertices $u_1,u_2$ in $G$.  Suppose that $u_1,u_2$ are non-adjacent.
Let $C_j = E_G(u_j),\ j = 1,2.$ If $|C_1| + |C_2| \le n-1,$ then $\{ C_1, C_2 \}$ is a SRCP.  Suppose $|C_1| + |C_2| \ge n.$  Then $\sum_{v \in V(G) - \{ u_1, u_2 \} }d_G(v) \le 4(n-2) - n = 3(n-2) - 2.$  Thus there exists a vertex $u_3 \in V(G) - \{ u_1, u_2 \}$ where $d_G(u_3) \le 2.$  Given that $G$ has no vertices of degree $2$, it follows that $d_G(u_3) = 1.$  Let $e = u_3v$ and let $\{ e, e' \}$ be the colour class containing $e.$  Let $G' = G/e\backslash e'$ and let $c' = c\big| E(G')$. By assumption, $(G',c')$ contains two edge-disjoint rainbow cocycles $D_1$ and $D_2$ where $|D_1| + |D_2| \le n-2.$  Now either $D_1$ or $D_1+e'$ is a rainbow cocycle of $G$ and we let $D_1'$ be one which is.  Furthermore, $D_2' = E_G(u_3) = \{ e \}$ is a rainbow cocycle of $G$ and 
$|D_1'| + |D_2'| \le |D_1| + 2 \le n-1.$   Thus $\{D_1', D_2' \}$ is a SRCP for $(M,c).$

From the above, we may assume that any two rainbow vertices are adjacent. Let $u_i, \ i\in [k]$ be all the rainbow vertices in $G$.  By the above, we may assume that each pair of vertices is adjacent.  However, this implies that for each edge $e_{ij} = u_i u_j$, no two edges with colour $c(e_{ij})$ are incident.  Thus there are at least $n - ((n-2) - \binom k2)  > k$ rainbow vertices.  This yields a contradiction.

\section{Circuit-achromatic graphic matroids $M$ with\\ $2r(M)-1$ elements}\label{sec-BigTheoremi)andii)Graphic}

Towards proving Theorem \ref{the-NoRainbowCircuit}, we shall first examine the case where $M = M(G)$ is a simple graphic matroid and $\varepsilon(M) = 2r(M) -1.$ 
For the remainder of this section, we let $M = M(G)$ be a simple graphic matroid where $G$ is connected, $n = \nu(G) \ge 3,$ and $\varepsilon(G) = 2(n-1)-1 = 2n-3.$  Let $c: E(G) \rightarrow [n-1]$ be an $(n-1)$-colouring of $M$ where $|c^{-1}(1)| =1$ and $|c^{-1}(i)| = 2, \ i =2, \dots ,n-1.$ 
For $i = 1, \dots ,n-1$ let $X_i = c^{-1}(i)$ where $X_1 = \{ e_1 \}$ and for all $i\ge 2,$ $X_i = \{ e_{i1}, e_{i2} \}.$
Assume that $(M,c)$ is circuit - achromatic. It follows by Theorem \ref{the1} that the colouring $c$ is stratified and we may assume that the graph $G$ is constructed as follows:
Let  $G_1$ be the graph consisting of a single edge $e_1 = v_{11}v_{12}.$  Next let $G_2$ be the graph obtained from $G_1$ by adding the vertex $v_2$ and edges $e_{21} =v_2 v_{21}$ and $e_{22} = v_2v_{22}.$  Continuing, assume that we have constructed $G_{i-1}.$  We construct $G_{i}$ by adding to $G_{i-1}$ the vertex $v_i$ and edges $e_{i1} = v_iv_{i1}$ and $e_{i2} = v_iv_{i2}$ where  $\{ v_{i1}, v_{i2} \} \subseteq V(G_{i-1}).$  The construction finishes with $G = G_{n-1}.$  Note that by our construction, for all $2\le i \le n-1,$ the minimum vertex degree in $G_i$ is $2$ and $\nu(G_{i}) = i+1.$

\begin{lemma}
For all pairs of distinct vertices $\{ u,v \} \subseteq G$, $dist_G^c(u,v) \le \lfloor \frac n2 \rfloor.$  Furthermore, when $n$ is even, there is at most one pair of vertices $u,v$ for which $dist_G^c(u,v) = \lfloor \frac n2 \rfloor = \frac {n}2.$ \label{lem1}
\end{lemma} 

\begin{proof}
We shall prove the assertion is true for $G_1, G_2, \dots ,G_{n-1} = G$ by induction.  That is, we shall show that for all $i$, for all pairs of distinct vertices $\{ u,v \} \subseteq G_i$, $dist_{G_i}^c(u,v) \le \lfloor \frac {\nu(G_i)}2 \rfloor = \lfloor \frac {i+1}2 \rfloor ,$ and when $i$ is odd, there is only one pair of vertices for which equality holds.
The assertion is clearly true for $G_1$.  Suppose that for $i = 1, \dots, k-1$, the lemma is true for $G_i$. We shall prove that the assertion is true for $G_k.$
We have $e_{k1} = v_{k}v_{k1}$ and $e_{k2} = v_{k}v_{k2}$. Let $u,v$ be a distinct pair of vertices in $V(G_k).$  If $\{ u,v \} \subseteq V(G_{k-1}),$ then by assumption, $dist_{G_{k-1}}^c(u,v) \le \lfloor \frac {k}2 \rfloor \le \lfloor \frac {k+1}2 \rfloor.$ 
Thus we may assume that $v = v_k.$

Suppose $k$ is odd.
Let $k=2\ell -1.$   By assumption, for all $j\in \{ 1,2 \}$ and for all $u \in V(G_{k-1}) - v_{kj},$ $dist_{G_{k-1}}^c(v_{kj}, u) \le \ell -1.$
Suppose there exists a distinct pair of vertices $\{ u, u' \} \subseteq V(G_{k-1}) - \{ v_{k1}, v_{k2} \}$ such that for $j=1,2$, $dist_{G_{k-1}}^c(u,v_{kj}) = d_{G_{k-1}}^c(u',v_{kj}) = \ell -1.$  We observe that at least three of the vertices in $\{ u, u', v_{k1}, v_{k2} \}$, say $x,y,z,$  belong to $V(G_{k-2}).$  Since $k-2$ is odd, it follows by assumption that for at most one pair of vertices in $\{ x,y,z \}$, the rainbow distance between these two vertices is at most $\ell -1.$  However, this is contradicted by the assumption that for $j=1,2$, $dist_{G_{k-1}}^c(u,v_{kj}) = dist_{G_{k-1}}^c(u',v_{kj}) = \ell -1.$  Thus there is at most one vertex $u \in V(G_{k-1}) - \{ v_{k1}, v_{k2} \}$ for which $dist_{G_{k-1}}^c(u, v_{k1}) = dist_{G_{k-1}}^c(u, v_{k2}) = \ell-1.$  It now follows that for all $u \in V(G_{k}) - v$, $dist_{G_k}^c(u,v) \le \ell,$ and there is at most one vertex $u$ for which equality holds. 

Suppose $k$ is even.
Let $k = 2\ell.$  If $u \in \{ v_{k1}, v_{k2} \}$, then $dist_{G_k}^c(u,v) =1 \le \ell.$  Thus we may assume that $u \not\in \{ v_{k1}, v_{k2} \}.$
Given that $\{ u, v_{k1}, v_{k2} \} \subseteq V(G_{k-1})$ and $\nu(G_{k-1}) =k = 2\ell$, which is even, it follows that for $j=1,2,$ $dist_{G_{k-1}}^c(u,v_{kj}) \le \ell$ and the inequality is strict for at least one $j\in \{ 1, 2 \}.$ We may assume that $dist_{G_{k-1}}^c(u,v_{k1}) < \ell.$  
Then $dist_{G_k}^c(u,v) \le dist_{G_{k-1}}^c(u, v_{k1}) + 1 \le \ell.$
This completes the proof.
\end{proof}

 \begin{lemma}
Let $u_1,u_2,u_1',u_2'$ be vertices in $G$ where the vertices are distinct except for possibly $u_1 = u_2.$  If $u= u_1 = u_2,$  then there are edge-disjoint rainbow paths $P_i,\ i = 1,2$ in $G$ from $u$ to $u_i'$ where $|P_i| + |P_2| \le n-1.$ If $u_1 \ne u_2,$ then there are edge-disjoint rainbow paths $P_i,\ i = 1,2$ in $G$ originating at distinct vertices in $\{ u_1, u_2 \}$ and terminating at distinct vertices in $\{ u_1', u_2' \}$ where $|P_i| + |P_2| \le n-2.$ 
\label{lem2}
\end{lemma}

\begin{proof}
  We shall prove that the lemma is true for $G_i,\ i = 2,3, \dots ,n-1$ using induction on $i.$ One can easily check that the lemma is true when $i\in \{ 2,3 \}.$  Suppose that for all $i < k$ the lemma is true for $G_i$ where $k\ge 4.$  We will show that it is true for $G_k.$  Let $u_1, u_2, u_1', u_2' \in V(G_k)$ where the vertices are distinct except for possibly $u_1 = u_2.$  Suppose first that $u = u_1 = u_2.$  If $\{ u, u_1', u_2' \} \subseteq V(G_{k-1}),$ then by assumption, there are edge-disjoint rainbow paths $P_1,P_2$ in $G_{k-1}$ where for $j=1,2,$ $P_j$ has terminal vertices $u$ to $u_j'$ and $|P_1| + |P_2| \le \nu(G_{k-1})-1 =  k-1.$ Thus we may assume that $v_k \in \{ u, u_1', u_2' \}.$

Suppose $u = v_k.$  Recall that $e_{kj} = v_k v_{kj}, \ j = 1,2.$  Assume first that $\{ v_{k1}, v_{k2} \} \cap \{ u_1', u_2' \} = \emptyset.$  By assumption, there are edge-disjoint rainbow paths $P_i',\ i = 1,2$ in $G_{k-1}$ originating at  
distinct vertices in $\{ v_{k1}, v_{k2} \}$ and terminating at distinct vertices in $\{ u_1', u_2' \}$ where $|P_1'| + |P_2'| \le \nu(G_{k-1}) -2 = k-2.$   We may assume that $v_{kj} \in V(P_j'),\ j = 1,2.$  For $j =1,2,$ let $P_j = P_j' + e_{kj} + u.$  
Then $P_j,\ j = 1,2$ are edge-disjoint rainbow paths from $u$ to vertices in $\{ u_1', u_2' \}$ where $|P_1| + |P_2| \le k = \nu(G_k) -1.$
 Suppose now that $\{ v_{k1}, v_{k2} \} \cap \{ u_1', u_2' \} \ne \emptyset.$  We may assume that
 $v_{k1} = u_1'.$  Suppose $v_{k2} \ne u_2'.$
 By Lemma \ref{lem1}, there is a rainbow path $P_2'$ from $v_{k2}$ to $u_2'$ in $G_{k-1}$ having length at most $\lfloor \frac k2 \rfloor .$  Let $P_1 = uu_1'$ and $P_2 = P_2' + e_{k2} + u.$  Then $|P_1| + |P_2| \le  1+ \lfloor \frac k2 \rfloor  \le k-1$ (since $k \ge 4$). 
  Suppose that $v_{k2}  =  u_2'.$ Then the paths $P_i = uu_i',\ i = 1,2$ will suffice.  
  
  From the above, we may assume that $u \ne v_k$ and thus either $v_k = u_1'$ or $v_k= u_2'.$  Without loss of generality, we may assume that $v_k = u_1'.$
Furthermore, we may assume that $v_{k1} \ne u_2'.$  By assumption there are two edge-disjoint rainbow paths $P_i',\ i = 1,2$ in $G_{k-1}$from $u$ to the vertices $v_{k1}$ and $u_2'$ where $|P_1'| + |P_2'| \le \nu(G_{k-1}) -1 = k-1.$  Let $P_1 = P_1' + e_{k1} + u_1'$ and let $P_2 = P_2'.$  Then we have that $|P_1| + |P_2| \le k = \nu(G_k) -1.$  This proves the lemma in the case where $u_1 = u_2.$

From the above, we may assume that $u_1 \ne u_2.$  As before, we may assume that $v_k \in \{ u_1, u_2, u_1', u_2' \}$.   Without loss of generality, we may assume that $v_k = u_1.$  Suppose first that $v_{k1} \in \{ u_1', u_2' \}$.   We may assume that $v_{k1} = u_1'.$  Let $P_1 = u_1u_1'$.  By Lemma \ref{lem1}, there is a rainbow path $P_2$ in $G_{k-1}$ from $u_2$ to $u_2'$ of length at most $\lfloor \frac {k}2 \rfloor$.
Now $|P_1| + |P_2| \le \lfloor \frac {k}2 \rfloor + 1 \le \frac k2 + 1 \le k-1 =\nu(G_k) -2$.  The last inequality follows from the fact that $k\ge 4.$  Thus we may assume that $v_{k1} \not\in \{ u_1', u_2' \}$ and by symmetry, the same holds for $v_{k2}.$
Either $u_2 \ne v_{k1}$ or $u_2 \ne v_{k2},$ and we may assume that $u_2 \ne v_{k1}.$  By assumption, there are edge-disjoint rainbow paths $P_i',\ i = 1,2$ in $G_{k-1}$ originating at vertices in  $\{ v_{k1}, u_2 \}$ and terminating at vertices in $\{ u_1', u_2'\}$ where $|P_1'| + |P_2'| \le \nu(G_{k-1})-2 = k-2.$  We may assume that $v_{k1} \in V(P_1').$  Let $P_1 = P_1' + e_{k1} + u_1$ and let $P_2 = P_2'.$  Then $|P_1| + |P_2| \le k-1 = \nu(G_k) -2.$  This completes the proof.
\end{proof}

\begin{lemma}
Let $u,v \in V(G)$ where  $\{ u,v \} \ne \{ v_{11}, v_{22} \}.$  Then there are two edge-disjoint rainbow paths $P_i,\ i = 1,2$ from $u$ to $v$ where $|P_1| + |P_2| \le n.$
\label{newlem3}
\end{lemma}

\begin{proof}
By induction, we shall prove that the assertion is true for all $G_i,\ i \in [n-1].$  That is, for all $i \in [n-1],$ we shall show that if $u,v \in V(G_i)$ where  $\{ u,v \} \ne \{ v_{11}, v_{22} \},$ then there are two edge-disjoint rainbow paths $P_j,\ j = 1,2$ in $G_i$ from $u$ to $v$ where $|P_1| + |P_2| \le \nu(G_i) = i+1.$
The assertion is clearly true for $G_1$ and $G_2.$
Assume that the lemma holds for all $G_i$, where $i <  k$ and $k\ge 4.$  If $\{ u,v \} \subseteq V(G_{k-1}),$ then by assumption there are edge-disjoint rainbow paths $P_i,\ i = 1,2$ in $G_{k-1}$ from $u$ to $v$ where $|P_1| + |P_2| \le \nu(G_{k-1}) = k.$  Thus we may assume that $u = v_k.$   Suppose first that $v \not\in \{ v_{k1}, v_{k2} \}$.
Given that $k \ge 4,$ it follows by Lemma \ref{lem2} that there are two edge-disjoint rainbow paths $P_j',\ j = 1,2$ from $v$ to $v_{kj}$ where $|P_1'| + |P_2'| \le \nu(G_{k-1})-1 = k-1.$  For $j=1,2,$ let $P_j = P_j' + e_{kj} + u$.  Then $P_i, \ i = 1,2$ are two edge-disjoint rainbow paths from $u$ to $v$ where $|P_1| + |P_2| \le k+1 = \nu(G_k).$  Suppose instead that $v \in \{ v_{k1}, v_{k2} \}.$ Here we may assume that $v = v_{k1}.$  Let $P_1 = uv_{k1}.$  By Lemma \ref{lem1}, there is a rainbow path $P_2'$ in $G_{k-1}$ from $v_{k2}$ to $v$ where $|P_2'| \le \lfloor \frac {k}2 \rfloor .$  Let $P_2 = P_2' + e_{k2} + u$.  Then $P_i,\ i = 1,2$ are edge-disjoint rainbow paths from $u$ to $v$ where $|P_1| + |P_2| \le \lfloor \frac {k}2 \rfloor + 2 \le \frac k2 + 2 \le k+1.$  Thus the assertion is true for $G_k$.  The proof now follows by induction. 
\end{proof}


Let $K_{23}^+$ denote the graph obtained from the complete bipartite graph $K_{23}$ by adding an edge between the two vertices of degree three.

\begin{lemma}
Assume that $G \not\simeq K_{23}^+.$ Let $N = M+T$ be a graphic extension of $M$ where $T$ is a $3$-circuit of $N.$  Then there exists an $T$-SRCT for $(M,c).$
%
\label{newLemma1}
\end{lemma}

\begin{proof}
We first note that without the assumption that $G \not\simeq K_{23}^+$ the theorem is false.  We shall use a proof by induction on $n.$ The lemma is seen to be true for the smallest case $n=3.$  Assume that $n\ge 4$ and the assertion is true for graphs having fewer than $n$ vertices.  We shall show that the assertion holds for $G$ as well.  Let $T = \{ f_1, f_2, f_3 \}$ and $H = G + T$ where in $H$, $f_1 = u_1u_2,\ f_2 = u_1 u_3,$ and $f_3 = u_2u_3.$  Let $U = \{ u_1,u_2,u_3 \}.$   

\begin{noname}
We may assume that if $u$ is a vertex of degree two in $G$, then $u\in U.$\label{nona1}
\end{noname} 

\begin{proof}
Suppose $G$ has a vertex $u \not\in \{ u_1, u_2, u_3 \}$ of degree two where $E_G(u) = \{ e, f \}.$  If $c(e) = c(f),$ then let $G' = G/e\backslash f$ and $M' = M/e\backslash f.$ Let $c' = c\big| E(M').$   If $G' \simeq K_{23}^+,$ then it is easy to find a $T$-SRCT for $(M,c).$  Thus we may assume that $G' \not\simeq K_{23}^+$ and hence $G'$ satisfies the conditions of the lemma and hence there is a $T$-SRCT for $(M',c')$ which is also seen to be a $T$-SRCT for $G'.$  Suppose $c(e) \ne c(f).$  Then $u\in \{ v_{11}, v_{12} \}$ and $e_1 \in \{ e,f \}.$  Furthermore, assuming that $e =e_1$ we have that $f \in X_2$ and we may assume that $f = e_{21}.$ 
Now let $G' = G/e_{21} \backslash e_{22},$ $M' = M/e_{21}\backslash e_{22},$ and $c' = c \big| E(M').$  Then as before, regardless of whether $G' \simeq K_{23}^+$ or not, $(M',c')$ has a $T$-SRCT $\{ C_1', C_2', C_3' \}$. If $e_1 \not\in C_1' \cup C_2' \cup C_3'$, then $\{ C_1', C_2', C_3' \}$ is seen to be a $T$-SRCT for $(M,c).$   If $e_1 \in C_1' \cup C_2' \cup C_3'$, and say $e_1 \in C_1'$, then there is a circuit $C_1$ where $E(C_1) = E(C_1') + e_{21}$ and $\{ C_1, C_2', C_3' \}$ is seen to be a $T$-SRCT for $(M,c).$
\end{proof}

\begin{noname}
We may assume that there is no vertex $v\in V(G) - U$ where $d_G(v) =3$ and $c(e) = c(f)$ for some edges $e,f \in E_G(u).$\label{nona2!}
\end{noname}

\begin{proof}
 Suppose there exists $v \in V(G) - U$ where $E_G(v) = \{ e,f,g \}$  and $c(e) = c(f).$  Let $G' = G/e\backslash f$ and let $G'' = G'/e\backslash f.$  If $e,g$ belong to a triangle in $G$ having edges $e,g,h$, then $c(g) = c(h)$ (since $(G,c)$ has no rainbow cycles).  Thus we see that either $e,g$ or $f,g$ belong to no triangle of $G$ depending on whether $c(e) \ne c(g)$ or $c(f) \ne c(g).$  Without loss of generality, we may assume that $e,g$ belong to no triangle of $G.$   Then $G'$ is simple.  Let $M' = M/e\backslash f.$ Given that $c$ is stratified, it can be easily shown that $G' \not\simeq K_{23}^+$.  Thus by assumption there is a $T$-SRCT $\{ C_1', C_2', C_3' \}$ for $(M',c')$ where $c' = c \big| E(M').$  If $g \not\in C_1' \cup C_2' \cup C_3'$,  then $\{ C_1', C_2',C_3' \}$ is seen to be a $T$-SRCT for $(M,c).$  If one of the circuits contains $g$, say $g\in C_1'$, then for some rainbow circuit $C_1$, $E(C_1) = E(C_1') + e$ and $\{ C_1,C_2',C_3' \}$ is a $T$-SRCT for $(M,c).$ 
 \end{proof}
 
 \begin{noname}
 We may assume that $v_{n-1} = u_1,\ v_{n-2} = u_2,$ $d_G(v_{n-1}) = d_G(v_{n-2}) = 2.$ 
 \label{nona3}
 \end{noname}
  
\begin{proof}
By (\ref{nona1}) and (\ref{nona2!}), 
we have $\{ v_{n-2}, v_{n-2} \} \subset \{ u_1, u_2, u_3 \}$ 
and we may assume that $v_{n-1} = u_{1}$ and $v_{n-2} = u_2.$  
Suppose first that $d_G(u_2) = 3.$ Then $u_2 = v_{(n-1)j}$ for some $j$, and we may assume that $u_2 = v_{(n-1)1}.$   Suppose $u_3 \ne v_{(n-1)2}.$  By Lemma \ref{lem2}, there are two edge-disjoint rainbow paths $P_1'$ and $P_2'$ in $G - u_1$ where for $i=1,2,$ $P_i'$ has terminal vertices $v_{(n-1)i}$ and $u_3$ and $|P_1'| + |P_2'| \le \nu(G-u_1)-1 = n-2.$  Let $P_1 = u_1u_2,\ P_2 = P_1',$ and $P_3 = P_2' + e_{(n-1)2} + u_1.$ Then $P_i,\ i = 1,2,3$ are edge-disjoint rainbow paths where for all $i<j,$ $|P_i| + |P_j| \le n-1.$  For $i = 1,2,3,$ let $C_i = E(P_i) + f_i.$  Then $\{ C_1, C_2, C_3 \}$ is seen to be an $T$-SRCT for $(M,c).$
Suppose instead that $v_{(n-1)2} = u_3$ (and $v_{(n-1)1} = u_2$).  For $i=1,2,$ let $C_i = \{ f_i, e_{(n-1)i} \}$ and let $C_3$ be an $f_3$-rainbow circuit for $(M - e_{(n-1)1} - e_{(n-1)2}, c)$ containing $f_3$ (noting that $|C_3| \le r(M - e_{(n-1)1} - e_{(n-1)2})  = n-1$).  Then $\{ C_1, C_2, C_3 \}$ is seen to be a $T$-SRCT for $(M,c).$  Thus we may assume that $d_G(v_{n-2}) = d_G(u_2) = 2.$
\end{proof}

\begin{noname}
We may assume that $u_3 \not\in \{ v_{(n-1)1}, v_{(n-1)2} \} \cup \{ v_{(n-2)1}, v_{(n-2)2} \}.$
\label{nona3.5}
\end{noname}

\begin{proof}
By (\ref{nona3}), we may assume that $v_{n-1} = u_1,\ v_{n-2} = u_2,$ $d_G(v_{n-1}) = d_G(v_{n-2}) = 2.$ 
Suppose that $u_3 \in \{ v_{(n-1)1}, v_{(n-1)2} \}$, where we may assume that $u_3 = v_{(n-1)1}.$  Suppose first that $\{ v_{(n-1)1}, v_{(n-1)2} \} \cap \{ v_{(n-2)1}, v_{(n-2)2} \} = \emptyset.$  By Lemma \ref{lem2}, there are edge-disjoint rainbow paths $P_i,\ i = 1,2$ where for $i=1,2,$ $P_i$ originates at $v_{(n-1)i}$ and terminates at a vertex in $\{ v_{(n-1)1}, v_{(n-1)2} \},$ the terminal vertices of $P_i,\ i = 1,2$ being distinct.  Moreover, $|P_1| + |P_2| \le \nu(G - u_1 - u_2) -2 = n-4.$  Suppose first that for $i=1,2$, $P_i$ has terminal vertex $v_{(n-2)i}.$ 
Let $C_1 = E(P_2) + \{ f_1, e_{(n-1)2}, e_{(n-2)2} \},\ C_2 = E(P_1) + \{ f_2, e_{(n-2)1} \}$, and $C_3 = \{ f_3, e_{(n-1)1} \}.$  Then $|C_1| + |C_2| \le (n-4) + 5 = n+1$ and thus it is seen that $\{ C_1, C_2, C_3 \}$ is a $T$-SRCT for $(M,c).$
If instead for $i=1,2,$ $P_i$ has terminal vertex $v_{(n-2)(3-i)},$ then by a similar construction, one can find a $T$-SRCT.  Suppose instead that $\{ v_{(n-1)1}, v_{(n-1)2} \} \cap \{ v_{(n-2)1}, v_{(n-2)2} \} \ne \emptyset.$  Suppose first that 
$u_3 \in  \{ v_{(n-2)1}, v_{(n-2)2} \}.$ We may assume that $u_3 = v_{(n-1)1} = v_{(n-2)1}.$  By Lemma \ref{lem1}, there exists a rainbow path $P$ from $v_{(n-1)2}$ to $v_{(n-2)2}$ in $G - u_1 - u_2$ having length at most $\lfloor \frac {n-2}2 \rfloor \le n-4.$
Let $C_1 = E(P) + \{ f_1, e_{(n-1)2}, e_{(n-2)2} \},$ $C_2 = \{ f_2, e_{(n-1)1} \},$ and $C_3 = \{ f_3, e_{(n-2)1} \}.$  Then $|C_1| \le n-1$ and $\{ C_1, C_2, C_3 \}$ is seen to be a $T$-SRCT for $(M,c).$  Suppose instead that $u_3 \not\in \{ v_{(n-2)1}, v_{(n-2)2} \}.$ 
We may assume that $v_{(n-1)2} = v_{(n-2)2}.$  By Lemma \ref{lem1}, there exists a rainbow path $P$ from $u_3$ to $v_{(n-2)1}$ in $G - u_1 - u_2$ having length at most $\lfloor \frac {n-2}2 \rfloor \le n-4.$
Let $C_1 = \{ f_1, e_{(n-1)2}, e_{(n-2)2} \},\ C_2 = \{ f_2, e_{(n-1)1} \},$ and $C_3 = E(P) + \{ f_3, e_{(n-2)1} \}.$  Then $|C_3| \le (n-4) + 2 = n-2$ and hence $\{ C_1, C_2, C_3 \}$ is seen to be a $T$-SRCT for $(M,c).$
It follows from the above that we may assume that $u_3 \not\in \{ v_{(n-1)1}, v_{(n-1)2} \}$.  Given that $u_1$ and $u_2$ are interchangeable, we may also assume that $u_3 \not\in \{ v_{(n-2)1}, v_{(n-2)2} \}$.
\end{proof}

\begin{noname}
If $\{ v_{(n-1)1}, v_{(n-1)2} \} \cap \{ v_{(n-2)1}, v_{(n-2)2} \} = \emptyset,$ then $v_{n-3} = u_3,\ d_G(v_{n-3}) = 2,\ d_G(v_{n-4}) =4$ and we may assume that $v_{n-4} = v_{(n-1)1} = v_{(n-3)1}.$
\label{nona4}
\end{noname}

\begin{proof} Suppose that  $\{ v_{(n-1)1}, v_{(n-1)2} \} \cap \{ v_{(n-2)1}, v_{(n-2)2} \} = \emptyset.$  Then \\ $d_G(v_{n-3}) \le 3$ and hence $v_{n-3} = u_3.$  By (\ref{nona3.5}), we may assume that $u_3 \not\in \{ v_{(n-1)1}, v_{(n-1)2} \} \cup \{ v_{(n-2)1}, v_{(n-2)2} \}.$  It follows that $d_G(u_3) = 2.$
 Now $d_G(v_{n-4}) \le 4.$  If $d_G(v_{n-4}) \le 3,$ then $v_{n-4} \in \{ u_1, u_2, u_3 \},$ which is impossible.  Thus $d_G(v_{n-4}) = 4$ and hence $v_{n-4}$ is adjacent to $u_i$ and $u_3$ for some $i \in \{ 1,2 \}.$ We may assume that $v_{n-4}$ is adjacent to $u_1$ and $u_3$ and $v_{n-4} = v_{(n-1)1} = v_{(n-3)1}.$
 \end{proof}  
 
 \begin{noname}
 If $\{ v_{(n-1)1}, v_{(n-1)2} \} \cap \{ v_{(n-2)1}, v_{(n-2)2} \} = \emptyset,$ then we may assume that $v_{(n-3)2} = v_{(n-1)2}.$
 \label{nona5}
 \end{noname}
 
 \begin{proof}
Assume that $\{ v_{(n-1)1}, v_{(n-1)2} \} \cap \{ v_{(n-2)1}, v_{(n-2)2} \} = \emptyset.$  By (\ref{nona4}), we have that $v_{n-3} = u_3,\ d_G(u_3) = 2,\ d_G(v_{n-4}) =4$ and we may assume that $v_{n-4} = v_{(n-1)1} = v_{(n-3)1}.$
Suppose that $v_{(n-3)2} \ne v_{(n-1)2}.$
Let $H = G - u_1 - u_3$ and let $u_1 ' = v_{(n-1)2}, u_2' = u_2,$ and $u_3' = v_{(n-3)2}.$  By Lemma \ref{lem2}, there are edge-disjoint rainbow paths $P_1'$ and $P_3'$ in $H$, where $P_1'$ joins $u_1'$ to $u_2'$, $P_3'$ joins $u_2'$ to $u_3'$ and $|P_1'| + |P_3'| \le \nu(H) - 1 = n-3.$  Let $P_1$ be the rainbow path in $G$ where $E(P_1) = E(P_1') + e_{(n-1)2} + u_1$ and let $P_3$ be the rainbow path in $G$ where $E(P_3) = E(P_3') + e_{(n-3)2} +u_3.$  Let $P_2 = u_1e_{(n-1)1}v_{(n-1)1}e_{(n-3)1}u_3.$
 We have $|P_1| + |P_3| \le n-1$ and hence $|P_1| + |P_2| \le n-1$ and $|P_2| + |P_3| \le n-1.$  For $i=1,2,3,$ let $C_i = E(P_i) + f_i$.  Then $\{ C_1, C_2, C_3 \}$ is seen to be a $T$-SRCT for $(M,c).$ 
 Thus we can assume that $v_{(n-3)2} = v_{(n-1)2}.$
 \end{proof}
 
 \begin{noname}
 We may assume that $\{ v_{(n-1)1}, v_{(n-1)2} \} \cap \{ v_{(n-2)1}, v_{(n-2)2} \} \ne \emptyset.$
 \label{nona6!}
 \end{noname}
 
\begin{proof} Suppose $\{ v_{(n-1)1}, v_{(n-1)2} \} \cap \{ v_{(n-2)1}, v_{(n-2)2} \} = \emptyset.$  By (\ref{nona4}) and (\ref{nona5}), we may assume that $v_{n-4} = v_{(n-1)1} = v_{(n-3)1}$ and $v_{(n-3)2} = v_{(n-1)2}.$  Let $H= G - U - v_{n-4}$.   
 By Lemma \ref{lem2}, there are edge-disjoint rainbow paths $P_1'$ and $P_3'$ in $H$, where $P_1'$ joins $v_{(n-1)2}$ to $v_{(n-2)1}$, $P_3'$ joins $v_{(n-1)2}$ to $v_{(n-2)2}$ and $|P_1'| + |P_3'| \le \nu(H) - 1 = n-5.$  Let $P_1$ be the rainbow path in $G$ where $E(P_1) = E(P_1') + e_{(n-1)2} + e_{(n-2)1}$ and let $P_3$ be the rainbow path in $G$ where $E(P_3) = E(P_3') + e_{(n-2)2} + e_{(n-3)1}.$  Let $P_2 = u_1e_{(n-1)1}v_{(n-1)1}e_{(n-3)1}u_3.$ 
 Letting $C_i = E(P_i) + f_i$, we see that $\{ C_1, C_2, C_3 \}$ is a $T$-SRCT for $(M,c).$ 
 \end{proof}
   
\begin{noname}
We may assume that for $j=1,2,$ $v_{(n-1)j} = v_{(n-2)j}.$
\label{nona7!}
\end{noname}

\begin{proof}   
By (\ref{nona6!}), we may assume that $\{ v_{(n-1)1}, v_{(n-1)2} \} \cap \{ v_{(n-2)1}, v_{(n-2)2} \} \ne \emptyset.$  Without loss of generality, we may assume that 
$v_{(n-1)1} = v_{(n-2)1}.$  
Suppose that $v_{(n-1)2} \ne v_{(n-2)2}.$  Assume first that $u_3 \not\in N_G(\{ u_1, u_2 \} ).$   For $j=1,2,$ let $u_j' = v_{(n-j)2},$ and let $u_3' = u_3.$  By Lemma \ref{lem2}, there edge-disjoint rainbow paths $P_2'$ and $P_3'$ in $G - u_1 - u_2$, joining $u_3'$ to $u_1'$ and $u_2'$, respectively, where $|P_1'| + |P_2'| \le n-3.$ Let $P_1 = u_1 e_{(n-1)1}v_{(n-1)1}e_{(n-2)1}u_2$, and $P_2$ and $P_3$ be the rainbow paths where $P_2 = P_2' + e_{(n-1)2 + u_1}$ and $P_3 = P_3' + e_{(n-2)2} + u_2.$
Then $P_i,\ i = 1,2,3$ are edge-disjoint rainbow paths where for all $i<j,$ $|P_i| + |P_j| \le n-1.$  For $i= 1,2,3,$ let $C_i = E(P_i) + f_i$.  Then $\{ C_1, C_2, C_3 \}$ is seen to be a $T$-SRCT for $(M,c).$
Assume now that $u_3 \in N_G(\{ u_1, u_2 \} ).$  Note that here $n\ge 6.$  Suppose first that $u_3 = v_{(n-1)1}.$  Then $v_{n-3} = u_3;$ otherwise, $d_G(v_{n-3}) \le 3$ implying that $v_{n-3} \in U$, a contradiction.
By Lemma \ref{lem1}, there is a rainbow path $P$ in $G-U$ from $v_{(n-1)2}$ to $v_{(n-2)2}$ of length at most $\lfloor \frac {n-3}2 \rfloor \le n-4$ (since $n\ge 6$).
Let $P_1$ be the path where $E(P_1) = E(P) + \{ e_{(n-1)2}, e_{(n-2)2} \}$.  Let $P_2 = u_1e_{(n-1)1}u_3$ and $P_3 = u_2e_{(n-2)1}u_3.$  For $i = 1,2,3$ let $C_i = E(P_i) + f_i.$  Then $\{ C_1, C_2, C_3 \}$ is seen to be a $T$-SRCT for $(M,c).$
Suppose instead that $u_3 \in \{ v_{(n-1)2}, v_{(n-2)2} \}.$  We may assume that $u_3 = v_{(n-1)2}.$  By Lemma \ref{lem1}, there is a rainbow path $P$ in $G - u_1 - u_2$ from $u_3$ to $v_{(n-2)2}$ of length at most $\lfloor \frac {n-2}2 \rfloor \le n-4$ (since $n\ge 6$).
Let $P_1 = u_1 e_{(n-1)1}v_{(n-1)1}e_{(n-2)1}u_2$, $P_2 = u_1 e_{(n-1)2}u_3$, and let $P_3$ be path where $P_3 = P + e_{(n-2)2} + u_2.$  For $i = 1,2,3$ let $C_i = E(P_i) + f_i$.  Then $\{ C_1, C_2, C_3 \}$ is seen to be a $T$-SRCT for $(M,c).$ 
Thus by the above, we may assume that $v_{(n-1)2} = v_{(n-2)2}.$ 
\end{proof}

\begin{noname}
We may assume that $v_{n-3} = u_3.$
\label{nona8!}
\end{noname}

\begin{proof}
By (\ref{nona7!}), we may assume that for $j= 1,2,$ $v_{(n-1)j} = v_{(n-2)j}.$  If $u_3 \in N_G(u_1),$ then finding a $F$-SRCT is straightforward.  Thus we may assume that $u_3 \not\in N_G(u_1).$

Suppose that $v_{n-3} \ne u_3.$  Then $d_G(v_{n-3}) \ge 4,$ for otherwise, (\ref{nona1}) and (\ref{nona2!}) would imply that $v_{n-3} \in U,$ a contradiction.  It now follows that $v_{n-3} \in \{ v_{(n-1)1}, v_{(n-1)2}.$ 
By symmetry, we may assume that $v_{n-3} = v_{(n-1)2}.$  If $\{ v_{11}, v_{12} \} \ne \{ v_{(n-1)1}, u_3 \},$ then by Lemma \ref{newlem3}, there are two edge-disjoint rainbow paths $P_2'$ and $P_3'$ in $G - u_1 - u_2 - v_{n-3}$ between $u_3$ and $v_{(n-1)1}$ where $|P_2'| + |P_3'| \le n-3.$  Now let $P_1 = u_1 e_{(n-1)2}v_{(n-1)2}e_{(n-2)2}u_2$, $P_2 = P_2' + e_{(n-1)1} + u_1$ and $P_3 = P_3' + e_{(n-1)2} + u_2.$  Then $P_i,\ i = 1,2,3$ are edge-disjoint rainbow paths where for all $1\le i< j\le 3,$ $|P_i| + |P_j| \le n-1.$  We now see that
for $C_i = E(P_i) + f_i,\ i = 1,2,3,$ $\{ C_1, C_2, C_3 \}$ is a $T$-SRCT for $(M,c).$  Suppose on the other hand that $\{ v_{11}, v_{12} \} = \{ v_{(n-1)1}, u_3 \}.$  Then we may assume that $v_{11} = v_{(n-1)1}$ and $v_{12} = u_3.$
If $n\ge 6,$ then we see that $d_G(v_{n-4}) \le 3$ and hence $v_{n-4} \in U.$  However, this would imply that $v_{n-4} = u_3$, which is not possible since $u_3 = v_{12}.$  Thus $n=5$ and we see that $G \simeq K_{23}^+,$ contradicting our assumptions.  Thus we may assume that $v_{n-3} = u_3.$
\end{proof}

By (\ref{nona8!}), we may assume that $u_3 = v_{n-3}.$  Again, if $u_3 \in N_G(u_1),$ then it is straightforward to find a $F$-SRCT.  Thus we may assume that $u_3 \not\in N_G(u_1).$  Since $u_3 = v_{n-3},$ it follows that $d_G(u_3) = 2.$  It is now seen that the roles of $u_1, u_2,$ and $u_3$ are interchangeable and one could for example assume that $v_{n-2} = u_3$ and $v_{n-3} = u_2.$  Because of this, we may assume that $N_G(u_3) = N_G(u_1)$ and for $j=1,2,$ $v_{(n-3)j} = v_{(n-1)j}.$
If $n=5,$ then $G \simeq K_{23}^+,$ contradicting our assumptions. Thus $n\ge 6.$  
By Lemma \ref{lem1}, there is a rainbow path $P$ between $v_{(n-1)1}$ and $v_{(n-1)2}$ in $G-U$ of length at most $\lfloor \frac {n-3}2 \rfloor \le n-5$ (since $n\ge 6$).  Let $P_1$ be the path where $E(P_1) = E(P) + e_{(n-1)2} + e_{(n-2)1}$ and let
$P_2 = u_1 e_{(n-1)1}v_{(n-1)1}e_{(n-3)1}u_3$ and $P_3 = u_2e_{(n-2)2}v_{(n-2)2}e_{(n-3)2}u_3.$  Then $P_i,\ i = 1,2,3$ are edge-disjoint rainbow paths where for all $1\le i< j \le 3,$ $|P_i| + |P_j| \le n-1.$  For $i = 1,2,3,$ let $C_i = E(P_i) + f_i.$  Then $\{ C_1, C_2, C_3 \}$  is seen to be a $T$-SRCT for $(M,c).$ This completes the proof.
\end{proof}

\subsection{Proof of Theorem \ref{the-NoRainbowCircuit} for $M = M(G)$ and $\varepsilon(M) = 2r(M)-1$ }\label{sec-ProofNoRainbowGraphic2n-1}

Let $T = \{ f_1, f_2, f_3 \}$ and $H = G + T$ where in $H$, $f_1 = u_1u_2,\ f_2 = u_1 u_3,$ and $f_3 = u_2u_3.$  Let $M = M(G).$
Suppose first that $G \not\simeq K_{23}^+.$  Then Lemma \ref{newLemma1} implies that $(M,c)$ has a $T$-SRCT $\{ C_1, C_2, C_3 \}$ where $f_i \in E(C_i),\ i = 1,2,3.$
On the other hand, suppose $G \simeq K_{23}^+.$  Then $V(G) = \{ v_{11}, v_{12}, v_2, v_3, v_4 \}.$  For at least one $j \in [2],\ d_G(v_{1j}) = 4$ and we may assume that $d_G(v_{11}) = 4.$  Then there are two possible graphs for $G$ (up to re-colouring) as illustrated in Figure \ref{figure2}.

 \begin{center}
 \begin{figure}
\begin{tikzpicture}[scale=0.7]
  
  \node[blackvertex](v) at (5,5){};
  \node at (5,4.5) {$v_{11}$};
  \node[blackvertex](v) at (7,5){};
  \node at (7,4.5) {$v_{12}$};
  \node[blackvertex](v) at (3,8){};
  \node at (3,8.5) {$v_{2}$};
  \node[blackvertex](v) at (6,8){};
  \node at (6,8.5) {$v_{3}$};
  \node[blackvertex](v) at (9,8){};
  \node at (9,8.5) {$v_{4}$};
  
  \draw (5,5) -- (7,5) node [below, midway]{\mbox{\tiny $1$}};
  \draw (5,5) -- (3,8) node [below, midway]{\mbox{\tiny $2$}};
  \draw (5,5) -- (6,8) node [above, midway]{\mbox{\tiny $3$}};
  \draw (5,5) -- (9,8) node [above, midway]{\mbox{\tiny $4$}};
  
  \draw (7,5) -- (3,8) node [above, midway]{\mbox{\tiny $2$}};
  \draw (7,5) -- (6,8) node [above, midway]{\mbox{\tiny $3$}};
  \draw (7,5) -- (9,8) node [below, midway]{\mbox{\tiny $4$}};


\begin{scope}[xshift=9cm]
 \node[blackvertex](v) at (5,5){};
  \node at (5,4.5) {$v_{11}$};
  \node[blackvertex](v) at (7,5){};
  \node at (7,4.5) {$v_{2}$};
  \node[blackvertex](v) at (3,8){};
  \node at (3,8.5) {$v_{12}$};
  \node[blackvertex](v) at (6,8){};
  \node at (6,8.5) {$v_{3}$};
  \node[blackvertex](v) at (9,8){};
  \node at (9,8.5) {$v_{4}$};
  
  \draw (5,5) -- (7,5) node [below, midway]{\mbox{\tiny $2$}};
  \draw (5,5) -- (3,8) node [below, midway]{\mbox{\tiny $1$}};
  \draw (5,5) -- (6,8) node [above, midway]{\mbox{\tiny $3$}};
  \draw (5,5) -- (9,8) node [above, midway]{\mbox{\tiny $4$}};
  
  \draw (7,5) -- (3,8) node [above, midway]{\mbox{\tiny $2$}};
  \draw (7,5) -- (6,8) node [above, midway]{\mbox{\tiny $3$}};
  \draw (7,5) -- (9,8) node [below, midway]{\mbox{\tiny $4$}}; 

%
%
%
%
\end{scope}  
\end{tikzpicture}
\caption{The two possible graphs for $G\simeq K_{23}^+.$} \label{figure2}
\end{figure}
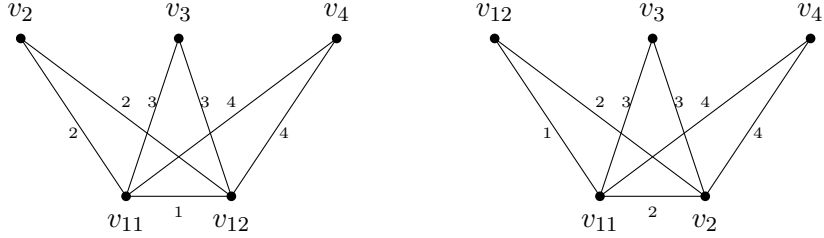
\end{center}

Suppose $G$ is the first graph in Figure \ref{figure2}.  If $\{ u_1, u_2, u_3 \} \ne \{ v_2, v_3, v_4 \},$ then it is readily seen that one can find a $T$-SRCT.  On the other hand, if $\{ u_1, u_2, u_3 \} = \{ v_2, v_3, v_4 \}$, then it is seen that for all $1 \le i < j \le 3,$ there is a $T$-SRCP $\{ C_1, C_2 \}$ where $f_i\in C_1,$ $f_j \in C_2$ and $e_1 \not\in C_1 \cup C_2.$  A similar argument can be used for the second graph in Figure \ref{figure2} and 
it follows that  
Theorem \ref{the-NoRainbowCircuit} i) holds for $M = M(G).$

To prove Theorem \ref{the-NoRainbowCircuit} ii) is true for $M= M(G),$ suppose $N = M+x$ is a graphic extension of $M$ where $x$ and $e = e_1$ are non-parallel.  Let $x = uv$ be the corresponding edge added to $G.$ By Lemma \ref{newlem3}, there exist two edge-disjoint rainbow paths $P_i,\ i = 1,2$ in $G$ from $u$ to $v$ where $|P_1| + |P_2| \le n.$  For $i= 1,2$ let $C_i = E(P_i) + x.$  Then $C_i,\ i = 1,2$ are circuits in $N$ where $E(C_1) \cap E(C_2) = \{ x \}$ and $|C_1| + |C_2| \le n+ 2 = r(M) +3.$  It follows that $C_i,\ i = 1,2$ is an $x$-semi-SRCP for $(M,c).$   This proves part ii) for $(M,c).$

\section{Circuit-achromatic cographic matroids $M$ with $2r(M)-1$ elements}\label{sec-BigTheoremi)andii)Cographic}

In this section, we prove Theorem \ref{the-NoRainbowCircuit} for cographic matroids $M$ where $\varepsilon(M) = 2r(M) -1.$   
For the remainder this section, we assume that $M = M^*(G)$ is a connected simple cographic matroid where $r(M) = n-1$ and $\varepsilon (M) = 2(n-1) -1 = 2n-3.$  We may also assume that $M$ is connected, meaning that $G$ is also connected.
  Let $c: E(G) \rightarrow [n-1]$ be an $(n-1)$ - colouring of $M$ such that $|c^{-1}(1)| = 1$ and for $i =2, \dots ,n-1$, $|c^{-1}(i)| = 2.$  Assume that $(M,c)$ is circuit-achromatic; that is, $G$ has no rainbow cocycles.   
By Theorem \ref{the1}, the colouring is stratified.   For $i = 1, \dots ,n-1$, let $X_i = c^{-1}(i)$ where $X_1 = \{ e_1\}$ and for $i\ge 2$, $X_i = \{ e_{i1}, e_{i2} \}.$

\subsection{Stratified Colourings in Co-Graphic Matroids}\label{sec-stratified}

In this section, we discuss how one constructs a stratified colouring in a cographic matroid.   First, we will describe how one constructs a cographic extension of cographic matroid $M$ by a single element.  We do this by introducing a splitting operation which is the reverse of contraction.   Let $H$ be a graph and let $v\in V(H).$  Partition $E_H(v)$ into two non-empty sets $E_1$ and $E_2.$  Let $H'$ be the graph obtained from $H$ as follows:  the graph $H'$ has vertex set $V(H') = V(H) - v + \{ v_1,v_2 \}$ and edge set $E(H') = E(H) + f_1 + \cdots + f_k.$  Here if $e \in E(H) - E_H(v)$, then $e$ has the same endvertices in $H'$ as in $H.$  For $i=1,2,$ if $e = uv \in E_i,$ then $e$ has endvertices $u$ and $v_i$ in $H'.$
 In $H'$, the edges $f_i,\ i \in [k]$ are parallel and have endvertices $v_1$ and $v_2.$   We say that $H'$ is obtained from $H$ by {\it splitting $v$ and adding $f_1, \dots ,f_k$} and we write $H  \xrightarrow[v]{f_1, \dots ,f_k}H'$.  When $v$ is not specified, we will simply write
 $H  \xrightarrow[]{f_1, \dots ,f_k}H'.$ See Figure \ref{figure1}.
 
 \begin{center}
 \begin{figure}
\begin{tikzpicture}[scale=0.7]

\node[blackvertex] (v) at (10,10){};
\node at (10, 10.5) {v};
\draw (10,10) -- +({3*cos(190)},{3*sin(190)}); 
\draw (10,10) -- +({3*cos(200)},{3*sin(200)});
\draw (10,10) -- +({3*cos(210)},{3*sin(210)});
\draw (10,10) -- +({3*cos(220)},{3*sin(220)});
\draw (10,10) -- +({3*cos(230)},{3*sin(230)});    
\draw (10,10) -- +({3*cos(-10)},{3*sin(-10)}); 
\draw (10,10) -- +({3*cos(-20)},{3*sin(-20)});
\draw (10,10) -- +({3*cos(-30)},{3*sin(-30)});
\draw (10,10) -- +({3*cos(-40)},{3*sin(-40)});
\draw (10,10) -- +({3*cos(-50)},{3*sin(-50)});
\draw[red]  (8.5, 10) arc (180:240:1.5);
\draw[red]   (11.5,10) arc (0:-55:1.5);
\node at (8.5, 10.5) {$E_1$};
\node at (11.5, 10.5) {$E_2$};

\begin{scope}[xshift=9cm]
\node[blackvertex] (v) at (10,10){};
\node[blackvertex] (v) at (12,10){};
\node at (10,10.5) {$v_1$};
\draw (10,10) -- +({3*cos(190)},{3*sin(190)}); 
\draw (10,10) -- +({3*cos(200)},{3*sin(200)});
\draw (10,10) -- +({3*cos(210)},{3*sin(210)});
\draw (10,10) -- +({3*cos(220)},{3*sin(220)});
\draw (10,10) -- +({3*cos(230)},{3*sin(230)});    
\draw[red]  (8.5, 10) arc (180:240:1.5);

\node at (12, 10.5) {$v_2$};
\draw (12,10) -- +({3*cos(-10)},{3*sin(-10)}); 
\draw (12,10) -- +({3*cos(-20)},{3*sin(-20)});
\draw (12,10) -- +({3*cos(-30)},{3*sin(-30)});
\draw (12,10) -- +({3*cos(-40)},{3*sin(-40)});
\draw (12,10) -- +({3*cos(-50)},{3*sin(-50)});
\draw[red]   (13.5,10) arc (0:-55:1.5);

\draw (10,10) .. controls (10.5,11) and (11.5,11) .. (12,10);
\draw (10,10) .. controls (10.6,10.7) and (11.4,10.7) .. (12,10);
\draw (10,10) .. controls (10.7,10.5) and (11.3,10.5) .. (12,10);
\draw (10,10) .. controls (10.8,10.3) and (11.2,10.3) .. (12,10);

\node at (11,11){$\scriptstyle{f_1}$};
\node at (11,10){$\scriptstyle{f_k}$};

\end{scope}  

\end{tikzpicture}
\caption{The operation $H  \xrightarrow[v]{f_1, \dots ,f_k}H'$} \label{figure1}
\end{figure}
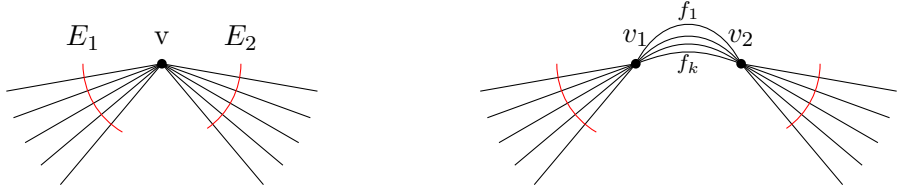

\end{center}

\subsection{The construction of $G$}\label{sec-constructG}

 Let $\nu(G) =m$.  We have $r^*(M) = r(M(G)) = m-1,$ and thus $m-1 = \varepsilon(M) - (n-1) = n-2$ and $m = n-1.$  Since $c$ is stratified, we may assume that $G$ is constructed as follows:
Let $G_1$ be the graph consisting of the single loop $e_1$ incident with a vertex $v_1.$ Next, let $G_2$ be the graph obtained from $G_1$ by changing $e_1$ to an edge $e_1 = v_{11}v_{12}$ and then adding two parallel edges $e_{21} = v_{11}v_{12}$ and $e_{22} = v_{11}v_{12}.$  Continuing,  let $G_3$ be obtained from $G_2$ by splitting one of the vertices $v_{11}$ or $v_{12}$ and adding edges $e_{31}$ and $e_{32}$; that is $G_2 \xrightarrow[]{e_{31}, e_{32}} G_3.$  Continuing, we define $G_1,G_2, \dots ,G_{n-1}$ where 
 $G_{i-1} \xrightarrow[]{e_{i1}, e_{i2}} G_i.$ 
 Then $G = G_{n-1}$. It is easy to see that for $i =1, \dots ,n-1,$ $M(G_i^*)$ has no rainbow circuits (i.e. $G_i$ has no rainbow cocycles) and $r(M(G_i^*)) = i.$  We will make frequent use of the next observation whose easy inductive proof we leave to the reader.
 
 \begin{observation}
 For all $i\ge 2$, $G_i$ has no loops, the maximum vertex degree in $G_i$ is at most $i+1$ and for all pairs of distinct vertices $u, v \in V(G_i),\ d_{G_i}(u) + d_{G_i}(v) \le i+4.$
 \label{obs-basicobs}
 \end{observation}

We shall first establish Theorem \ref{the-NoRainbowCircuit} ii) for $M$.  Let $H$ be the graph obtained from $G$ by splitting a vertex $v$ and adding an edge $x$; that is, $G \xrightarrow[v]{x} H$ and $N = M^*(H)$ is a cographic extension of $M$ by an element $x.$
  Given that $\nu(H) = \nu(G) + 1 =  m + 1 = n$ and $\varepsilon(H) = \varepsilon(G) +1,$ we see that $r(N) = r(M) = n-1$. 

\begin{lemma}
If $r(M) \ge 2$ and $x$ and  $e=e_1$ are non-parallel, then there is an $x$-semi-SRCP for $(M,c).$
%
\label{colem2}
\end{lemma}

\begin{proof}
By induction on $r(M).$  The lemma is seen to be true when $r(M)=2.$  We may assume the lemma is true for such cographic matroids of rank less than $n-1.$
 Suppose that for some $i$, the edges $e_{i1}$ and $e_{i2}$ form a monochromatic digon (i.e. $2$-cycle) in $G.$  Let $M' = M/e_{i1}\backslash e_{i2}$ and $N' = N/e_{i1}\backslash e_{i2}.$ Let $c' = c\big| E(M').$  
 We observe that any circuit of $N'$ is also a circuit of $N$.  As such, $x$ and $e_1$ are non-parallel in $N'.$  By assumption, there exists an $x$-semi-SRCP for $(M',c')$ and such is also seen to be a $x$-semi-SRCP for $(M,c).$  Thus we may assume that $G$ contains no monochromatic digons. 
 Given that $|c(M)| = n-1$ and $e_1$ is a colour-singular edge, there are exactly $n-2 = \nu(H) -2$ colour classes with two elements.  Since by assumption no two edges in any such colour class form a monochromatic digon in $G,$ it follows that $H -x$ has
  at least two rainbow vertices, say $x_1$ and $x_2.$   Given that $G= H/x$ contains no rainbow vertices (because $(M,c)$ is circuit-achromatic), it follows that in $H$, $x = x_1x_2.$  Now by Observation \ref{obs-basicobs}, $d_G(v) = d_{G_{n-1}}(v) \le n$, it follows that $d_H(x_1) + d_H(x_2) \le n+2 = r(M) +3.$  For $i=1,2,$ let $C_i \subseteq E_H(x_i)$ be a circuit of $N$ containing $x.$  Then $\{ C_1,C_2\}$ is seen to be an $x$-semi-SRCP for $(M,c).$  
\end{proof}

\subsection{Completing the proof of Theorem \ref{the-NoRainbowCircuit} i) for $M$}\label{sec-completeNoRainbowCircuitCographic}

To complete the proof of Theorem \ref{the-NoRainbowCircuit} i) for $M$, it will suffice to prove the following:

\begin{lemma}
Suppose $G \xrightarrow[v^1]{e^1} G^1 \xrightarrow[v^2]{e^2} G^2 \xrightarrow[v^3]{e^3} G^3$ where $T=\{ e^1, e^2, e^3 \}$ is a co-independent circuit of $N = M^*(G^3).$  
Extend the colouring $c$ to $G^3$ where $c(e^1) = n,\ c(e^2) = n+1$ and $c(e^3) = n+2.$
Then there are three edge-disjoint rainbow cocycles $C_i,\ i \in [3]$ in $G^3$ where $e^i \in C_i,\ i = 1,2,3$ and for all $i < j$, $|C_i| + |C_j| \le r(M) + 2 = n+1.$  That is, $\{ C_1, C_2, C_3 \}$ is a T-SRCT for $(M,c).$
\label{colem3}
\end{lemma} 

\begin{proof}
In $G^3$ let $e^i = u_{1i}u_{2i},\ i \in [3].$   For $i=1,2$ we may assume that $u_{i1}, u_{i2}, u_{i3}$ belong to the same component of $G^3 - T.$  For $i = 1,2,$ let $U_i$ be the set of vertices in the component containing $u_{ij},\ j\in [3].$  For all elements $x\in E(N)$ which are not colour-singular, we let $x'$ denote the other element in the colour class containing $x.$  

\begin{noname}
There is no element $x\in E(G)$ and subset $T' \subseteq T$ for which $x \cup T'$ is the edge set of a cycle in $G^3.$\label{nonacolem3.part0}
\end{noname}

\begin{proof}
Suppose there is a cycle $C$ in $G^3$ for which $E(C) = T' + x$ for some subset $T' \subseteq T$ and element $x \in E(G).$  Then $C/T'$ is loop $x$ in $G$, a contradiction since $G$ has no loops by Observation \ref{obs-basicobs}.
\end{proof}

\begin{noname}
We may assume that $(G^3,c)$ has no monochromatic digons.
\label{nonacolem3.part1}
\end{noname}

\begin{proof} Suppose $(G^3,c)$ has a monochromatic digon with edges $e,f.$  Let $H = G/e\backslash f$ and $H^3 = G^3 /e\backslash f.$  We see that $H \xrightarrow[]{e^1} H^1 \xrightarrow[]{e^2} H^2 \xrightarrow[]{e^3} H^3.$
Since $T$ is a cocyle of $G^3,$ it is a cocycle of  $H^3.$ Then $H$ and $H^3$ satisfy the conditions of the lemma.  Thus the lemma holds for $H^3$ and hence it contains $3$ edge-disjoint rainbow cocycles $C_i',\ i = 1,2,3$ where $e^i \in C_i',\ i = 1,2,3$ and for all $1 \le i < j \le 3$, $|C_i| + |C_j| \le (r(M)-1) + 2 = r(M) +1.$ Since $C_i',\ i = 1,2,3$ are also cocycles of $G^3,$ the lemma holds for $G^3.$  Thus we may assume that $(G^3,c)$ has no monochromatic digons.
\end{proof}
 
 By (\ref{nonacolem3.part1}), we may assume that $G^3$ has no monochromatic digons.
It follows that for each colour class $X_i$,  there is at most one vertex $v \in V(G^3)$ for which $X_i \subseteq E_{G^3}(v).$
Given that there are only $r(M) -1 = n-2 = m-1 = \nu(G^3) - 4$ colour classes with $2$ elements, there are at least $4$ rainbow vertices in $G^3$.  Furthermore, since $G$ contains no rainbow cocycles, each rainbow vertex of $G^3$ is one the vertices $u_{ij}\ i\in [3], \ j \in [2].$  

\begin{noname}
We may assume that for $i= 1,2,$ $|U_i| \ge 2.$
\label{nonacolem3.part2}
\end{noname}

\begin{proof} Suppose that for some $i,$ $|U_i| = 1.$  We may assume that $U_1 = \{ u \}.$  Then $u = u_{11} = u_{12} = u_{13}$ and $u, u_{21}, u_{22}, u_{23}$ are rainbow vertices in $G^3.$  Since $G$ has no loops, none of the vertices $u_{21}, u_{22}, u_{23}$ are adjacent.  We also have that for $j=1,2,3,$ $d_{G^3}(u_{2j}) \ge 2$.  Furthermore, since (by Observation \ref{obs-basicobs}) the degree of any vertex in $G$ is at most $r(M)+1 = n$, it follows that for all $1\le i < j\le 3,$ $d_{G^3}(u_{2i}) + d_{G^3}(u_{2j}) \le n+1= r(M) +2.$  Letting  $C_i = E_{G^3}(u_{2i}),\ i = 1,2,3$, it is seen that $\{ C_1, C_2, C_3 \}$ is a $T$-SRCyP for $(M,c).$  Thus we may assume that for $i=1,2,$ $|U_i| \ge 2.$
\end{proof}

\begin{noname}
If for some $i\in \{ 1,2 \}$ and distinct $j_1, j_2 \in \{ 1,2,3 \}$,  $u_{ij_1} = u_{ij_2}$, then we may assume that the vertices $u_{(3-i)j}, \ j \in [3],$ are non-distinct.
\label{nonacolem3.part3}
\end{noname}

\begin{proof}
Suppose that for some $i\in \{ 1,2 \}$ and distinct $j_1, j_2 \in \{ 1,2,3 \}$, $u_{ij_1} = u_{ij_2}$ and the vertices $u_{(3-i)j}, \ j \in [3],$ are distinct.
We may assume that
$u = u_{11} = u_{12}$ and the vertices $u_{2j}, \ j \in [3]$ are distinct.  By (\ref{nonacolem3.part2}), we may assume that $u_{13} \ne u.$  We see that $v^3 \in V(G)$ and hence by Observation \ref{obs-basicobs}, $d_G(v^1) + d_G(v^3) \le n+3.$
Since $\{ e^1, e^2 \}$ is not a cocycle of $G^3$, we have that $d_{G^3}(u) \ge 3.$  Furthermore, (\ref{nonacolem3.part0}) implies that $u_{21}$ and $u_{22}$ are non-adjacent in $G^3.$ 
At most one of the vertices $u,u_{13}$ and $u_{2j},\ j \in [3]$ is not a rainbow vertex.  Suppose that $u_{2j},\ j = 1,2$ are rainbow vertices.

Suppose that $u$ is not a rainbow vertex (and $u_{13}$ and $u_{2j},\ j\in [2]$ are rainbow vertices).
We have
\begin{align*}
d_G(v^1) &= d_{G^3}(u) + d_{G^3}(u_{21})  + d_{G^3}(u_{22}) -4\\
d_{G^3}(u_{21})+ d_{G^3}(u_{22}) &= d_{G}(v^1) - d_{G^3}(u) +4\\
&\le d_G(v^1) + 1 \le n+1 =r(M) +2.
\end{align*}
Suppose that $d_{G^3}(u_{13}) = 2.$  Then $u_{13}$ is a rainbow vertex where for $j=1,2,$ $d_{G^3}(u_{2j}) + d_{G^3}(u_{13}) \le r(M) +2.$  Letting $C_j = E_{G^3}(u_{2j}),\  j \in [2],$ and $C_3 = E_{G^3}(u_{13}),$ one sees that $\{ C_1, C_2, C_3 \}$ is a $T$-SRCyT for $(M,c).$  Suppose that $d_{G^3}(u_{13}) \ge 3.$  If $d_{G^3}(u_{23})\ge 3,$ then $d_{G^3}(u_{13}) \le d_G(v^3)-1$ and hence for $j=1,2$, $$d_{G^3}(u_{2j}) + d_{G^3}(u_{13}) \le d_G(v^1) -1 + d_G(v^3) -1 \le r(M) +2.$$  We can now define $C_i,\ i \in [3]$ as before.

Suppose $u$ is a rainbow vertex in $G^3.$  Define $C_i,\ i = 1,2$ as before and define $C_3 = E_{G^3}(U_1 -u).$ Then $C_i,\ i \in [3]$ are pairwise disjoint rainbow cocycles and for $j=1,2,$ $|C_j| + |C_3| \le d_G(v^1) + 1 \le r(M) +2.$
Thus $\{ C_1, C_2, C_3 \}$ is seen to be a $T$-SRCyT for $(M,c).$

From the above, we may assume that for some $j\in [2],$ $u_{2j}$ is not a rainbow vertex.  Then $u$ and $u_{i3},\ i = 1,2$ are rainbow vertices. Since $G$ has no rainbow vertices, $v^3$ is not a rainbow vertex and $E_{G^3}(\{ u_{13}, u_{23} \})$ contains a colour class $\{ y, y' \} .$  Then $y$ and $y'$ are non-incident and hence $G^3$ must have at least $5$ rainbow vertices, implying that $u$ is a rainbow vertex.  We can now argue as before.
\end{proof}

\begin{noname}
We may assume that the vertices $u_{ij},\ i \in [2], \ j \in [3]$ are distinct.
\label{nonacolem3.part4}
\end{noname}

\begin{proof}
Suppose the vertices $u_{ij},\ i \in [2],\ j\in [3]$ are not distinct.  We may assume that
$u = u_{11} = u_{12}.$  By (\ref{nonacolem3.part2}), we may assume that $u_{13} \ne u$.  By (\ref{nonacolem3.part3}), we may assume that for some $1\le i < j\le 3,$ $u_{2i} = u_{2j}.$  Noting that $u_{21} \ne u_{22},$ we may assume that $v = u_{22} = u_{23}.$  Given that $G^3$ has at least $4$ rainbow vertices, it follows that all the vertices $u,v, u_{21}, u_{13}$ are rainbow vertices and both pairs of vertices $u, u_{13}$ and $u_{21}, v$ are nonadjacent in $G^3$ (by (\ref{nonacolem3.part0})).  
Let $C_1 = E_{G^3}(U_2 - v),\ C_2 = E_{G^3}(U_1 - u_{13} + u_{21}), \ C_3 = E_{G^3}(U_1 -u).$  Given $d_G(v^1) \le r(M)+1$, it follows that for all $1 \le i < j\le 3$, $|C_i| + |C_j| \le r(M)+2$ and $\{ C_1, C_2, C_3 \}$ is seen to be a $T$-SRCyT for $(M,c).$ 
\end{proof}

By (\ref{nonacolem3.part4}), we may assume that all the vertices $u_{ij},\ i \in [2], \ j \in [3]$ are distinct.  Then $v^j,\ j = 1,2,3$ are distinct vertices of $G.$  For all $i,j,$ let $D_{ij} = E_{G^3}(u_{ij}).$  Since there are at least $4$ rainbow vertices, we may assume that $u_{11}$ and $u_{21}$ are rainbow vertices.  

\begin{noname}
We may assume that the vertices $u_{1j},\ j \in [3]$ and $u_{2j},\ j \in [2]$ are rainbow vertices in $G^3.$ \label{nonacolem3.part4.5}
\end{noname}

\begin{proof} Since $G$ has no rainbow vertices, it follows that $E_G(v^1) = D_{11} \triangle D_{21} = E_{G^3}( \{ u_{11}, u_{21} \} )$ contains a colour class, say $\{ x, x' \}$, where we may assume that $x\in D_{11}$ and $x\in D_{21}.$  Then $x$ and $x'$ are non-incident and
as such, there must be at least $5$ rainbow
 vertices in $G^3$. Thus we may assume that $u_{1j},\ j= 1,2,3$ and $u_{2j},\ j = 1,2$ are rainbow vertices.
 \end{proof}

\begin{noname}
We may assume that for some $i\in [2],$ at least one pair of the vertices in $\{ u_{ij}\ \big| \  j \in [3]\}$ is adjacent.
\label{nonacolem3.part5}
\end{noname}

\begin{proof}
Suppose that for $i=1,2,$ the vertices in $\{ u_{ij}\ \big| \ j\in [3] \}$ are pairwise non-adjacent.  
 For $i = 1,2,3$ let
$$C_i = \left\{ \begin{array}{lr} D_{1i} & \mathrm{if}\ d_{G^3}(u_{2i}) >2\\ D_{2i} & \mathrm{if}\ d_{G^3}(u_{2i}) = 2 \end{array} \right.$$
 Given that each vertex of $G$ has degree at most $r(M)+1$, it follows that for $j = 1,2,3$, either $d_{G^3}(u_{1j}) \le r(M)$ or $d_{G^3}(u_{2j}) = 2.$  
Since it follows by Observation \ref{obs-basicobs} that for any two vertices $x,y$ in $V(G)$, $d_G(x) + d_G(y) \le r(M)+4,$ one can check that for all $1 \le i < j\le 3,$ $|C_i| + |C_j| \le r(M)+2.$  Thus $\{ C_1, C_2, C_3 \}$ is a $T$-SRCyT for $(M,c).$
\end{proof}

\begin{noname}
We may assume that all the vertices $u_{ij},\ i\in [2],\ j\in [3]$ are rainbow vertices.
\label{nonacolem3.part6}
\end{noname}
\begin{proof}
By (\ref{nonacolem3.part6}), we may assume that $u_{1j},\ j \in [3]$ and $u_{2j},\ j \in [2]$ are rainbow vertices in $G^3.$  Assume that $G^3$ only has $5$ rainbow vertices and $u_{23}$ is not a rainbow vertex.
 Then $d_{G^3}(u_{23}) \ge 3$.  Recall that $\{ x,x' \}$ is a colour class where $x\in D_{11}$ and $x' \in D_{21}.$ Using the same argument for $D_{12}\triangle D_{22}$ as we did for $D_{11} \triangle D_{21},$ it follows that there is a colour class, say $\{ y, y' \}$,where
 $y\in D_{12}$ and $y' \in D_{22}.$  If $\{ x,x' \} \ne \{ y, y' \}$, then $G^3$ must have at least $6$ rainbow vertices, a contradiction.  Thus $y =x$ and $y' = x'$ and hence $x = u_{11}u_{12}$ and $x' = u_{21}u_{22}.$
Since $T$ does not properly contain a cocycle and $M$ is simple (and hence has no $2$-cocycles),  there exists $(i,j) \in [2] \times [2]$ such that $d_{G^3}(u_{ij}) \ge 3$ and $d_{G^3}(u_{(3-i)(3-j)}) \ge 3.$  We may assume that this is true for $i= j=1.$   Suppose that $u_{13}$ is not adjacent to $u_{12}.$
Let $C_1 = D_{21},\ C_2 = D_{12}$ and $C_3 = D_{13}.$  Then for all $i\in [3],\ |C_i| \le d_G(v^i) -1$ and for all $1\le i<j \le 3,$ $|C_i| + |C_j| \le d_G(v^i) + d_G(v^j) -2 \le r(M) + 2.$  Thus $\{ C_1, C_2, C_3 \}$ is a $T$-SRCyT for $(M,c).$
Suppose that $u_{13}$ is adjacent to $u_{12}$ and let $z = u_{12}u_{13}.$  If $z \ne e_1,$ then $z$ and $z'$ are non-incident edges and it follows that $G^3$ contains at least $6$ rainbow vertices, a contradiction.  Thus $z = e_1.$  This also shows that $u_{13}$ is not adjacent to $u_{11}.$
If $d_{G^3}(u_{21}) \ge 3$, then defining $C_1 = D_{11},\ C_2 = D_{21}$ and $C_3 = D_{13},$ we see that as before $\{ C_1, C_2, C_3 \}$ is a $T$-SRCyT for $(M,c).$  Thus we may assume that $d_{G^3}(u_{21}) = 2.$  Now define $C_1 = D_{21},\ C_2 = D_{12}$ and 
$C_3 = E_{G^3}(U_2 - u_{21} - u_{22}).$  Then $\{ C_1, C_2, C_3 \}$ is seen to be a $T$-SRCyT for $(M,c).$  
\end{proof}

By (\ref{nonacolem3.part6}), $G^3$ has (exactly) $6$ rainbow vertices which are the vertices $u_{ij},\ i \in [2], \ j \in [3].$  Since $M = M^*(G)$ is simple, $G$ has no $2$-cocycles and thus for all $j\in [3],\ d_{G^3}(u_{1j}) + d_{G^3}(u_{2j}) \ge 5.$ Moreover, since $(M,c)$ is circuit-achromatic, for all $j\in [3],$ there is a colour class $\{ f_{1j}, f_{2j} \}$ where  $\{ f_{1j}, f_{2j} \} \subseteq E_{G^3}(\{ u_{1j}, u_{2j} \} )$ and for $i\in [2]$ and $j\in [3],$ $f_{ij}\in D_{ij}.$ 

For a function $\phi: [3] \rightarrow [2]$ we define  $S_{\phi} = \{ u_{\phi(i)i} \ \big| \ i = 1,2,3 \}.$  The set $S_{\phi}$ is said to be {\bf short} if it is independent and for all $1\le i < j\le 3,$ $|D_{\phi(i)i}| + |D_{\phi(j)j}| \le r(M) +2.$ For the function $\phi$, we define its counterpart  $\phi': [3] \rightarrow [2]$ where for all $i,$ $\phi'(i) = 3 - \phi(i).$ We let $S_{\phi}' = S_{\phi_i'}.$
The next observation, whose proof we omit, follows from the fact that $r(M) \ge 3$ and for all $x,y \in V(G),$ $d_G(x) + d_G(y) \le r(M)+4.$

\begin{noname}
Let $\phi: [3] \rightarrow [2]$ where $S_{\phi}$ is independent. For all $1 \le i < j \le 3$, if 
$$|D_{\phi(i)i}| + |D_{\phi(j)j}| \le 5\  \mathrm{or}\  |D_{\phi'(i)i}| + |D_{\phi'(j)j}| \ge 6 \ \ \ (\star),$$ then $|D_{\phi(i)i}| + |D_{\phi(j)j}| \le r(M) +2.$  Thus if ($\star$) holds for all $1 \le i < j \le 3,$ then $S_{\phi}$ is short.
\label{nonacolem3.part7}
\end{noname}

Since $G^3$ has $6$ rainbow vertices, only edges belonging to at most $3$ colour classes in $(G,c)$ can join pairs of vertices among $ u_{ij} \ i \in [2],\ j\in [3]$ and in the case that there are $3$ such colour classes, one of them is $X_1 = \{ e_1 \}$. Thus at least two of the colour classes $\{ f_{1j}, f_{2j} \},\ j \in [3]$ are equal.  We may assume that $f_{11} = f_{12}$ and $f_{21} = f_{22}.$
Also, for some $i \in \{ 1,2 \},$ at least two vertices of $u_{ij},\ j  \in [3]$ are non-adjacent.  We may assume that $u_{11}$ and $u_{13}$ are non-adjacent.  Let $\phi_1:[3] \rightarrow [2]$ where for $i\in \{ 1,3\},\ \phi_1(i) = 1$ and $\phi_1(2) =2.$  Let $S_1 = S_{\phi_1} = \{ u_{11}, u_{22}, u_{13} \}$ and let $S_1' = S_{\phi_1}' = \{ u_{21}, u_{12}, u_{23} \}.$   Let $\phi_2:[3] \rightarrow [2]$ where $\phi_2(1) =2$ and for $i\in \{ 2,3 \},$ $\phi_2(i) =1.$ Let $S_2 = S_{\phi_2} = \{ u_{21}, u_{12}, u_{13} \}$ and $S_2' = S_{\phi_2}'.$ Then $S_1$ is independent.  We may assume that for $i= 1,2,$ $S_i$ and $S_i'$ are not short.  Given that $S_1$ is not short, it follows by (\ref{nonacolem3.part7}) that at least one of the cocyles $D_{21}, D_{12},$ or $D_{23}$ is a $2$-cocycle.  

\begin{noname}
We may assume that $|D_{21}| \ge 3.$
\label{nonacolem3.part8}
\end{noname}

\begin{proof}
Suppose $|D_{21}| =2.$ Then $|D_{22}| \ge 3$; otherwise, $E_{G^3}(\{ u_{21}, u_{22} \} ) = \{ e^1, e^2 \} \subset T$ is a cocycle of $G^3.$  Since $|D_{22}| \ge 3,$ it follows that $|D_{12}| \le d_G(v^2) -1.$  We also have that $|D_{11}| \ge 3;$ otherwise, $E_G(v^1) = \{ f_{11}, f_{21} \}$ which is impossible since $G$ has no $2$-cocycles.  Then $S_1'$ is seen to be an independent set.
It follows by (\ref{nonacolem3.part7}) that $|D_{13}| = 2.$  If $|D_{23}| =2,$ then we have $|D_{13}| = |D_{23}| = 2$ and hence $d_g(v^3) =2$, contradicting the assumption that $G$ has no $2$-cocycles.
Thus $|D_{23}| \ge 3.$  
 If $u_{12}$ and $u_{13}$ are non-adjacent, then $S_2$ is independent and it is seen to be short by (\ref{nonacolem3.part7}).  Thus we may assume that $f_{13} = u_{12}u_{13}.$  

Suppose $f_{23}$ is not incident with $u_{22}.$  Since $E_{G^3}(\{ u_{11}, u_{12}, u_{21}, u_{22} \})$ is seen to be a cocycle of $G$, it must contain a colour class, say $\{ f, f' \}$ of $(G,c).$  Now $f$ and $f'$ must be incident and we may assume that $f\in D_{11}$ and $f' \in D_{12}.$ 
Then $|D_{12}| \ge 4.$   Since $|D_{22}| \ge 3,$ it follows that $d_G(v^2) \ge 5.$  Since $d_G(v^1) + d_G(v^2) \le r(M) +4,$ it follows that $|D_{11}| = d_G(v^1) \le r(M) -1.$ Thus $|D_{11}| + |D_{13}| \le r(M) +1.$  By (\ref{nonacolem3.part7}), we have that for all $j\in \{ 1,3 \}$, $|D_{1j}| + |D_{22}| \le r(M) +2$, implying that $S_1$ is short - a contradiction.
It follows that $f_{23} = u_{22}u_{23}.$  

If for some $i\in \{ 1,2 \},$ $|D_{i2}| \ge 4,$ then $d_G(v^2) \ge 5$ and one can show that either $S_1$ or $S_1'$ is short using similar arguments as before.  Thus for all $i\in [2],$ $|D_{i2}| = 3.$  Given that $E_G(\{ v^1, v^2, v^3 \} )$ is a cocycle in $G$, it contains a colour class $\{ f, f' \}$ where $f\in E_G(v^1)$ and $f' \in E_G(v^3).$  In $G^3,$ we have $f \in E_{G^3}(u_{11})$ and $f' \in E_{G^3}(u_{23})$ and as such, $f$ and $f'$ are non-incident.  This yields a contradiction since there are at most two such colour classes.
It follows from the above that we may assume that $|D_{21}| \ge 3.$
\end{proof}

\begin{noname}
We may assume that $|D_{12}| \ge 3.$
\label{nonacolem3.part9}
\end{noname}

\begin{proof}
By (\ref{nonacolem3.part8}), we may assume that $|D_{21}| \ge 3.$  Suppose $|D_{12}| =2$.  Then $|D_{11}| \ge 3$ and $|D_{22}| \ge 3.$ 
Furthermore,  $u_{12}$ and $u_{13}$ are non-adjacent and $S_2 = \{ u_{21}, u_{12}, u_{13} \}$ is independent.  If $|D_{23}| \ge 3,$ then it follows by (\ref{nonacolem3.part7}) that $S_2$ is short. Thus we may assume that $|D_{23}| = 2.$
Suppose $u_{21}$ and $u_{23}$ are non-adjacent.  Then $S_1'$ is independent.  Since $|D_{11}| + |D_{21}| -2 = d_G(v^1) \le r(M) +1,$ it follows that $|D_{21}| \le r(M).$  Since $|D_{12}| = |D_{23}| = 2,$ it follows that $S_1'$ is short, a contradiction.  Thus we may assume that $u_{21}$ and $u_{23}$ are adjacent and $f_{23} = u_{21}u_{23}.$  If $|D_{13}| = 2,$ then the previous argument shows that $S_2$ is short .  Thus we may assume that $|D_{13}| \ge 3.$
We see that $E_{G^3}(\{ u_{11}, u_{12}, u_{21}, u_{22} \} )$ is a cocycle of $G$ and hence it contains a colour class $\{ f,f' \}$ of $(G,c)$ where $f$ and $f'$ are incident.  We may assume that  $f\in D_{21}$ and $f' \in D_{22}.$
It follows that $|D_{21}| \ge 4$ and $d_G(v^1) \ge 5.$  This in turn implies that $d_G(v^2) \le r(M) + 4 - d_G(v^1) \le r(M) -1$ and hence $|D_{22}| \le r(M) -1.$  Now it follows by (\ref{nonacolem3.part7}) that
$S_2' = \{ u_{11}, u_{22}, u_{23} \}$ is short, yielding a contradiction.  It follows that we may assume that $|D_{12}| \ge 3.$
\end{proof}
 
By (\ref{nonacolem3.part8}) and (\ref{nonacolem3.part9}), we may assume that $|D_{21}| \ge 3$ and $|D_{12}| \ge 3.$  We may also assume that $|D_{23}| =2,$ for otherwise it follows by (\ref{nonacolem3.part7}) that $S_1$ is short.  We now have that $|D_{13}| \ge 3.$ If 
$u_{22}$ and $u_{23}$ are non-adjacent, then $S_2'$ is independent and it follows by (\ref{nonacolem3.part7}) that it is short.
Thus we may assume that $f_{23} = u_{22}u_{23}.$   Now $u_{21}$ and $u_{23}$ are non-adjacent and $S_1'$ is independent.  If $|D_{11}| \ge 3,$ then it follows by (\ref{nonacolem3.part7}) that $S_1'$ is short.  Thus we may assume that $|D_{11}| = 2.$
We have $d_G(v^2) = |D_{12}| + |D_{22}| - 2 \ge 4$ and as such $|D_{21}| = d_G(v^1) \le r(M) + 4 - d_G(v^2) \le r(M).$

If $|D_{22}| \ge 4,$ then it follows by (\ref{nonacolem3.part7}) that $S_1' = \{ u_{21}, u_{12}, u_{23} \}$.  Thus we may assume that $|D_{22}| = 3.$  
If $f_{13} \ne u_{12}u_{13},$ then the cocycle $E_{G^3}(\{ u_{11}, u_{12}, u_{21}, u_{22} \} )$ contains a colour class $\{ f,f' \}$ of $(G,c)$ where $f$ and $f'$ are incident and we may assume that $f\in D_{21}$ and $f'\in D_{22}.$
Then $|D_{22}| \ge 4,$ contradicting our assumptions.  Thus $f_{13} = u_{12}u_{13}.$  Given that we may assume that $|D_{22}|= 3,$ by symmetry we may also assume that $|D_{12}| = 3.$  Now $E_{G^3}(\{ u_{ij}\ \big| \ i \in [2],\ j \in [3] \} )$ is a cocycle of $G$ and hence it contains a colour class $\{ f,f' \}$ where $f\in D_{21}$ and $f' \in D_{13}.$  However $f$ and $f'$ are non-incident and this yields a contradiction. This completes the proof of the lemma.    
\end{proof} 

\section{Circuit-achromatic graphic matroids $M$ with\\ $2(r(M)-1)$ elements.}\label{sec-BigTheoremiii)Graphic}

In this section, we prove Theorem \ref{the-NoRainbowCircuit} iii) when $M$ is graphic.  This was the most difficult challenge in this paper.   The proof is broken into two parts.  The first part, which is the bulk of the proof,
rests on proving the next theorem.

\begin{theorem}
Let $k \in [3]$ and let $(G,c)$ be a simple coloured graph where $\nu(G) = n$ and $\varepsilon(G) = 2(n-1) - k.$   Furthermore, assume that $c: E(G) \rightarrow [n-1]$ is a $2$-bounded $(n-1)$-colouring of $G$ where $(G,c)$ is circuit-achromatic and 
$A = \{ a_1, \dots ,a_k \}$ is the set of colour-singular edges in $(G,c).$
Let $G'$ be the graph obtained from $G$ by adding three edges $e_1, e_2, e_3$ to $G$ forming a triangle $T$ and extend $c$ to a colouring $c'$ for $G'$ where for $i = 1,2,3,$ $c'(e_i) = n.$  Suppose that the edges of $T \cup A$ do not induce a subgraph isomorphic to $K_4$ in $G'.$
Then there is a pair of edge-disjoint rainbow cycles $\{ C_1, C_2 \}$ in $G'$ where $|C_1| + |C_2| \le n+2$ and for $i=1,2,$ $|C_i \cap T| =1.$  
\label{newlemma2}
\end{theorem}

\begin{proof}
We shall refer to a pair of edge-disjoint rainbow cycles $C_1,C_2$ in $G'$ for which $|E(C_i) \cap T| = 1,\ i = 1,2$ and $|C_1| + |C_2| \le n+2$ as a {\bf near}-$\mathbf{T}${\bf-short rainbow cycle pair} (near $T$-SRCP).  It follows by Lemma \ref{lem2} that the theorem holds for $k=1.$  Thus we may assume that $1<k \le 3$ and the theorem holds when there are fewer than $k$ colour-singular edges.  Assume that the theorem is false.  Among the counterexamples, let $(G,c)$ be a counterexample having the least number of vertices and let $n = \nu(G).$    For $i = 1, \dots ,n-1,$ let $A_i = c^{-1}(i).$ We may assume that for $i = 1, \dots ,k,$ $A_i = \{ a_i\} $.  For $i= k+1, \dots ,n-1$, let $A_i = \{ a_{i1}, a_{i2} \}.$ By Theorem \ref{the1}, the colouring $c$ is stratified and we may assume that $G$ is constructed as follows:
\begin{itemize}
\item  For $i=1, \dots ,k$, $G_i$ is the spanning subgraph of $G$ where $E(G_i) = \{ a_1, \dots ,a_i \}$.
\item  For $i = k+1, \dots ,n-1,$ $G_i$ is the spanning subgraph of $G$ where  $E(G_i) = E(G_{i-1}) \cup A_i$.
\item  For $i= 2, \dots ,n-1$,  $\kappa (G_i) = \kappa (G_{i-1}) -1.$
\end{itemize}

Note that $G = G_{n-1}.$  For $i=k+1, \dots ,n-1$ and $j=1,2,$ let $a_{ij} = u_{ij}v_{ij}$ where $u_{i1}, u_{i2}$ (resp. $v_{i1}, v_{i2}$) belong to the same component of $G_{i-1}.$  Let $A' = \bigcup_{i=k+1}^{n-1} A_i.$
 We let $e_1 = x_1x_2, e_2 = x_1x_3,$ and $e_3 = x_2x_3$
 
 For $i=1, \dots ,n-1$, define a weighting $w_i: V(G_i) \rightarrow \{ 0, 1,2 \}$ as follows:
\begin{itemize}
\item[i)] $w_i(v) = 0$ if $d_{G_i}(v) \ge 4$ or $d_{G_i}(v) = 0.$
\item[ii)] $w_i(v) = 1$ if $d_{G_i}(v) =2$ and $v$ is a rainbow vertex in $G_i$, or $d_{G_i}(v) =3$ and $|c^{-1}(E_{G_i}(v))| =2.$
\item[iii)] $w_i(v) = 2$ if $d_{G_i}(v) =1,$ or $d_{G_i}(v) =2$ and $|c^{-1}(E_{G_i}(v))| =1.$ 
\end{itemize}

For $i \in [n-1]$ and $j \in \{ 0,1,2\}$ let $V_{ij} = \{ v \in V(G_i)\ \big| \ w_i(v) = j \} .$  When $i = n-1,$ we let $V_j = V_{(n-1)j}.$  For $i\in [n-1]$ and $j \in [2],$  let $w(G_i) = \sum_{v\in V(G_i)} w_i(v)$ and let $V_{i2}^j = \{ v\in V_{i2}\ \big| \ d_{G_i}(v) = j \}.$   For $j\in \{ 1,2 \},$ let $V_2^j = V_{(n-1)2}^j .$

For all $i\in [n-1],$  let $\epsilon_i(v) = \left\{ \begin{array}{lr} 0 & \mathrm{if}\ v\in V_{i0}\\ 1 & \mathrm{otherwise} \end{array} \right.$.
%

\begin{noname}
If $\epsilon_{n-1}(v) =1$ then $v \in V(T).$\label{nona2}
\end{noname}

\begin{proof}
Suppose that for some vertex $v,$ $v \not\in V(T)$ and $\epsilon_{n-1}(v) =1.$ 

\ssms
\noindent (A.1) $d_G(v) \ge 2.$
\ssms

\begin{proof}
Suppose that $d_G(v) =1.$  Let $a$ be the edge incident with $v.$  Then by construction, $a\in A.$  Let $H' = G'/a$  By assumption, the lemma holds for $H'$ and hence it contains a near $T$-SRCP
which is also seen to be a near $T$-SRCP in $G'$, a contradiction.
It follows that $d_G(v) \ge 2.$
\end{proof}

\ssms
\noindent (A.2) $d_G(v) = 3.$
\ssms

\begin{proof}
 By (A.1), $d_G(v) \ge 2.$  Suppose $d_G(v) = 2.$  Let $E_G(v) = \{ e,f \}.$  If $c(e) =c(f)$, then by assumption, the lemma holds for $H' = G' -v$ and any near-$T$-SRCP in $H'$ would also be a near-$T$-SRCP in $G'.$  Thus
$c(e) \ne c(f)$.  By construction, at least one of $e$ or $f$ is an edge in $A.$  We may assume that $e\in A.$  Suppose first that $e,f$ do not belong to a triangle.
Then $H' = G'/e$ is simple and is cycle-achromatic. 
Hence the lemma holds for $H'$ and it contains a near-$T$-SRCP $\{ C_1', C_2' \}$  (where $|C_1'| + |C_2'| \le \nu(H) + 2 = n+1$).
 If $f \not\in E(C_1') \cup E(C_2'),$ then $\{ C_1', C_2' \}$ is seen to be a near-$T$-SRCP in $G'$.  Thus we may assume that $f\in E(C_1').$  Then the circuit $C_1$ in $G'$ where $E(C_1) = E(C_1') + e$ is a rainbow cycle.  Furthermore, $C_1$ and $C_{2}'$, are rainbow cycles such that $|C_1| + |C_2'| \le n+2.$   Thus $\{ C_1, C_2' \}$ is seen to be a near-$T$-SRCP in $G'.$  We conclude that $e$ and $f$ belong to a triangle $S$ where for some edge $g$,  $E(S) = \{ e,f,g \}.$   Given that $e\in A,$ we must have that $c(f) = c(g).$
Thus $H = G/f\backslash g$ is simple and cycle-achromatic. Let $H' = G'/f\backslash g.$ If $A \cup T$ induces a $K_4$ in $H',$ then it is straightforward to find a near $T$-SRCP in $G'.$  Thus we may assume that such subgraph does not exist.
 Now the conditions of lemma are seen to hold for $H$ and $H'$ and thus $H'$ contains a near-$T$-SRCP $\{ C_1',C_2' \}$  (where $|C_1'| + |C_2'| \le \nu(H) + 2 = n+1$). 
 If $e\not\in E(C_{1}') \cup E(C_{2}'),$  then $\{ C_1', C_2' \}$ is seen to be a near-$T$-SRCT in $G'$.  Thus we may assume that $e\in C_{1}'.$  Then the cycle $C_{1}$ in $G'$ where $E(C_{1})  = E(C_{1}') + f$ is a rainbow cycle in $G'$ and $\{ C_1, C_{i_2}' \}$ is seen to be a near-$T$-SRCP in $G'$.  We conclude that $d_G(v) \ge 3$ and hence $d_G(v) =3$ (since $\epsilon_{n-1}(v) =1$).
 \end{proof}  
 
 From (A.2), we have that $d_G(v) =3$.  It follows that for some $i\in [k+1,n-1]$ and edge $g,$ 
$E_G(v) = A_i + g.$ For some $j\in \{ 1,2 \}$, $a_{ij}$ and $g$ do not belong to a triangle.
We may assume that $a_{i1}$ and $g$ do not belong to a triangle.  Let $H = G/a_{i1}\backslash a_{i2}$ and $H' = G'/a_{i1}\backslash a_{i2}.$ Then $H$ is simple, cycle - achromatic and $|c(H)| = \nu(H) - 1 = n-2.$  We may assume that $A \cup T$ does not induce a subgraph isomorphic to $K_4$ in $H';$ for if such is the case, then $g\in A$ and hence $a_{i2}$ and $g$ belong to no triangle of $G.$  In this case, one can redefine $H$ and $H'$ to be $H = G/a_{i2}\backslash a_{i1}$ and $H' = G'/a_{i2}\backslash a_{i1}$.  Then $A \cup T$ does not induce a $K_4$ in $H'.$  By the above, $H'$ has a near-$T$-SRCP $\{ C_1', C_2' \}$.
If $g \not\in E(C_1') \cup E(C_2')$, then $\{ C_1', C_2' \}$ is seen to be a near-$T$-SRCP in $G'$. Thus we may assume that $g \in E(C_1')$. Then the cycle $C_1$ in $G$ where $E(C_1)  = E(C_1') + a_{i1}$ is a rainbow cycle in $G'$ and $\{ C_1, C_2' \}$ is seen to be a near-$T$-SRCPin $G'.$  We conclude that $v\in V(T).$ 
\end{proof}

\begin{noname}
$w(G) \le 6.$\label{nona3}
\end{noname}

\begin{proof}
By (\ref{nona2}), there are at most $3$ vertices $v$ for which $\epsilon_{n-1} = 1.$  Given that for all $v$, $w(v) \le 2,$ it follows that $w(G) \le 6.$
\end{proof}

\begin{noname}
No pair of vertices in $V_{2}^2$ have a common neighbour. 
\label{newnoname4}
\end{noname}

\begin{proof}
Suppose to the contrary that two vertices in $V_2^2$ share a common neighbour $u.$  By (\ref{nona2}), we have that $V_2^2 \subset V(T) = \{ x_1, x_2, x_3 \}.$  Without loss of generality, we may assume that $\{ x_1, x_2 \} \subseteq V_2^2$ where
$x_1 = v_i = v_{i1}=v_{i2},$ $x_2 = v_j = v_{j1} = v_{j2}$ and $u = u_{i1} = u_{j1}.$  Suppose $x_3 = u.$  Then letting $C_1 = x_1a_{i1} x_3 e_2 x_1$ and $C_2 = x_2 a_{j1} x_3 e_3 x_2,$ $\{ C_1, C_2 \}$ is seen to be a near-$T$-SRCP in $G'.$
Thus $x_3 \ne u.$  Suppose $x_3 = u_{i2}.$  Then letting  $C_1 = x_1a_{i2}x_3e_2x_1$ and $C_2 = x_1 a_{i1} u a_{j1} x_2 e_1 x_1,$ $\{ C_1, C_2 \}$, is seen to be edge-disjoint pair of rainbow cycles
$|C_1| + |C_2| = 5 \le n+2.$  Thus $\{ C_1, C_2 \}$ is seen to be a near-$T$-SRCP in $G'$ and hence $x_3 \not\in \{ u_{i1}, u_{i2} \}.$  By symmetry, $x_3 \not\in \{ u_{j1}, u_{j2} \}.$  This also implies that $n\ge 5.$
 There is a rainbow path $P$ in $G - x_1 - x_2$ between $u_{i2}$ and $x_3$ having length at most $n-3$. Let $P' = P + a_{i2} + x_1$ and let $C_1 = P' + e_2$.  Let $C_2$ be the rainbow cycle $C_2 = x_1a_{i1}ua_{j1}x_2 e_2 x_1$.  
Then $C_1, C_2$ are edge-disjoint rainbow cycles in $G'$ where $|C_1| + |C_2| \le n-1 + 3 = n+2,$ implying that $\{ C_1, C_2 \}$ is a near-$T$-SRCP in $G'$, a contradiction.  It follows that no pair of vertices in $V_2^2$ have a common neighbour. 
\end{proof}

\begin{noname}
$|V_{2}^2| \ge 1.$
\label{newnoname5}
\end{noname}

\begin{proof}
Suppose to the contrary that $V_{2}^2 = \emptyset.$ Then $G - A_{n-1} = G_{n-2}$ consists of two non-trivial components $G_{n-2}^1$ and $G_{n-2}^2$.  
We may assume that for $j=1,2,$ $u_{(n-1)j} \in V(G_{n-2}^1)$ and $v_{(n-1)j} \in V(G_{n-2}^2).$   We may also assume that $\{ x_1, x_2 \} \subseteq V(G_{n-2}^1).$   
If $x_3 \in V(G_{n-2}^1),$ then by assumption, $G_{n-2}^1 + T$ contains a near-$T$-SRCP  which is also a near-$T$-SRCP in $G'$.   
Thus $x_3 \in V(G_{n-2}^2).$  By the construction of $G$, each component $G_{n-2}^j,\ j = 1,2$ must contain at least one edge of $A.$  We may assume that for $j=1,2,$ $a_j \in E(G_{n-2}^j).$

\sms
\noindent (D.1) $\varepsilon(G_{n-2}^1) \ge 2.$
\sms

\begin{proof}  
Suppose $\varepsilon(G_{n-2}^1) =1.$  Then $E(G_{n-2}^1) = \{ a_1 \}$ and 
$a_1 = x_1x_2$.  There is a rainbow path $P$ from $x_3$ to one of the vertices $v_{(n-1)1}, v_{(n-1)2}$ (which might be the trivial path) where $|P| \le \nu(G_{n-2}^2) -1 = n-3.$  We may assume that $P$ terminates at $v_{(n-1)1}$.  We can extend $P$ to a path $P'$ where $P' = P + a_{(n-1)1}+ u_{(n-1)1}.$  Here we may assume that $u_{(n-1)1} = x_1.$
Let $C_1$ be the rainbow cycle where $C_1 = x_1 a_1 x_2 e_1 x_1$ and let $C_2$ be the rainbow cycle where $C_2 =P' + e_2.$ Then $C_1$ and $C_2$ are edge-disjoint rainbow cycles and $|C_1| + |C_2| \le 2 + (n-3) +2 = n+1.$  Consequently, $\{C_1, C_2 \}$  is a near-$T$-SRCP, contradicting our assumptions.  Thus $\varepsilon(G_{n-2}^1) \ge 2.$
\end{proof}

\sms
\noindent (D.2) $\varepsilon(G_{n-2}^1) \ge 3.$
\sms

\begin{proof}
By (D.1), $\varepsilon (G_{n-2}^1) \ge 2.$  Suppose $\varepsilon (G_{n-2}^1) = 2.$  Then $E(G_{n-2}^1) \subseteq A$ and hence $E(G_{n-2}^1) = \{ a_1, a_3 \}$ (and $a_2 \in E(G_{n-2}^2)$).  Then $G_{n-2}^1$ is a path $P = w_1a_1w_2a_3w_3.$  It is seen that for $i \in \{ 1,3 \},$ $\epsilon(w_i) =1.$  By (\ref{nona2}), $\{ w_1, w_3 \} \subseteq V(T)$ and we can assume that $w_1 = x_1$ and $w_3 = x_2$.  Since $x_3 \not\in V(G_{n-2}^1),$ it follows that $\epsilon (w_2) = 0$.  
It follows that at least one of $a_{(n-1)1}$ or $a_{(n-1)2}$ is incident with $w_2.$  Without loss of generality, we may assume that $a_{(n-1)1}$ is incident with $w_2;$ that is, $u_{(n-1)1} = w_2.$  

Suppose $u_{(n-1)2} = u_{(n-1)1} = w_2.$  Then $v_{(n-1)1} \ne v_{(n-1)2}.$  By Lemma \ref{lem2}, there exist edge-disjoint rainbow paths $P_1, P_2$ in $G_{n-2}^2$ where for $i=1,2,$ $P_i$ joins $x_3$ to $v_{(n-1)i}$ and $|P_1| + |P_2| \le \nu(G_{n-2}^2) -1 = n-4.$  Now extend each of the paths $P_i,\ i = 1,2$ to path $P_i',\ i = 1,2$ where $P_1' = P_1 + a_{(n-1)1} + w_2 + a_1 + x_1$ and $P_2' = P_2 + a_{(n-1)2} + w_2 + a_3 + x_2.$  Let $C_1$ and $C_2$ be cycles where $C_1 = P_1' + e_3$ and $C_2 = P_2' + e_3$.  Then $C_1$ and $C_2$ are seen to be edge-disjoint rainbow cycles where $|C_1| + |C_2| = |P_1| + |P_2| + 6 \le (n-4) + 6 = n+2.$  Thus $\{ C_1, C_2 \}$ is a near-$T$-SRCP in $G.$ 

From the above, we have that $u_{(n-1)2} \ne u_{(n-1)1} $ and hence $u_{(n-1)2} \in \{ x_1, x_2 \}$.  Without loss of generality, we may assume that $u_{(n-1)2} = x_2.$ By Lemma \ref{lem1}, there is rainbow path $P_2'$ in $G_{n-2}^2$ from $x_3$ to $v_{(n-1)2}$ of length at most $\lfloor \frac {\nu(G_{n-2}^2)}2 \rfloor  = \lfloor \frac {n-3}2 \rfloor \le n-4.$  Let $P_1 = P$ and let $P_2 = P_2' + a_{(n-1)2} + x_2$ and let $C_1 = P_1 + e_1$ and $C_2 = P_2 + e_3.$  Then $|C_1| + |C_2| \le 3 + n-2 = n+1$ and $\{ C_1, C_2 \}$ is seen to be a near-$T$-SRCP, a contradiction.   Thus $\varepsilon (G_{n-2}^1) \ge 3.$ 
%
\end{proof}

Let $G_{n-2}^{1'} = G_{n-2}^1 + a_{(n-1)1} + v_{(n-1)1}.$  Let $x_3' = v_{(n-1)1}$ and $a_2' = a_{(n-1)1.}$   Let $T'$ be the triangle having vertices $x_1,x_2, x_3'$ and $e_1' = x_1x_2,\ e_2' = x_1x_3', e_3' = x_2x_3'.$  Suppose that that the edges $A- a_2 + a_2'$ together with the edges of $E(T')$ do not induce a $K_4.$  Let $G_{n-2}^{1''} = G_{n-2}^{1'} + E(T').$  Then by assumption, $G_{n-2}^{1''}$ has a near-$T'$-SRCP $\{ C_1, C_2 \}.$
 Exactly one of these cycles must contain $a_2',$ say $C_1.$  We may assume that $e_2' \in C_1.$  Let $P_1 = C_1 - e_2'.$  There is a rainbow path $P_2$ from $x_3$ to $x_3'$ where $|P_2| \le \nu(G_{n-2}^2) -1.$  Let $C_1'$ be the rainbow cycle where $E(C_1') = (E(P_1) \cup E(P_2)) + e_2$  Now $C_1'$ and $C_2$ are edge-disjoint rainbow cycles where $|C_1'| + |C_2| = |C_1| + |C_2| + |P_2| \le  \nu(G_{n-2}^{1'}) + 2 + \nu(G_{n-2}^2) -1 = n+2.$  Thus $\{ C_1, C_2\}$ is a near-$T$-SRCP.

By the above, the edges of $A - a_2 + a_2' \cup E(T')$ induce a $K_4$.  Likewise, the same is true if $a_{(n-1)1}$ is replaced by $a_{(n-1)2}.$   Then $u_{(n-1)1} = u_{(n-1)2} = u$ and $a_1, a_3$ are incident with $u$ and we may assume that $a_1 = x_1u$ and $a_3 = x_2 u.$  Let $G_{n-2}^{2'} = G_{n-2}^2 + u + a_{(n-1)1} + a_{(n-1)2}.$   By Lemma \ref{newlem3} there are two edge-disjoint rainbow paths $P_1$, $P_2$ from $x_3$ to $u$ where $|P_1| + |P_2| \le \nu(G_{n-2}^{2'}) = \nu(G_{n-2}^2) + 1.$  Extend $P_1$ to the path $P_1'$ where $P_1' = P_1 + a_1 + x_1,$ and extend $P_2$ to the path $P_2'$ where $P_2' = P_2 + a_3 + x_2.$  Let $C_1,C_2$ be the edge-disjoint rainbow cycles in $G'$ where $C_1 = P_1' + e_2$ and $C_2 = E(P_2') + e_3.$  Then $|C_1| + |C_2| \le \nu(G_{n-1}^2) + 1 + 4 \le n+2.$ Thus $\{ C_1, C_2 \}$ is a near-$T$-SRCP, a contradiction.
Thus $V_{2}^2 \ne \emptyset.$ 
\end{proof}

\begin{noname}
$|V_{2}^2| \ge 2.$\label{newnoname6}
\end{noname}

\begin{proof}
From (\ref{newnoname5}),  we have that $|V_{2}^2| \ge 1.$  Suppose that  $|V_{2}^2| = 1.$  By (\ref{nona2}), we may assume that $V_2^2 = \{ x_1 \}$ and $v_{(n-1)1} = v_{(n-1)2} =  x_1.$

\ssms
\noindent (E.1) $V_{(n-2)2}^2 = \emptyset.$

\begin{proof}
Suppose there exists $v \in V_{(n-2)2}^2.$  Given that $v\not\in V_{2}^2$ it follows that $v$ is adjacent to $x_1$ and we may assume that $a_{(n-1)1} = x_1v.$  We observe that $v \in V_{1}$ (and hence $w(v) > 0$).  Thus by (\ref{nona2}), $v \in V(T)$ and we may assume that $x_2 = v.$  Let $C_1$ be the rainbow cycle $C_1 = x_1 a_{(n-1)1} x_2 e_1 x_1.$  Noting that $G_{n-2}$ is connected
there is a rainbow path between $x_2$ and $x_3$ in $G_{n-2},$ say $P$, where $|P| \le \nu(G_{n-2}) -1 = n-2.$  Let $C_2 = P + e_3.$   Now $C_i, \ i = 1,2$ are edge-disjoint rainbow cycles where $|C_1| + |C_2| \le 2 + n-1 = n+1$ and thus $C_i,\ i = 1,2$ 
is a near-$T$-SRCP, a contradiction.  Thus $V_{(n-2)2}^2 = \emptyset.$
\end{proof}
   
By (E.1), $V_{(n-2)2}^2 = \emptyset$ and this in turn implies that $G_{n-2} - A_{n-2} = G_{n-3}$ has exactly two non-trivial components, say $G_{n-3}^1$ and $G_{n-3}^2.$  We may assume that $\{ u_{(n-2)1}, u_{(n-2)2} \} \subseteq V(G_{n-3}^1)$ and 
$\{ v_{(n-2)1}, v_{(n-2)2} \} \subseteq V(G_{n-3}^2).$ 
We may also assume that $x_2 \in V(G_{n-3}^1).$  We observe that each of the components $G_{n-3}^j,\ j = 1,2$ must contain at least one edge of $A.$  Without loss of generality, we may assume that for $j= 1,2,$ $a_j \in E(G_{n-3}^j).$

\ssms
\noindent (E.2) $x_3 \in V(G_{n-3}^1).$

\begin{proof}
Suppose $x_3 \in V(G_{n-3}^2).$  Assume first that $u_{(n-1)1}$ and $u_{(n-1)2}$ belong to different components of $G_{n-3}.$ Without loss of generality, we may assume that $u_{(n-1)1} \in V(G_{n-3}^1)$ and $u_{(n-1)2} \in V(G_{n-3}^2).$  Let $P_1$ be a rainbow path from $x_2$ to $u_{(n-1)1}$ in $G_{n-3}^1$ and let $P_2$ be a rainbow path from $x_3$ to $u_{(n-1)2}$ in $G_{n-3}.$  Let $P_1' = P_1 + a_{(n-1)1} + x_1$ and let $P_2' = P_2 + a_{(n-1)2} + x_1.$  Let $C_1 = P_1' + e_1$ and $C_2 = P_2' + e_2.$  Then $C_i, \ i = 1,2$ are edge-disjoint rainbow cycles where
$|C_1| + |C_2| \le \nu(G_{n-3}^1) + \nu(G_{n-3}^2) + 2 = n+1$ and hence $\{ C_1, C_2 \}$  is a near-$T$-SRCP in $G'$.  Thus $u_{(n-1)1}, u_{(n-1)2}$ belong to the same component (either $G_{n-3}^1$ or $G_{n-3}^2.$).  Without loss of generality, we may assume that 
$\{ u_{(n-1)1}, u_{(n-1)2} \} \subseteq V(G_{n-3}^1).$  Let $x_3' = v_{(n-2)1}$, $a_2' = a_{(n-2)1}$, and let $G_{n-3}^{1'} = G_{n-3}^1 + a_{(n-1)1} + a_{(n-1)2} + x_1 + a_2' + x_3'.$  Let $T'$ be the triangle with vertices $x_1,x_2,x_3'$ and let $e_1' = x_1x_2,\ e_2' = x_1x_3',\ e_2' = x_2 x_3'.$  Clearly, the edges of $A-a_2 + a_2' \cup E(T')$ do not induce a $K_4$ in $G_{n-3}^{1'}.$  Thus by assumption, there is a near-$T'$-SRCP $\{ C_1, C_2 \}$ for $G_{n-1}^{1'}.$ 
Exactly one of these cycles contains $a_2'$, and we may assume $a_2' \in E(C_2).$  Without loss of generality, we may assume that $e_2' \in E(C_2).$  Let $P$ be a rainbow path in $G_{n-3}^2$ from $x_3$ to $x_3'$.  Then $|P| \le \nu(G_{n-3}^2) -1.$   Let $C_2'$ be the rainbow cycle where $E(C_2') = E(C_2) - e_2' + E(P) + e_2.$  Then $|C_2'| = |C_2| + |P| \le |C_2| + \nu(G_{n-3}^2) -1.$  Thus we have that 
$$|C_1| + |C_2'| \le  |C_1| + |C_2| + \nu(G_{n-3}^2) -1 \le \nu(G_{n-3}^{1'}) + 2 +\nu(G_{n-3}^2) -1 = n+2.$$  Thus $\{C_1, C_2' \}$ is a near-$T$-SRCP, a contradiction.  Thus $x_3 \in V(G_{n-3}^1).$
\end{proof}

\ssms
\noindent (E.3) For some $i\in \{ 1,2 \},$ $\{ u_{(n-1)1}, u_{(n-1)2} \} \subseteq V(G_{n-3}^i).$

\begin{proof} By (E.2), $x_3 \in V(G_{n-3}^1).$  Assume that $u_{(n-1)1}$ and $u_{(n-1)2}$ belong to different components of $G_{n-3}.$   Without loss of generality, we may assume that $u_{(n-1)1} \in V(G_{n-3}^1)$ and $u_{(n-1)2} \in V(G_{n-3}^2).$  Let $G_{n-3}^{1'} = G_{n-3}^1 + x_1 + a_{(n-1)1}.$  Let $a_2' = a_{(n-1)1}.$  Suppose first that the edges of $(A - a_2 + a_2') \cup E(T)$ do not induce a $K_4$.  Then by assumption there is a near-$T$-SRCP $\{ C_1, C_2 \}$ for $G_{n-3}^{1'}$ and such a pair is seen to be a near-$T$-SRCP for  $G'.$   Thus 
 the edges of $(A - a_2 + a_2') \cup E(T)$ induce a $K_4$. We may assume that $a_1 = u_{(n-1)1}x_{2}$ and $a_3 = u_{(n-1)1}x_3.$ 
There is a rainbow path $P$ in $G$ starting at $x_1$ and containing $a_{(n-1)2}$ which terminates at one of the vertices $u_{(n-1)1}, x_2, x_3$ and we may assume that $P$ is a shortest such path.  Suppose that $P$ terminates at $u_{(n-1)1}$ (and does not contain $x_2$ or  $x_3$).  Then $|P| \le n-3.$ Let $P_1 = P + a_1 + x_2$ and let $P_2 = x_1 a_2' u_{(n-1)1} a_3 x_3$.  For $i=1,2,$ let $C_i = P_i + e_i.$ Then $\{ C_1, C_2\}$ are edge-disjoint rainbow cycles where 
$|C_1| + |C_2| \le (n-2) + 1 + 3 = n+2.$  Thus $\{ C_1, C_2 \}$ is a near-$T$-SRCP.  If $P$ terminates at $x_2$ or $x_3$, then a similar argument can be used.  Thus $u_{(n-1)1}$ and $u_{(n-1)2}$ belong to the same component of $G_{n-3}.$  
\end{proof}

\ssms
\noindent (E.4) $\{ u_{(n-1)1}, u_{(n-1)2} \} \subseteq V(G_{n-3}^2).$

\begin{proof} By (E.3), there exists $i\in \{ 1,2 \}$ such that $\{ u_{(n-1)1}, u_{(n-1)2} \} \subseteq V(G_{n-3}^i).$  
Suppose $\{ u_{(n-1)1}, u_{(n-1)2} \} \subseteq V(G_{n-3}^1).$  Let $G_{n-3}^{1'} = G_{n-3}^1 + x_1 + A_{n-1}.$  By assumption, there exists a near-$T$-SRCP for $G_{n-3}^{1'}$ and such is seen to be a near-$T$-SRCP for $G'$ as well.  
Thus $\{ u_{(n-1)1}, u_{(n-1)2} \} \subseteq V(G_{n-3}^2).$
\end{proof}

By (E.4), $\{ u_{(n-1)1}, u_{(n-1)2} \} \subseteq V(G_{n-3}^2).$
Let $x_1' = v_{(n-2)1}$ and let $a_2' = a_{(n-2)1}.$  Let $G_{n-3}^{1'} = G_{n-3}^1 + a_2' + x_1'.$  Let $T'$ be the triangle having vertices $x_1', x_2, x_3$ having edges $e_1' = x_1'x_2, e_2' = x_1'x_3$ and $e_3' = x_2x_3.$  
Assume that the edges of $(A - a_2 + a_2') \cup E(T')$ do not induce a $K_4.$  Then by assumption,
 $G_{n-3}^{1'}$ has a near-$T'$-SRCP $\{ C_1, C_2 \}$ where $|C_1| + |C_2| \le \nu(G_{n-3}^{1'}) +2) = \nu(G_{n-3}^1) +3$.
 Exactly one of these cycles contains $a_2'$, and we may assume that $a_2' \in E(C_1)$  Furthermore, we may assume $e_1' \in C_1.$ Let $P$ be a rainbow path from $x_1'$ to $x_1$ in $G - A_2.$  Then $|P| \le \nu(G_{n-3}^2).$  
 Let $C_1'$ be the cycle where $E(C_1') = E(C_1) \cup E(P) - e_1' +  e_1.$  Then $C_1'$ and $C_2$ are edge-disjoint rainbow cycles where
$|C_1'| + |C_2| \le \nu(G_{n-3}^1) + 3 + \nu(G_{n-3}^2) = n+ 2.$  Thus $\{ C_1', C_2 \}$ is a near-$T$-SRCP for $G.$
It follows that the edges of $(A - a_2 + a_2') \cup E(T')$ induce a $K_4$, and the same is true if we replace $a_{(n-1)1}$ with $a_{(n-1)2}.$
Thus $u = u_{(n-1)1} = u_{(n-1)2}$ and $a_1, a_3$ are both incident with $u$ (in $G_{n-3}^1$). We may assume $a_1 = x_2u,\ a_3 = x_3u.$  However, as in a previous case, letting $G_{n-3}^{2'} = G_{n-3}^2 + u + x_1 + A_{n-2} + A_{n-1},$ there are two edge-disjoint rainbow paths $P_1, P_2$ in $G_{n-3}^{2'}$ from $x_1$ to $u$ where $|P_1| + |P_2| \le \nu(G_{n-3}^2) + 2.$  The paths $P_1$ and $P_2$ can now be extended to paths $P_1', P_2'$ where $P_1' = P_1 + a_1 + x_2$ and $P_2' = P_2 + a_3 + x_3.$  For $i = 1,2,$ let $C_i = P_i' + e_i.$ Then $C_i,\ i = 1,2$ are edge-disjoint rainbow cycles where $$|C_1| + |C_2| = |P_1'| + |P_2'|+2 = |P_1| + |P_2| +4 \le \nu(G_{n-3}^2) + 6 \le n+2,$$ implying that $\{ C_1, C_2 \}$is a near-$T$-SRCP, a contradiction.  Thus $|V_{(n-1)2}^2| \ge 2$.
\end{proof}

\begin{noname}
$|V_{2}^2| =3.$
\label{nona5}
\end{noname}

\begin{proof}
By (\ref{newnoname6}), $|V_{2}^2| \ge 2.$ Suppose that $|V_{2}^2|=2$.  Since by (\ref{nona2}), $V_{2}^2 \subset V(T),$ we may assume that $V_{2}^2 = \{ x_1, x_2 \}$ and $x_1 = v_{(n-1)1} = v_{(n-1)2}$ and $x_2 = v_{(n-2)1} = v_{(n-2)2}.$  By (\ref{newnoname4}), $x_1$ and $x_2$ have no common neighbour.

\sms
\noindent (F.1) $V_{(n-3)2}^2 = \emptyset.$

\begin{proof}
Suppose there exists $u\in V_{(n-3)2}^2.$  Then at exactly one of $x_1$ or $x_2$ is adjacent to $u$ (because $u \not\in V_2^2$).  We may assume that $u = u_{(n-1)1}.$  We see that $d_G(u) = 3$ and hence $w_{n-1}(u) =1.$  It follows that $u\in V(T)$ and $u = x_3.$  There is a rainbow path $P$ in $G - x_1$ from $x_3$ to $x_2$ where $|P| \le n-2$.  Let 
$C_1$ be the cycle in $G'$ where $C_1 =  x_1 a_{(n-1)1} x_3 e_2 x_1$ and let $C_2$ be the cycle where $C_2 = P + e_3$.  Then $C_i,\ i=1,2$ are edge-disjoint rainbow cycles where $|C_1| + |C_2| \le 2 + (n-2) +1 = n+1.$  
Thus $\{ C_1, C_2 \}$ is a near-$T$-SRCP.  Thus $V_{(n-3)2}^2 = \emptyset.$
\end{proof}

By (F.1), $V_{(n-3)2}^2 = \emptyset$ and thus $G_{n-4}$ has two non-trivial components  $G_{n-4}^1$ and $G_{n-4}^2$ where 
we may assume that $\{ u_{(n-3)1}, u_{(n-3)2} \} \subseteq V(G_{n-4}^1)$ and $\{ v_{(n-3)1}, v_{(n-3)2} \} \subseteq V(G_{n-4}^2).$  We may also assume that $x_3 \in V(G_{n-4}^1).$  By construction, each of $G_{n-4}^j,\ j = 1,2$ contains at least one edge of $A$ and we may assume that for $j=1,2,$ $a_j \in E(G_{n-4}^j).$  

\ssms
\noindent (F.2)  There exists $i\in \{ n-2, n-1 \}$ and $j\in \{ 1,2 \}$ such that $\{ u_{i1}, u_{i2} \} \subseteq V(G_{n-4}^j).$

\begin{proof}
Suppose that for all $i \in \{ n-1, n-2 \}$ and $j\in \{ 1,2 \},$ $\{ u_{i1}, u_{i2} \} \not\subseteq V(G_{n-4}^j).$  We may assume that for all  $i \in \{ n-1, n-2 \}$ and $j\in \{ 1,2 \},$ $u_{ij} \in V(G_{n-4}^j).$ Let $P_1$ be a rainbow path in $G_{n-4}^1$ from $x_3$ to $u_{(n-1)1}$ and let $P_2$ be a rainbow path in $G_{n-4}^2$ from $u_{(n-1)2}$ to $u_{(n-2)2}.$  Now extend $P_1$ to a path $P_1' = P_1 + a_{(n-1)1} + x_1$ and let $C_1 = P_1' + e_2.$  Let $C_2$ be the cycle where $E(C_2) = E(P_2) + a_{(n-1)2} + a_{(n-2)2} + e_1.$  Then $|C_1| + |C_2| \le \nu(G_{n-4}^1) + 1 + \nu(G_{n-4}^2) + 2 = n+1.$  Thus $\{ C_1, C_2 \}$ is a near-$T$-SRCP, a contradiction.
\end{proof}

By (F.2), we may assume that for some $j\in \{ 1,2 \},$  $\{ u_{(n-1)1}, u_{(n-1)2} \} \subseteq V(G_{n-4}^j).$

\sms
\noindent (F.3) $\{ u_{(n-1)1}, u_{(n-1)2} \} \subseteq V(G_{n-4}^2).$

\begin{proof}
Suppose that $\{ u_{(n-1)1}, u_{(n-1)2} \} \subseteq V(G_{n-4}^1).$
If $\{ u_{(n-2)1}, u_{(n-2)2} \} \subseteq V(G_{n-4}^1),$ then by assumption, $G_{n-4}^{1'} = G_{n-4}^1 +  x_1 + x_2 + A_{n-1} + A_{n-2}$ has a near-$T$-SRCP which is seen to be a near-$T$-SRCP for $G'$ as well.  Thus $\{ u_{(n-2)1}, u_{(n-2)2} \} \cap V(G_{n-4}^2) \ne \emptyset$ and we may assume that $u_{(n-2)1} \in V(G_{n-4}^2).$  Let $G_{n-4}^{1'} = G_{n-4}^1 + x_1 + A_{n-1} + a_{(n-3)1} + v_{(n-3)1}.$ 
 Let $x_2' = v_{(n-3)1}$ and let $a_2' = a_{(n-3)1}.$  Let $T'$ be the triangle with vertices $x_1,x_2', x_3$ and edges $e_1' = x_1x_2'$, $e_2' = x_1x_3$ and $e_3' = x_2'x_3.$  
 Clearly the edges of $(A - a_2 + a_2') \cup E(T')$ do not induce a $K_4$ in $G_{n-4}^{1'}.$  By assumption, $G_{n-4}^{1'}$ has a near-$T'$-SRCP $C_i,\ i = 1,2.$ 
There is a rainbow path $P$ from $x_2'$ to $x_2$ in $G_{n-4}^2 + a_{(n-2)1} + x_2$ where $|P| \le \nu(G_{n-4}^2).$  
Exactly one of the cycles $C_i, \ i = 1,2$ must contain $a_2'$ and we may assume that $a_2' \in E(C_1).$  
Furthermore, we may assume that $e_1' = x_1x_2' \in E(C_1)$ (the ensuing argument being very similar if $e_3' \in E(C_1)$).  Let $C_1'$ be the cycle where $C_1' = (C_1 - e_1') + P + e_1.$  Then 
$$|C_1'| + |C_2| \le  |C_1| + |C_2| + |P| \le \nu(G_{n-4}^1) + 4 + \nu(G_{n-4}^2)  = n +2$$ and hence $\{ C_1',C_2 \}$ is a near-$T$-SRCP, a contradiction.  Thus\\ $\{ u_{(n-1)1}, u_{(n-1)2} \} \subseteq V(G_{n-4}^2).$
\end{proof}

By (F.3), $\{ u_{(n-1)1}, u_{(n-1)2} \} \subseteq V(G_{n-4}^2).$ If $\{ u_{(n-2)1}, u_{(n-2)2} \} \subseteq V(G_{n-4}^1),$ then we could argue as we did above with $x_2$ in place of $x_1.$  Thus we may assume that $u_{(n-2)1} \in V(G_{n-4}^2).$

\sms
\noindent (F.4)  $u_{(n-2)2} \in V(G_{n-4}^1).$

\begin{proof}
Suppose that $u_{(n-2)2} \in V(G_{n-4}^2).$  Let $G_{n-4}^{2'} = G_{n-4}^2 + x_1 + x_2 + A_{n-2} + A_{n-1} + u_{(n-3)1} + a_{(n-3)1}.$  Let $x_3' = u_{(n-3)1}$ and and let $a_1' = a_{(n-3)1}.$  Let $T'$ be the triangle with vertices $x_1, x_2, x_3'$ and edges $e_1' = x_1x_2,$ $e_2' = x_1x_3',\ e_3' = x_2x_3'.$  Clearly the edges of $(A - a_1 + a_1') \cup E(T')$ do not induce a $K_4$.  By assumption, $G_{n-4}^{2'} + E(T')$ has a near-$T'$-SRCP $\{ C_1, C_2 \}.$
For some $i$, $a_1'\in E(C_i)$ and we may assume this holds for $i=1$.  Furthermore, we may assume that $e_2' \in E(C_1).$  Let $P$ be a rainbow path in $G_{n-4}^1$ from $x_3$ to $x_3'$, where $|P| \le \nu(G_{n-4}^1)-1.$  Let $C_1'$ be the cycle where 
$E(C_1') = E(C_1) - e_2' +E(P) + e_2.$  Then $C_1'$ and $C_2$ are edge-disjoint rainbow cycles where 
$$|C_1'| + |C_2| \le \nu(G_{n-4}^{2'}) + 2 + \nu(G_{n-4}^1) -1 =  \nu(G_{n-4}^{2}) + 5 + \nu(G_{n-4}^1) -1 = n+2.$$  Thus $\{ C_1',C_2 \}$ is a near-$T$-SRCP, a contradiction.  It follows that $u_{(n-2)2} \in V(G_{n-4}^1).$
\end{proof}

By (F.4), $u_{(n-2)2} \in V(G_{n-4}^1).$  Let $P_1$ be a rainbow path between $x_3$ and 
$u_{(n-2)2}$ in $G_{n-4}^1$ and let $P_2$ be a rainbow path between $u_{(n-1)1}$ and $u_{(n-2)1}$ in $G_{n-4}^2.$  Let $C_1$ be the cycle where $E(C_1) = E(P_1) + a_{(n-2)2} + e_3$ and let $C_2$ be the cycle where $E(C_2) = E(P_2) + a_{(n-1)1}  + e_1  + a_{(n-2)1}.$
Then $C_1, C_2$ are edge-disjoint rainbow cycles where $|C_1| + |C_2| \le \nu(G_{n-4}^1) -1 + 2 + \nu(G_{n-4}^2) -1 + 3 = n+1.$  Thus $\{ C_1, C_2 \},$ is a near-$T$-SRCP, a contradiction.
It follows that $|V_2^2| = 3.$ 
\end{proof}

By (\ref{nona2}) and (\ref{nona5}), it follows that $V_{2}^2 = V(T).$  We may assume that for $i=1,2,3,$ $x_i = v_{(n-i)1} = v_{(n-i)2}.$  By (\ref{newnoname4}), no two vertices in $V(T)$ have a common neighbour.  Suppose there exists $u\in V_{(n-4)2}^2.$  Since $u \not\in V(T),$ it follows that $w(u) =0$ and hence $u$ is adjacent to at least two vertices in $V(T)$, a contradiction.  
Thus we have that $V_{(n-4)2}^2 = \emptyset$ and hence $G_{n-5}$ has exactly two non-trivial components $G_{n-5}^1$ and $G_{n-5}^2$ where we may assume that
$\{ u_{(n-4)1}, u_{(n-4)2} \} \subseteq V(G_{n-5}^1)$ and $\{ v_{(n-4)1}, v_{(n-4)2} \} \subseteq V(G_{n-5}^2).$  By construction, both $G_{n-5}^1$ and $G_{n-5}^2$ must contain edges of $A.$ We may assume that for $j= 1,2,$ $a_j \in E(G_{n-5}^j).$

\begin{noname}
There is at most one integer $i\in  \{ n-3, n-2, n-1 \}$ for which $u_{i1}, u_{i2}$ belong to different components either $G_{n-5}^1$ or $G_{n-5}^2$.
\label{newnoname7}
\end{noname}

\begin{proof}
Suppose there exist two distinct $i \in \{ n-3, n-2, n-1 \}$ such that $u_{i1}, u_{i2}$ belong to different components either $G_{n-5}^1$ or $G_{n-5}^2$.  Without loss of generality, we may assume that for all $i \in \{ n-1, n-2 \}$
$u_{i1} \in V(G_{n-5}^1)$ and $u_{i2} \in V(G_{n-5}^2).$  We may assume that $u_{(n-3)1} \in V(G_{n-5}^1).$  Let $P_1$ be a rainbow path from $u_{(n-1)1}$ to $u_{(n-3)1}$ in $G_{n-5}^1$ and let $P_2$ be a rainbow path from $u_{(n-1)2}$ to $u_{(n-2)2}$ in $G_{n-5}^2.$  Let $C_1$ be the rainbow cycle where $E(C_1) = E(P_1) + a_{(n-1)1} + e_2  + a_{(n-3)1}$ and let $C_2$ be the rainbow cycle where $E(C_2) = E(P_2) + a_{(n-1)2} + e_1 + a_{(n-2)2}.$  Then $C_1$ and $C_2$ are edge-disjoint rainbow cycles where $$|C_1| + |C_2| = |P_1| + |P_2| +6 \le \nu(G_{n-4}^1) -1 + \nu(G_{n-4}^1) -1 + 6 \le n-3 + 4 = n+1,$$ and hence $\{ C_1, C_2 \}$ is a near-$T$-SRCP, a contradiction.
\end{proof}

\begin{noname}
For all integers $i\in  \{ n-3, n-2, n-1 \},$ $u_{i1}$ and  $u_{i2}$ belong to the same component either $G_{n-5}^1$ or $G_{n-5}^2$.
\label{newnoname8}
\end{noname}

\begin{proof}
By (\ref{newnoname7}),  there is at most one integer $i\in  \{ n-3, n-2, n-1 \}$ for which $u_{i1}, u_{i2}$ belong to different components either $G_{n-5}^1$ or $G_{n-5}^2$.  Suppose there is exactly one such $i$.  We may assume that 
for $j=1,2,$ $u_{(n-1)j} \in V(G_{n-5}^j).$  Furthermore, we may assume that for $j=1,2,$ $u_{(n-2)j} \in V(G_{n-5}^1).$  If for some $j\in \{ 1,2 \},$ $u_{(n-3)j} \in V(G_{n-5}^2),$ then we can construct a near-$T$-SRCP as was done in the proof of (\ref{newnoname7}).  Thus for $j=1,2,$ $u_{(n-3)j} \in V(G_{n-5}^1).$  Let $a_2' = a_{(n-1)1}$ and 
let $G_{n-5}^{1'} = G_{n-5}^1 + V(T) + a_2' + A_{n-2} + A_{n-3}.$  Clearly the edges of $(A - a_2 + a_2') \cup E(T)$ do not induce a $K_4$ in $G_{n-5}^{1'}.$  By assumption,  $G_{n-5}^{1'}$ has a near-$T$-SRCP and such is seen to be a near $T$-SRCP in $G'$ as well, a contradiction.
 \end{proof}
 
By (\ref{newnoname8}), we may assume that for all $i\in \{ n-1, n-2 \}$ and $j\in \{ 1,2 \}$, $u_{ij} \in V(G_{n-5}^1).$  Suppose that $\{ u_{(n-3)1}, u_{(n-3)2} \} \subset V(G_{n-5}^1).$   
Let $G_{n-5}^{1'} = G_{n-5}^1 + V(T) + A_{n-3} + A_{n-2} + A_{n-3}.$  Clearly, $A\cup E(T)$ does not induce a $K_4$ in $G_{n-5}^{1'}.$  Thus by assumption $G_{n-5}^{1'}$
has a near-$T$-SRCP which is also seen to be a near-$T$-SRCP in $G'.$  Thus $\{ u_{(n-3)1}, u_{(n-3)2} \} \subseteq V(G_{n-5}^2).$ Let $a_2' = a_{(n-4)1}$ and let
$G_{n-5}^{1'} = G_{n-5}^1 + x_1 + x_2 + A_{n-1} + A_{n-2} + v_{(n-4)1} + a_2'.$  
Let $x_3' = v_{(n-4)1}.$   Let $T'$ be the triangle having vertices $x_1, x_2, x_3'$ and edges $e_1' = x_1x_2, \ e_2' = x_1x_3',$ and $e_3' = x_2x_3'.$  One sees that the edges of $(A - a_2 + a_2') \cup E(T')$ do not induce a $K_4$ in $G_{n-5}^{1'}$ and thus
by assumption $G_{n-5}^{1'}$ has a near-$T'$-SRCP $\{C_1, C_2 \}$.
Let $P$ be a rainbow path in $G_{n-5}^2+ x_3 + a_{(n-3)1}$ from $x_3'$ to $x_3.$  For some $i\in [2]$, $a_2' \in C_i$ and we may assume this holds for $i=1.$  Furthermore, we may assume that $e_2' \in E(C_1).$  Let $C_1'$ be the cycle where $E(C_1') = E(C_1) - e_2' + E(P) + e_2.$  Then $C_1'$ and $C_2$ are edge-disjoint rainbow cycles for which $$|C_1'| + |C_2| \le |C_1| + |C_2| + |P| \le \nu(G_{n-5}^1) + 5 + \nu(G_{n-5}^2) = n+2$$ and hence $\{ C_1',C_2\}$ is a near-$T$-SRCP, yielding final contradiction. This completes the proof of the Theorem
\ref{newlemma2}.
\end{proof}

%

\subsection{Completing the proof of Theorem \ref{the-NoRainbowCircuit} iii) for graphic matroids}\label{sec-completeNoRainbowCircuitgraphiciii)}

Let $M=M(G)$ be a simple graphic matroid where $\nu(G) = n$ and $\varepsilon(M) = \varepsilon(G) = 2(n-2)$.  Let $c$ be $2$-uniform $(n-2)$-colouring of $M$ (and $G$) where $c$ is circuit - achromatic.  Then $G$ is connected and $r(M) = n-1.$
Let $G'$ be the graph obtained from $G$ by adding the edges $e_i, \ i = 1,2,3$ which form a $3$-cycle $T$ in $G'$ where $E(T)$ is co-independent in $M(G).$
Let $V(T) = \{ x_1, x_2, x_3 \}$ where $e_1 = x_1x_2,\ e_2 = x_1 x_3,$ and $e_3 = x_2x_3.$ 
We shall use a proof by contradiction, assuming that the Theorem \ref{the-NoRainbowCircuit} iii) is false and $G$ is a minimum counterexample; that is, $(M,c)$ has no $T$-SRCP.   As before, for all $x \in E(M)$, we let $x'$ denote the other element in the colour class containing $x.$

\begin{noname}
For all $v\in V(G),$ if $d_G(v) \le 2$ or $d_G(v) = 3$ and $|c^{-1}(E_G(v))| =2,$ then $v\in V(T).$
\label{nona6}
\end{noname}

\begin{proof}
 Let $v\in V(G) \backslash V(T).$  Assume that either $d_G(v) \le 2$ or $d_G(v) = 3$ and $|c^{-1}(E_G(v))| =2.$ 
 
 \sms
 \noindent (I.1) $d_G(v) \ge 2.$

\begin{proof} 
 Suppose $d_G(v) =1.$  Let $E_G(v) = \{ e \}$ and let $H = G/e\backslash e'$ and $M' = M(H)$.  Let $c' = c \big| E(M')$.  Then $M'$ is simple, $\varepsilon(M) = 2(r(M')-1) = 2(n-3)$ and $c'$ is a $2$-uniform $(n-3)$-colouring.   By assumption,  $(M',c')$ has a $T$-SRCP and such is also seen to be a $T$-SRCP for $(M,c).$
Thus $d_G(v) \ge 2.$ 
\end{proof}

\sms
\noindent (I.2) $d_G(v) =3.$

\begin{proof}
By (I.1), $d_G(v) \ge 2.$
Suppose $d_G(v) =2$ and let $E_G(v) = \{ e,f \}.$  
Let $H = G/e\backslash e'$ and let $M' = M(H).$  Let $c' = c\big| E(M').$
Then $M'$ is simple, $\varepsilon(M') = 2(r(M')-1) = 2(n-3)$ and $c'$ is a $2$-uniform $(n-3)$-colouring for which $(M',c')$ is circuit - achromatic.
By assumption, $(M,c')$ has a $T$-SRCP $\{ C_1, C_2 \}.$
If $f \not\in C_1 \cup C_2$, then $\{ C_1, C_2\}$ is a $T$-SRCP for $(M,c).$   Thus $f\in C_1 \cup C_2$ and hence $f \ne e'.$  We may assume that $f\in C_1.$  Then the circuit $C_1' = C_1 + e$ is a rainbow circuit and furthermore, $C_1'$ and $C_2$ are edge-disjoint and $|C_1'| + |C_2| = |C_1| + |C_2| \le r(M') +2 +1 = r(M) +2.$  This implies that $\{ C_1',C_2 \}$ is a $T$-SRCP for $(M,c).$  Thus $d_G(v) =3$.
\end{proof} 

By (I.2) we have $d_G(v) =3$ and hence $|c^{-1}(E_G(v))| =2.$  Let $E_G(v) = \{ e,e',f \}.$  We note that if $e,f,g$ are the edges of a $3$-circuit in $M,$ then and $e'$ and $f$ do not belong to a $3$-circuit in $M$.
Because of this, we may assume that $e$ and $f$ do not belong to a $3$-circuit in $M$.
Let $H = H/e\backslash e'$ and let $M' = M(H).$ Let $c' = c\big| E(M').$  Then $M'$ is simple and $c'$ is a $2$-uniform $(r(M')-1)$-colouring for which $(M',c')$ is circuit - achromatic.
By assumption, there is a $T$-SRCP $\{ C_1, C_2 \}$ for $(M',c').$
 If $f \not\in C_1 \cup C_2$, then $\{ C_1, C_2 \}$ is a $T$-SRCP for $(G,c)$.  Thus  $f\in C_1 \cup C_2$ and we may assume that $f \in C_1.$  Let $C_1' = C_1 + e.$  Then $\{ C_1',C_2 \}$ is seen to be
 a $T$-SRCP for $(M,c).$  Thus no such vertex $v$ exists.
\end{proof}

The next assertion follows from (\ref{nona6}) and that fact that $\varepsilon(G') = 2n-1$ and $\sum_{v\in V(G')} d_{G'}(v) = 4n-2.$    

\begin{noname}
There are at least two vertices $v\in V(G')$ where $d_{G'}(v) \le 3.$
\label{nona6.5}
\end{noname}

%
\begin{noname}
There exists a rainbow vertex $v \in V(G) - V(T)$ where $d_G(v) =3.$
\label{nona7}
\end{noname}

\begin{proof} 
By (\ref{nona6.5}), there exist vertices $v_i,\ i = 1,2$ for which $d_{G'}(v_i) \le 3.$
Suppose that $\{ v_1, v_2 \}  \subseteq V(T).$  Given that $G$ is connected, it follows that for $i = 1,2,$ $d_{G'}(v_i) = 3$
and $d_G(v_i) =1.$  For $i=1,2,$ let $E_G(v_i) = \{ f_i \}.$ Let $H = G - v_1 - v_2$ and let $c'$ be the (circuit-achromatic) colouring of $H$ obtained by restricting $c$ to $H$.  If $c(f_1) \ne c(f_2),$ then $|c'(H)| = |c(G)| = n-2 = \nu(H),$ implying that 
$H$ contains a rainbow cycle.  Thus $c(f_1) = c(f_2),$ $|c'(H)| = n-3 = \nu(H) -1,$ and $c'$ is $2$-uniform.  However, since $H$ is simple, Corollary \ref{cor-cor2} implies that $(H,c')$ contains a rainbow cycle, a contradiction.  Thus for some $i\in \{ 1,2 \},$  
$v_i \not\in V(T).$  Now (\ref{nona6}) implies that $d_G(v_i) =3$ and $v_i$ is a rainbow vertex. 
\end{proof}

By (\ref{nona7}), there exists a rainbow vertex $v \in V(G) - V(T),$ where $d_G(v) =3.$  Let $E_G(v) = \{ f_1, f_2, f_3 \}$.  We may assume that for $i= 1,2,3,$ $c(f_i) = i.$  
Let $H = G - v$ and let $c'$ be the colouring of $H$ obtained by restricting $c$ to $H.$  Then $|c'(H)| = n-2$ and $f'_i,\ i = 1,2,3$ are the colour-singular edges in $(H,c').$ 

\begin{noname}
The edges of $\{ f_1',f_2',f_3' \} \cup E(T)$ induce a $K_4$ in $H.$
\label{nona8}
\end{noname}

\begin{proof} Suppose to the contrary that the edges of $\{ f_1', f_2', f_3' \} \cup E(T)$ do not induce a $K_4.$  Then by Theorem \ref{newlemma2}, $(H,c')$ has a near-$T$-SRCP $\{ C_1, C_2 \}$ for which $(E(C_1), E(C_2))$ is seen to be a $T$-SRCP for $(M,c),$ a contradiction.
\end{proof}

By (\ref{nona8}),  the edges of $\{ f_1', f_2', f_3' \} \cup E(T)$ induce a $K_4$ in $H.$ Then for some vertex $u$, $\{ f_1', f_2', f_3' \} \subseteq E_G(u).$  We may assume that for $i = 1,2,3,$ $f_i' = x_i u.$   

\begin{noname}
$d_{G'}(u) \ge 4.$
\label{nona9}
\end{noname}

\begin{proof}
Suppose to the contrary that $d_{G'}(u) = 3.$  Then $u$ is a rainbow vertex in $G'.$   Then we can interchange the roles of $u$ and $v.$  Arguing with $v$ in place of $u$ in (\ref{nona8}), we obtain that the edges of $\{ f_1,f_2,f_3 \} \cup E(T)$ induce a $K_4.$
Now letting $N$ be $M$ restricted to $\bigcup_i \{ f_i, f_i' \}$ and $c' = c \big| E(N),$ it is straightforward to find a $T$-SRCP for $(N,c').$  Thus $d_{G'}(u) \ge 4.$
\end{proof}

\begin{noname}
For some $i\in \{ 1,2,3 \},$ $d_{G'}(x_i) = 3.$
\label{nona10}
\end{noname} 

\begin{proof}
Suppose to the contrary that for all $i\in \{ 1,2,3 \}$ $d_{G'}(x_i) \ge 4.$  By (\ref{nona6.5}), there are at least two vertices in $G'$ of degree at most three.   Given that $d_{G'}(u) \ge 4$ (by (\ref{nona10})), there is a vertex $v' \in V(G) - V(T) - u -v$ where $d_G(v') \le 3.$  By a previous argument, $d_{G'}(v') = 3$ and $v'$ is a rainbow vertex.  Let $E_G(v') = \{ g_1, g_2, g_3 \}$
Arguing with $v'$ in place of $v$ in the proof of (\ref{nona8}), it follows that the edges $g_1', g_2', g_3'$ together with $E(T)$ induce a $K_4.$  Then for some vertex $u',$ $\{ g_1', g_2', g_3' \} \subseteq E_G(u')$ However, it is now seen that the subgraph of $G$ induced by
$V(T) \cup \{ u, u' \}$ contains a rainbow cycle, a contradiction.
%
\end{proof}

By (\ref{nona10}), for some $i\in [3],$ $d_{G'}(x_i) = 3.$  Without loss of generality, we may assume that $d_{G'}(x_1) = 3.$  Suppose that  for some $j \in \{ 2,3 \},$ $d_{G'}(x_j) =3.$  We may assume that $d_{G'}(x_2) = 3.$  Let $K = G - x_1 -x_2$ and let $c' = c\big| E(K).$
Then either $|c'(K)| = |c(G)| = n-2 = \nu(K)$ or $|c'(K)| = n-3 = \nu(K)-1$ and $c'$ is $2$-uniform.  In either case, $(K,c')$ contains a rainbow cycle, a contradiction.
Thus for all $j\in \{ 2,3 \},\  d_{G'}(x_j) \ge 4.$   Suppose $d_G(u) > 4.$  Then given that $\sum_{v\in V(G')} d_{G'}(v) = 4n-2,$ $G'$ must have at least $3$ vertices of degree $3$ and hence there exists a rainbow vertex $v' \in V(G) - V(T) - u - v$ where $d_G(v') = 3.$  Arguing as in the proof of (\ref{nona10}), there is a $T$-SRCP for $(M,c).$
Thus $d_G(u) = 4.$   Let $K = G - u - x_1$ and let $c' = c\big| E(K).$
Then $|c'(K)| \ge n-3 = \nu(K) -1.$  Given that $K$ has no rainbow cycles, there is a spanning rainbow tree in $K.$  Thus there is a rainbow path $P$ in $K$ between $x_2$ and $x_3.$  Let $C_1 = P + e_3$ and let $C_2 = x_1 f_1' u f_2' x_2 e_1 x_1.$  Then $C_1, C_2$ are edge-disjoint rainbow cycles where $|C_1| + |C_2| \le n-2 + 3 = n+1.$  Thus $\{ E(C_1), E(C_2) \}$ is a $T$-SRCP for $(M,c).$ This yields a final contradiction.

\section{Circuit-achromatic cographic matroids $M$ with $2(r(M) -1)$ elements.}\label{sec-BigTheoremiii)Cographic}

In this section, we shall prove Theorem \ref{the-NoRainbowCircuit} iii) for cographic matroids. 
Let $G$ be a graph with $n$ vertices having no $2$-cocycles.  Let $M = M^*(G)$ where $\varepsilon(M) = 2(r(M) -1).$  We may assume that $G$ is connected.  As such, $r(M^*) = n-1$ and 
$r(M) = \varepsilon(M) - n+1 = 2(r(M)-1) - n+1.$  It follows that $r(M) = n+1$ and thus $\varepsilon(M) = \varepsilon(G) = 2n.$ Let $c: E(M) \rightarrow [n]$ be a $2$-uniform $n$-colouring of $M$ such that $(M,c)$ is circuit-achromatic; that is, $G$ has no rainbow cocycles.  
Let $G^3$ be the graph where
$$G \xrightarrow[v^1]{e^1} G^1 \xrightarrow[v^2]{e^2} G^2 \xrightarrow[v^3]{e^3} G^3.$$
Furthermore, assume that the edges of $T = \{ e^1, e^2, e^3 \}$ form a co-independent $3$-cocycle in $G^3.$  Extend the colouring $c$ to $G^3$ where for $i = 1,2,3$, $c(e^i) = n+1.$
%
To prove Theorem \ref{the-NoRainbowCircuit} iii) for cographic matroids, it will suffice to prove the following:

\begin{theorem}
There are two edge-disjoint rainbow cocycles $C_i,\ i =1,2$ in $(G^3,c)$ where for $i= 1,2,$ $e^i \in C_i$ and $|C_1| + |C_2| \le r(M) + 2 = n+3.$  That is, $\{ C_1, C_2 \}$ is a T-SRCoP for $(G^3,c).$
\label{newtheorem3} 
\end{theorem}

\begin{proof}
We will prove the theorem using induction on $n.$  When $n=1$, $G$ consists of two loops and one vertex. In this case, it can be shown that $G^3$ has a $T$-SRCoP.  We shall assume that $n>1$ and the theorem is true for graphs with 
fewer than $n$ vertices.   
 Let $H = G^3$ and let $m = \varepsilon(H) = 2n+3.$  For $j=1,2,3$, let $e^j = u_{1j}u_{2j}$ where for $i=1,2,$ the vertices in $U_i = \{ u_{i1}, u_{i2}, u_{i3} \}$ belong to the same component of $H - T.$  Let
 Let $U = U_1 \cup U_2.$  We say that a colour class $\{ e,e' \}$ in $H$ is {\bf non-incidental} if $e$ and $e'$ are non-incident.  Let $\xi$ be the number of non-incidental colour classes in $H.$  Let $R$ be the set of rainbow vertices in $H$ and let $\rho = |R|.$   
 Since $(M,c)$ is circuit-achromatic, $(G,c)$ has no rainbow vertices and it follows that $R \subseteq U.$  Similar to (\ref{nonacolem3.part1}) in Section \ref{sec-completeNoRainbowCircuitCographic}, we have the following:

\begin{noname}
$(H,c)$ has no monochromatic digons.
\label{newtheorem3.nona1}
\end{noname}

It follows by (\ref{newtheorem3.nona1}) that there are at most $n+2$ vertices having an incident pair of edges of the same colour in $H$ (where at most two pairs belong to $T$). Thus we have that $|R| \ge \nu(H) - (n+2) = 1.$ 
%
 
 \begin{noname}
 We may assume that there is at most one vertex $v\in V(H)$ where $d_H(v) =2.$
 \label{newtheorem3.nona2}
 \end{noname}

 \begin{proof} Suppose there is a vertex $v\in V(H)$ where $d_H(v) = 2.$ Since $M$ is simple, $G$ has no $2$-cocycles and hence $v\in U.$  It follows that $|E_H(v) \cap T| =1.$  If there exists another vertex $v'\in V(H)$ where $d_H(v') = 2,$ then again we have that $v' \in U$ and $|E_H(v') \cap T| =1.$
 If $v$ and $v'$ are non-adjacent, then letting $C_1 = E_H(v)$ and $C_2 = E_H(v'),$ $\{ C_1, C_2 \}$ is seen to be a $T$-SRCoP in $H.$  If $v$ and $v'$ are adjacent, then $C = E_H(v) \triangle E_H(v')$ is a $2$-cocycle of $H$ where either $C \subset T$ or $C \subset E(G)$.  Since by assumption $T$ is co-independent, the former can not occur.  However, the latter is also impossible since $M$ is simple. 
 \end{proof}
 
 \begin{noname}
 We may assume that every pair of vertices in $R$ is adjacent. \label{newtheorem3.nona3}
 \end{noname}
 
 \begin{proof}
 Suppose that $v_1, v_2 \in R.$  Then $\{ v_1, v_2 \} \subseteq U$ and it follows that for $i=1,2,$ $|E_H(v_i) \cap T| =1.$  Suppose first that $v_1$ and $v_2$ are non-adjacent.  If $d_H(v_1) + d_H(v_2) \le n+3,$ then letting $C_i = E_H(v_i),\ i = 1,2,$ $\{ C_1, C_2 \}$ is seen to be a $T$-SRCoP in $H.$  Suppose instead that $d_H(v_1) + d_H(v_2) \ge n+4.$
If for all $v \in V(H) - \{ v_1, v_2 \}$, $d_H(v) \ge 3,$  then
\begin{align*} 
 2(2n+3) &= 2\varepsilon(H) = \sum_{v\in V(H)}d_H(v) \ge n+4 + 3(\nu(H)- 2)\\ &= n+4 + 3(n+1) = 2(2n+3) +1,\end{align*} yielding a contradiction.  Thus there is a vertex $x \in V(H)$ where $d_H(x) =2.$  Then $x\in U$ and $|E_H(x) \cap T| =1.$  
 We may assume that $E_{H}(x)= \{ e^1, e \}$ and $x = u_{11}.$  Let $\{ e, e' \}$ be the colour class containing $e$.
 Given there is only one such vertex $x$ (by (\ref{newtheorem3.nona2})),
 the above inequality also implies that $d_H(v_1) + d_H(v_2) = n+4.$  
 
 \sms
\noindent (C.1) We may assume that $x$ is adjacent to at most one of $v_i,\ i = 1,2.$
 
 \begin{proof}
 Suppose that $x$ is adjacent to both $v_1$ and $v_2$. We may assume that $u_{21} = v_1.$  We have $e = xv_2.$  Then $\{ e, e' \}$ is a non-incidental colour class. Let $C = E_{H}(\{ x, v_1, v_2 \}).$  Then $C$ is cocycle in $H$ where $|C| = d_{H}(v_1) + d_{H}(v_2) -2 = n+2.$ 
 Given that $|c(H)| = n+1,$ at least one pair of edges in $C$, say $f$ and $f'$, have the same colour.  The colour class class $\{ f,f' \}$ is also seen to be non-incidental. It follows that $\xi \ge 2$ and we see that $\rho \ge n+3  - (n+1 -2) = 4.$
 Thus there exists $v_3 \in R - \{ x, v_1, v_2 \}$ and $d_{H}(v_3) \le n+1.$  Let $C_1 = E_{H}(x)$ and $C_2 = E_{H}(v_3).$  Then $|C_1| + |C_2| \le 2 + n+1 = 3$ and for $i=1,2,$ $|C_i\cap T| =1.$  It follows that  $\{ C_1, C_2 \}$ is a $T$-SRCoP for $(H,c).$
 Thus we may assume that $x$ is adjacent to at most one of $v_i,\ i = 1,2.$
 \end{proof}  
 
 By (C.1) we may assume that $x$ and $v_1$ are non-adjacent and hence $d_H(v_1) \le n+1.$
 Let $C_1 = E_H(v_1)$ and let $C_2 = E_H(x).$  Then $C_i,\ i = 1,2$ are disjoint rainbow cocycles where $|C_i\cap T| =1$ and $|C_1| + |C_2| \le n+1 + 2 = n+3.$  It follows that $\{ C_1, C_2 \}$ is a $T$-SRCoP for $(H,c).$
 Thus we may assume that every pair of vertices in $R$ is adjacent.
 \end{proof} 
 
 \begin{noname}
 We may assume that $|R| \ge 2.$\label{newtheorem3.nona4}
 \end{noname}
 
 \begin{proof}
 Suppose $|R| < 2.$  As noted before, there are at most $n+2$ incident pairs of edges of edges of the same colour in $H$ (where at most two pairs belong to $T$) and we have that $|R| \ge \nu(H) - (n+2) = 1.$   Then $|R| =1$ and there are exactly $n+2$ vertices having an incident pair of edges of the same colour and two of these vertices belong to $U$ and have exactly two edges of $T$ incident with them.   Thus all colour classes $\{ e,e' \}$ are incidental and the edges of $T$ form a path which we may assume
 is $u_{11} e^1 u_{21} e^2 u_{12} e^3 u_{23}.$ We may also assume that $v = u_{11}$ is the unique rainbow vertex.
 Let $H_i,\ i = 1,2$ be the components of $H-T$ where for $i=1,2,$ $U_i \subseteq V(H_i).$  Let $C_1 = E_{H}(v)$ and $C_2 = E_H(V(H_2) - u_{21} ).$  Then $C_i,\ i = 1,2$ are seen to be a pair of edge-disjoint rainbow cocycles.  If $|C_1| + |C_2| \le n+3,$ then $\{ C_1, C_2 \}$ is seen to be a $T$-SRCoP for $(H,c).$  Thus we may assume that $|C_1| + |C_2| \ge n+4.$  It follows that $d_H(v) + d_H(u_{21}) -1 = |C_1| + |C_2| \ge n+4,$ and hence there is at least one colour class $\{ e,e' \}$ where $e \in E_H(v) - e_1$ and $e'\in E_H(u_{21}) - e_1 - e_2.$  However,
 we have $e \in E(H_1)$ and $e'\in E(H_2)$ and hence $\{ e,e' \}$ is a non-incidental colour class.  This gives a contradiction.
 \end{proof}
 
 \begin{noname}
 We may assume that $|R| \ge 3.$\label{newtheorem3.nona5}
 \end{noname}
 
 \begin{proof}
 By (\ref{newtheorem3.nona4}), we may assume that $|R| \ge 2.$  Suppose $|R|=2$ and let $R = \{ v_1, v_2 \}.$  By (\ref{newtheorem3.nona3}), we may assume that $e = v_1v_2$ is an edge.  Suppose first that $e \not\in T.$  Then the colour class $\{ e, e' \}$ is non-incidental.  Thus $\xi \ge 1$ and hence $2 = |R| \ge \nu(H) - (n+2 -\xi) \ge \xi +1\ge 2.$  Thus $\xi =1$ and as in the proof of (\ref{newtheorem3.nona4}), there are exactly $n+1$ vertices having an incident pair of edges of the same colour and two of these vertices belong to $U$ and have exactly two edges of $T$ incident with them.  Thus
 the edges of $T$ form a path $P$ from $v_1$ to $v_2$.  However $C = P +e$ is seen to be a cycle where $E(C) \cap T = T$ which is not possible since $|T| = 3.$
 It follows from the above that $e \in T.$  Now $E_{H}(R)$ is seen to be a cocycle in $G$ and thus it contains two edges of the same colour.   Such edges must form a non-incidental colour class and hence $\xi \ge 1.$  Arguing as before, we see that $T$ is a path having endvertices $v_1$ and $v_2.$  However, this implies that $e = v_1v_2 \not\in T,$ a contradiction.
 Thus $|R| \ge 3.$ 
 \end{proof}
 
 By (\ref{newtheorem3.nona5}), we may assume that $|R| \ge 3.$  Since $R \subseteq U,$ at most one vertex of $H$ has two incident edges from $T.$  Let $k = |R|.$  By (\ref{newtheorem3.nona3}), each pair of vertices in $R$ is adjacent.  At most $\lfloor \frac k2 \rfloor$ pairs of vertices in $R$ are joined by an edge in $T.$  Since $k\ge 3,$ there are at least ${\binom k2} - \lfloor \frac k2 \rfloor \ge k-1$ edges between vertices of $R$ which do not belong to $T.$   Consequently, there are at least $k-1$ non-incidental colour classes and hence $\xi \ge k-1.$
 It now follows that $k = |R| \ge n+3 - (n+1 - \xi) \ge 2 + (k-1) = k+1,$ a contradiction.   This completes the proof of Theorem \ref{newtheorem3}.

\end{proof}

\section{The proof of Theorem \ref{the-NoRainbowCircuit} i)}\label{sec-ProofNoRainbowCircuitRegulari)}

Let $M$ be a simple rank-$n$ regular matroid where $\varepsilon(M) = 2n-1.$  Let $c$ be a $2$-bounded $n$-colouring of $M$ where $(M,c)$ is circuit-achromatic.  By Observation \ref{obs-onesingular}, $(M,c)$ contains a unique singular element which we shall denote by $e.$
Let $N$ be a regular extension by $T,$ a co-independent $3$-circuit in $N.$ 
%
It follows from the proofs in Sections \ref{sec-BigTheoremi)andii)Graphic}, \ref{sec-BigTheoremi)andii)Cographic}, \ref{sec-BigTheoremiii)Graphic}, and \ref{sec-BigTheoremiii)Cographic}
%
that Theorem \ref{the-NoRainbowCircuit} is true for graphic and cographic matroids.  
To prove that Theorem \ref{the-NoRainbowCircuit} i) is true for regular matroids, we shall use induction on $n$, where we note that the theorem is true for $n =3.$   We may assume that $N$ is connected and is neither graphic nor cographic.  Furthermore, $M$ is not isomorphic to $R_{10}$ since $\varepsilon(R_{10}) = 2 r(R_{10}).$  It follows by Seymour's decomposition theorem \cite{Sey}, that $N$ is a $2$- or $3$-sum of matroids which are either graphic, cographic, or isomorphic to $R_{10}.$  

\subsection{The case where $N = N_1 \oplus_2 N_2$}

Suppose first that $N = N_1 \oplus_2 N_2.$ where $E(N_1) \cap E(N_2) = \{ s \}$. Given that $T$ is a $3$-circuit of $N$, we may assume that $T \subseteq E(N_1).$  Let $M_1 = N_1\backslash T$ and $M_2 = N_2.$  For $i = 1,2,$ let $n_i = r(M_i),$ $M_i' = M_i \backslash s,$ $\varepsilon_i' = \varepsilon(M_i'),$ and let $c_i' = c \big| E(M_i').$   We have that $n =  n_1 + n_2 -1.$

   Since $M$ is simple and $(M,c)$ is circuit-achromatic, it follows that for $i=1,2,$ $(M_i',c_i')$ is simple and circuit-achromatic and thus for $i= 1,2,$ $\varepsilon_i' \le 2n_i -1.$

\begin{noname}
We may assume that $\varepsilon_1' = 2n_1 -2$ and $\varepsilon_2' = 2n_1-1.$
\label{regnona1}
\end{noname}

\begin{proof}
We have that $2n-1 = \varepsilon(M) = \varepsilon_1' + \varepsilon_2'.$  Thus $2(n_1 + n_2 -1) -1 = \varepsilon_1' + \varepsilon_2' $ and hence it follows that for some $i\in \{ 1,2 \},$ $\varepsilon_i' \ge 2n_i -1 = 2r(M_i') -1;$ that is, for some $i$, $\varepsilon_i' = 2n_i -1 .$  If $\varepsilon_1' = 2n_1-1,$ then by assumption, Theorem \ref{the-NoRainbowCircuit} i) holds for $(M_1', c_1')$ and $T$.
Then it is seen that Theorem \ref{the-NoRainbowCircuit} i) holds for $(M,c)$ and $T$ as well. 
As such, we may assume that
$\varepsilon_2' = 2n_2-1$ and $\varepsilon_1' = 2n_1 -2.$
\end{proof}

By Observation \ref{obs-onesingular}, $(M_2', c_2')$ has a unique singular element which we denote by $f.$  
 %
Since $|c(M_2')| = n_2,$ it follows by Observation \ref{obs-erainbowcircuit} that $M_2'$ has an $s$-rainbow circuit $C$ where $|C|\le n_2.$   Extend the colouring $c_1'$ to a colouring $c_1$ of $M_1$ where $c_1(s):= c(f).$ 
We claim that $(M_1,c_1)$ is circuit-achromatic.  For if it has a rainbow circuit $D$, then $s\in D$ and $C\triangle D$ would be a rainbow circuit in $(M,c)$, a contradiction.  Thus $(M_1,c_1)$ is circuit-achromatic.  By Observation \ref{obs-onesingular},  $(M_1, c_1)$ contains a unique colour-singular element which we denote by $g.$  Note that if $e \in E(M_2'),$ then $f = e$ and $g = s;$ otherwise, $g = e.$

\begin{noname}
$M_1$ is simple.\label{refnona3}
\end{noname}

\begin{proof}
Suppose that $M_1$ is not simple.  Then there is an element $x \in E(M_1')$ which is parallel with $s.$ Since $(M_1,c_1)$ is circuit-achromatic, it follows that $c_1(x) = c_1(s) = c(f).$  Thus $x = f'$.  Now $M_2' + f'$ is a rank-$n_2$ simple matroid with $2n_2$ elements where each colour class contains two elements.  It follows by Corollary \ref{cor-cor2} that $M_2' + f'$ contains a rainbow circuit, a contradiction. Thus $M_1$ is simple.
\end{proof}
 
Given that $\varepsilon(M_1) = 2n_1 -1$ and $M_1$ is simple, it follows by assumption that Theorem \ref{the-NoRainbowCircuit} i.1) or i.2) holds for $(M_1,c_1)$ and $T$.   Thus either 
  
\begin{itemize}
\item[a)] $(M_1,c_1)$ has a $T$-SRCT $\{ D_1, D_2, D_3 \}$ or 
\item[b)] for all $x\in T,$ there exists a contains a $T$-SRCP $\{ D_1, D_2 \}$ for $(M_1-g,c_1)$ where $x \not\in D_1 \cup D_2.$
\end{itemize}
Suppose a) holds.
If $s \not\in \bigcup_{i=1}^3D_i$, then $\{ D_1, D_2, D_3 \}$ is seen to be a $T$-SRCT for $(M,c).$
Suppose $s\in \bigcup_{i=1}^3D_i.$  We may assume that $s\in D_1.$ Then $D_1' = D_1 \triangle C$ is seen to be a rainbow circuit and $\{ D_1', D_2,D_3 \}$ is seen to be a $T$-SRCT for $(M,c).$ 
 Thus if a) holds, then Theorem \ref{the-NoRainbowCircuit} i.1) holds for $M.$
 
 Suppose b) holds.  Let $x\in T$ and assume $\{ D_1, D_2\}$ is a $T$-SRCP for $(M_1 - g, c_1)$ where $x \not\in D_1 \cup D_2.$  If $s \not\in D_1 \cup D_2,$ then $\{ D_1, D_2\}$ is seen to be a $T$-SRCP for $(M - e,c).$
 Suppose that $s \in D_1 \cup D_2$ where we may assume that $s \in D_1.$   Then $g \ne s$ and hence $g =e$.  Let $D_1' = C \triangle D_1.$  Then $\{ D_1', D_2\}$ is seen to be a $T$-SRCP for $(M - e,c).$
 Thus if b) holds, then Theorem \ref{the-NoRainbowCircuit} i.2) holds for $(M,c).$    
 
\subsection{The case where $N = N_1 \oplus_3 N_2$}

Suppose now that  $N = N_1 \oplus_3 N_2.$ where $E(N_1) \cap E(N_2) = S= \{ s_1, s_2, s_3 \}$.  We may assume as before that $T \subset E(N_1).$  Let $M_1 = N_1\backslash T,$ $M_2 = N_2$ and for $i = 1,2,$ let $n_i = r(M_i),$ $M_i' = M_i \backslash S,$ $\varepsilon_i' = \varepsilon(M_i'),$ and let $c_i' = c \big| E(M_i').$   We have that $r(M) = n =  n_1 + n_2 -2.$   Let $T = \{ t_1, t_2, t_3 \}.$. We shall need the following lemma:

\begin{lemma}
Suppose that for some $i\in \{ 1,2 \},$ $|c_i'(M_i')| \ge n_i -1.$  Then for some $j \in [3],$ there exists an $s_j$-rainbow circuit $C$ for $(M_i',c_i')$ where $|C| \le n_i$.
Furthermore, if $|c_i'(M_i')| = n_i,$ then for any element $f \in E(M_i'),$ one can choose $C$ so that $f \not\in C.$\label{cl1}
\end{lemma}

\begin{proof}
It suffices to prove the lemma when $|c_1'(M_1')| \ge n_1-1.$   There exists $B \subseteq E(M_1)$ where $|B| = n_1$, $s_1 \in B$, and $B-s_1$ is a rainbow subset of $(M_1',c_1').$  Note that if $|c_1'(M_1')| = n_1,$ then for any element $f,$ we can choose $B$ such that $f\not\in B.$  If $B$ is not a basis for $M_1$, then it contains a circuit $C$ for which $s_1 \in C$ and hence $C$ is an $s_1$-rainbow circuit for $(M_1',c_1')$ where $|C| \le n_1.$
Thus we may assume that $B$ is a basis for $M_1.$  Then $B+s_2$ contains a circuit $C$.  If $C \cap S \ne \{ s_1, s_2 \},$ then $C$ is seen to be either a $s_1$- or $s_2$- rainbow circuit for $(M_1',c_1')$ where $|C| \le n_1.$  Suppose that $C\cap S = \{ s_1, s_2 \}.$  Then
$C\triangle S$ is seen to be an $s_3$- rainbow circuit for $(M_1',c_1')$ for which $|C\triangle S| \le n_1.$  In either case we have the desired circuit $C.$  If $|c_1'(M_1')| = n_1,$ then as noted above, for any element $f$, we can choose $B$ so that $f \not\in B$ and consequently $f \not\in C.$ 
\end{proof}

 We have $\varepsilon(M) = 2(n_1 + n_2 -2) -1 = 2n_1 + 2n_2 -5.$  Thus for some $i\in \{ 1,2 \},$ $\varepsilon_i' \ge 2n_i -2.$  Given that for $i=1,2$, $(M_i',c_i')$ rainbow circuit -free, we have that for $i=1,2,$ $\varepsilon_i' \le 2n_i -1$ and  $|c(M_i')| \le n_i.$ 


\sms
\noindent{\bf Case 0}:  Suppose that for some $i\in [2],$ $\varepsilon_i' = 2n_i -1.$  
\sms
 
 Suppose $\varepsilon_1' = 2n_1-1.$  By assumption Theorem \ref{the-NoRainbowCircuit} i) holds for $(M_1',c_1')$ and $T$, and in this case,  Theorem \ref{the-NoRainbowCircuit} i) is readily seen to hold for $(M,c)$.  
 
 Suppose $\varepsilon_2' = 2n_2 -1.$  
 Let $f$ be the unique colour-singular element of $(M_2',c_2').$ We also have that $\varepsilon_1'(M_1') = 2n_1-4$ and $(M_1',c_1')$ has $0$ or $2$ colour singular elements; in the former case, $f=e$, and in the latter, $f\ne e$ and $e,f'$ are its colour singular elements.
 By assumption Theorem \ref{the-NoRainbowCircuit} i) holds for $(M_1',c_1')$ and $S.$  Thus there is a $S$-SRCP $\{ C_1', C_2' \}$ for $(M_2',c_2')$ where $f \not\in C_1' \cup C_2'.$  We may assume that for $i=1,2,$ $s_i \in C_i'.$
Furthermore, we may assume that if $(M_2',c_2')$ has a $S$-SRCT, then there is a circuit $C_3'$ such that $\{ C_1', C_2', C_3' \}$ is a $S$-SRCT for $(M_2',c_2').$
 Extend $c_1'$ to a colouring $c_1$ of $M_1$ where for $i=1,2,$ $c_1(s_i) := 0$ and $c_1(s_3):= c(f).$  Suppose $(M_1,c_1)$ has a rainbow circuit $D$.  Then $D\cap S \ne \emptyset.$  If for some $i\in [2],$ either $D\cap S = \{ s_i \}$ for $D\cap S_i = \{ s_i, s_3 \}.$  In the former, $D\triangle C_i'$ is a rainbow cycle, and in the latter $D\triangle S \triangle C_{3-i}'$ is seen to be a rainbow cycle.   If $D\cap S = \{ s_3 \},$ then for any $s_3$-rainbow circuit $C$ for $(M_2',c_2'),$ $C\triangle D$ is seen to be a rainbow cycle.  It follows from the above that $(M_1,c_1)$ is circuit-achromatic.  We also observe that $M_1$ is simple; for if there is an element $x\in E(M_1')$ which is parallel to an element of $S,$ then $M_2' + x$ would have rank $n_2$ and $2n_2$ elements and hence would have a rainbow circuit, a contradiction.  By assumption Theorem \ref{the-NoRainbowCircuit} i) holds for $(M_1,c_1).$  
 
\ssms 
\noindent{\it Case 0.1} Suppose that Theorem \ref{the-NoRainbowCircuit} i.1) holds for $(M_1,c_1)$ and $\{ D_1, D_2, D_3 \}$ is a $T$-SRCT for $(M_1,c_1).$ 
\ssms 

We may assume that $S \cap \bigcup_i D_i \ne \emptyset$; otherwise, $\{ D_1, D_2, D_3 \}$ is a $T$-SRCT for $(M,c).$  

\ssms
\noindent{\it Case 0.1.1} Suppose $e\in E(M_2').$
\ssms

We have that $f=e$.  Suppose first that Theorem \ref{the-NoRainbowCircuit} i.1) holds for $(M_2',c_2').$  Then by assumption there is a circuit $C_3'$ such that
$\{ C_1', C_2', C_3' \}$ is a $S$-SRCT for $(M_2',c_2')$ (where $s_3 \in C_3'$).
Suppose first that for all $i\in [3],$ $|D_i \cap T| \le 1.$  We may assume for all $i\in [3]$ that if $D_i \cap T \ne \emptyset,$ then $D_i \cap T = \{ s_i \}.$
For $i=1,2,3$ let $C_i = \left\{ \begin{array}{ll} D_i & \mathrm{if}\ D_i \cap T = \emptyset\\ D_i \triangle C_i' &\mathrm{if}\ D_i \cap T \ne \emptyset. \end{array}\right.$ 
Then $\{ C_1, C_2, C_3 \}$ is seen to be a $T$-SRCyT for $(M,c).$  Suppose instead that for some $i\in [3],$ $|D_i \cap T| = 2.$  We may assume that $D_1 \cap T = \{ s_1, s_3 \}.$
We may assume that $D_3 \cap T = \emptyset.$  If $D_2 \cap T = \emptyset,$ then letting $C_1 = D_1 \triangle S \triangle C_2', C_2 = D_2,$ and $C_3 = D_3,$ then it is seen that $\{ C_1, C_2, C_3 \}$ is a $T$-SRCyT for $(M,c).$
Suppose $D_2 \cap T \ne \emptyset.$ Then $D_2 \cap T = \{ s_2 \}.$  Clearly $e$ is not parallel with $s_2$, for otherwise, $C_2' - s_2 + e$ is a rainbow circuit in $(M,c).$  Since Theorem \ref{the-NoRainbowCircuit} ii) holds for $(M_2',c_2'),$
there is an $s_2$-semi-SRCP $\{ C_1'', C_2'' \}$ for $(M_2', c_2').$  Let $C_1 = D_1 \triangle S \triangle C_1'',\ C_2 = D_2\triangle C_2''$ and $C_3 = D_3.$  Then $\{ C_1, C_2, C_3 \}$ is seen to be a $T$-SRCyT for $(M,c).$
Thus Theorem \ref{the-NoRainbowCircuit} i.1) holds for $(M,c).$  

Suppose instead that Theorem \ref{the-NoRainbowCircuit} i.2) holds for $(M_2',c_2').$  We may assume that for all $i\in [3],\ |D_i \cap S| \le 1;$ otherwise, the previous arguments show that there is $S$-SRCT for $(M,c).$
We may also assume that for all $i\in [3],$ if $D_i \cap S \ne \emptyset,$ then $D_i \cap S = \{ s_i \}.$ For all $1 \le i < j \le 3$ there exists a $S$-SRCP $\{ C_i^{ij}, C_j^{ij} \}$ for $(M_2'-e, c_2')$ where $s_i \in C_i^{ij}$ and $s_j \in C_j^{ij}.$
For all $1 \le i < j \le 3,$ let $D_i^{ij} = \left\{ \begin{array}{ll} D_i \triangle C_i^{ij} & \mathrm{if}\ s_i \in D_i\\ D_i &\mathrm{if}\ s_i \not\in D_i\end{array} \right.$ and $D_j^{ij} = \left\{ \begin{array}{ll} D_j \triangle C_j^{ij} & \mathrm{if}\ s_j \in D_j\\ D_j &\mathrm{if}\ s_j \not\in D_j.\end{array} \right.$ Then $\{ D_i^{ij}, D_j^{ij} \}$ is seen to be a $S$-SRCyP for $(M-e,c)$ and hence Theorem \ref{the-NoRainbowCircuit} i.2) holds for $(M,c).$ 

\ssms
\noindent{\it Case 0.1.2}:  Suppose that $e\in E(M_1').$
\ssms

Then $e$ and $f'$ are the colour singular elements of $(M_1', c_1')$ and $|c_1'(M_1')| = n_1-1.$  By Lemma \ref{cl1}, for some $j\in [3]$, there is an $s_j$-rainbow circuit $D$ for $(M_1',c_1')$ where $|D| \le n_1.$
Then $D$ must be a $s_3$-rainbow circuit; for otherwise, $D\triangle C_1'$ or $D\triangle C_2'$ is a rainbow cycle.  If Theorem \ref{the-NoRainbowCircuit} i.2) holds for $(M_2', c_2'),$ then there is an $s_3$-rainbow circuit $C$ for $(M_2',c_2')$ where $f \not\in C.$  As such, $C\triangle D$ is seen to be
a rainbow cycle in $(M,c)$, a contradiction.  Thus Theorem \ref{the-NoRainbowCircuit} i.1) holds for $(M_2', c_2')$ and by assumption, there is a circuit $C_3'$ such that $\{ C_1', C_2', C_3' \}$ is a $S$-SRCT for $(M_2',c_2')$ (where  $s_3 \in C_3'$).
Extend the colouring $c_1'$ to a colouring $c_1$ of $M_1$ where $c_1(s_3):= c(f)$ and for $i= 1,2,$ $c_1(s_i):= 0.$
As in Case 0.1.1, $M_1$ is simple and $(M_1,c_1)$ is circuit-achromatic.  Now using the circuits $C_i', D_i,\ i = 1,2,3$ one can construct a $T$-SRCT for $(M,c)$ using a similar argument as in Case 0.1.1. 
 
 \ssms
 \noindent{\it Case 0.2}:  Suppose that Theorem \ref{the-NoRainbowCircuit} i.2) holds for $(M_1,c_1)$ and $T$.
 \ssms
 
 Let $g$ be the unique colour-singular element of $(M_1,c_1),$ noting that either $g=e$ or $g= s_3$, depending on whether $e\in E(M_1')$ or $e\in E(M_2').$  By assumption, for all $1 \le i < j \le 3,$ there is a $T$-SRCP $\{ D_{i}^{ij}, D_j^{ij} \}$ for $(M_1-g, c_1)$ where $t_i \in D_i^{ij}$ and $t_j \in D_j^{ij}.$ 
 
\ssms
\noindent{\it Case 0.2.1} Suppose $e\in E(M_2').$
\ssms
  
 We observe that $g=s_3$ and $f=e.$
 By the choice of $C_i',\ i = 1,2,$  $e \not\in C_1' \cup C_2'.$  For all $1 \le i < j \le 3,$ let $C_i^{ij} = \left\{ \begin{array}{ll} D_i^{ij} \triangle C_{i'}' &\mathrm{if}\ D_i^{ij} \cap S = \{ s_{i'} \}\\ D_i^{ij} &\mathrm{if}\ D_i^{ij} \cap S = \emptyset \end{array} \right.$ and 
 $C_j^{ij} = \left\{ \begin{array}{ll} D_j^{ij} \triangle C_{j'}' &\mathrm{if}\ D_j^{ij} \cap S = \{ s_{j'} \}\\ D_j^{ij} &\mathrm{if}\ D_j^{ij} \cap S = \emptyset. \end{array} \right.$  Then $\{ C_i^{ij}, C_j^{ij} \}$ is seen to be a $T$-SRCyP for $(M-e,c)$ where $t_i \in C_i^{ij}$ and $f_j \in C_j^{ij}.$
 
\ssms
\noindent{\it Case 0.2.2} Suppose $e\in E(M_1').$
\ssms    

We have that $g=e$.  It suffices to show that $\{ D_1^{12}, D_2^{12} \}$ can be extended to a $T$-SRCP $\{ C_1^{12}, C_2^{12} \}$ for $(M-e,c)$ where for $i= 1,2,$ $t_i \in C_i^{12}.$  We may assume that $S\cap (D_1^{12} \cup D_2^{12}) \ne \emptyset.$
Suppose first that for some $i\in \{ 1,2 \},$ $|D_i^{12} \cap S| = 2.$  We may assume that $D_1^{12} \cap S = \{ s_1, s_3 \}.$  If $D_2^{12} \cap S = \emptyset,$ then let $C$ be a $s_2$-rainbow circuit for $(M_2',c_2')$ and let
$C_1^{12} = D_1^{12} \triangle S \triangle C$ and $C_2^{12} = D_2^{12}.$  Then $\{ C_1^{12}, C_2^{12} \}$ is seen to be a $T$-SRCyP for $(M-e,c)$.  Suppose instead that $D_2^{12} \cap S = \{ s_2 \}.$  Since Theorem \ref{the-NoRainbowCircuit} ii) holds for $(M_2',c_2')$ and $f$ is not parallel with $s_2$ (otherwise, $C_2' - s_2 +f$ is a rainbow circuit), there exists a $s_2$-semi-SRCP $\{ C_1'', C_2''\}$ for $(M_2',c_2')$ where we may assume that $f\not\in C_2''.$  Letting $C_1^{12} = D_1^{12} \triangle S \triangle C_1''$ and $C_2^{12} = D_2^{12} \triangle C_2''$, we see that $\{ C_1^{12}, C_2^{12} \}$ is a $T$-SRCyP for $(M-e,c)$.

By the above, we may assume that for $i= 1,2,$ $|D_i^{12} \cap S| \le 1.$ Suppose first that $(M_2',c_2')$ has a $S$-SRCT.  Then by assumption, there is a circuit $C_3'$ where $\{ C_1',C_2',C_3' \}$ is a $T$-SRCT for $(M_2',c_2').$
For $i=1,2,$ let $$C_i^{12} = \left\{ \begin{array}{ll} D_i^{12} \triangle C_{i'}' &\mathrm{if}\ D_i^{12} \cap S = \{ s_{i'} \} \\ D_i^{12} &\mathrm{if}\ D_i^{12} \cap S = \emptyset.\end{array}\right.$$  Then $\{ C_1^{12}, C_2^{12} \}$ is seen to be a $T$-SRCyP for $(M-e,c).$
Suppose now that Theorem \ref{the-NoRainbowCircuit} i.2) holds for $(M_2', c_2').$  Then for all $1 \le i < j \le 3,$ there exists a $S$-SRCP $\{ F_i^{ij}, F_j^{ij} \}$ for $(M_2'-f,c_2')$ where $f \not\in F_i^{ij} \cup F_j^{ij}.$
Suppose that $D_1^{12} \cap S = \{ s_{i} \}$ and $D_2^{12} \cap S = \{ s_j \}.$  Letting $C_1^{12} = D_1^{12} \triangle F_{i}^{ij}$ and $C_2^{12} = D_2^{12} \triangle F_{j}^{ij},$ we see that $\{ C_1^{12}, C_2^{12} \}$ is a $T$-SRCyP for $(M-e,c).$
Suppose that for some $i\in [2],\ D_i^{12} \cap S = \emptyset.$  We may assume that $D_2^{12} \cap S = \emptyset.$  Then for some $i\in [2],\ D_1^{12} \cap S = \{ s_i \}.$  Let $C= F_i^{ij}$, where $j\ne i.$  Letting $C_1^{12} = D_1^{12} \triangle C$ and $C_2^{12} = D_2^{12},$ we seen that $\{ C_1^{12}, C_2^{12} \}$ is a $T$-SRCyP for $(M-e,c).$

From the above, we may assume that $\varepsilon_2' \le 2n_2-2$ and $\varepsilon_1' \ge 2n_1-3$. As such, there are two remaining cases where either $\varepsilon_1'= 2n_1-2$ or $\varepsilon_2' = 2n_2-2.$ 
If for some $i$, $\varepsilon_i = 2n_i -2,$ then $n_i -1 \le |c_i'(M_i')| \le n_i.$  We say that $(M_i', c_i')$ is of {\bf type 1} or {\bf type 2},  depending on whether $|c_i'(M_i')| = n_i-1$ or  $|c_i'(M_i')| = n_i.$

\sms
\noindent{\bf Case 1}: Suppose $\varepsilon_1' = 2n_1-2.$
\sms 

\ssms
\noindent{\it Case 1.1} $(M_1',c_1')$ is of type 1.
\ssms

 We have that $|c(M_1')| = n_2-1$ and hence $c_1'$ is $2$-uniform and $e \in E(M_2').$  We also have that $\varepsilon_2' = 2n_2 - 3.$  By Lemma \ref{cl1}, there exists $j \in [3]$ for which there is an $s_j$-rainbow circuit $C$ for $(M_2',c_2')$ where $|C|\le n_2.$  We may assume this is true for $j=1.$  
Extend $c_1'$ to a colouring $c_1''$ of $M_1'' = M_1' + s_1$ where $c_1''(s_1):=0.$  If $(M_1'', c_1'')$ contains a rainbow circuit, say $D$, then $s_1 \in D$ and hence $C\triangle D$ is a seen to be a rainbow cycle of $(M,c)$.  Thus $(M_1'', c_1'')$ is circuit-achromatic and hence no element of $E(M_1')$ is parallel with $s_1,$ implying that $M_1''$ is simple.  We also have that $\varepsilon(M_1'') = 2n_1-1$ and $s_1$ is the unique colour-singular element of $(M_1'', c_1'').$ 
By assumption,  Theorem \ref{the-NoRainbowCircuit} i) holds for $(M_1'',c_1'')$ and $T$ and either
\begin{itemize}
\item[a)] $(M_1'',c_1'')$ has a $T$-SRCT $\{ D_1, D_2, D_3 \}$ or 
\item[b)] For all $x\in T$, $(M_1'' - s_1, c_1'')$ has a $T$-SRCP $\{ D_1, D_2 \}$ where $x \not\in D_1 \cup D_2.$ 
\end{itemize}
Suppose a) occurs. If $s_1 \not\in \bigcup_{i\in [3]}D_i,$ then $\{ D_1, D_2, D_3 \}$ is seen to be a $T$-SRCT for $(M,c).$  Thus we may assume  that $s_1 \in \bigcup_{i\in [3]}D_i$ and more specifically, $s_1 \in D_1$.
Let $C_1 = C\triangle D_1.$  Since $|C| \le n_2,$ it is seen that $\{ C_1, D_2, D_3 \}$ is a $T$-SRCyT for $(M,c)$.  Thus if a) holds, then Theorem \ref{the-NoRainbowCircuit} i.1) holds for $(M,c).$

Suppose b) occurs.  Let $x\in T.$  Then $(M_1'' - s_1, c_1'')$ has a $T$-SRCP $\{ D_1, D_2 \}$ where $x \not\in D_1 \cup D_2.$  Then $\{ D_1, D_2 \}$ is seen to be a $T$-SRCP for $(M-e,c).$  Thus if b) holds, then
Theorem \ref{the-NoRainbowCircuit} i.2) holds for $(M,c).$

\ssms
\noindent{\it Case 1.2} $(M_1', c_1')$ is of type 2.
\ssms

We have that $(M_1',c_1')$ contains exactly two colour-singular elements which we denote by $f$ and $g.$  We shall consider two sub-cases depending on whether $e\in E(M_1')$ or $e\in E(M_2').$

\ssms
\noindent{\it Case 1.2.1} $e \in E(M_2').$ 
\ssms

We see that $(M_2',c_2')$ has exactly three singular elements which are seen to be $e, f'$ and $g'$.  Also, $|c_2'(M_2')| = n_2.$  
By Lemma \ref{cl1}, there exists $j\in [3]$ for which $(M_2',c_2')$ has an $s_j$-rainbow circuit $C$ where $|C|\le n_2.$  Moreover, $C$ can be chosen such that either $f' \not\in C$ or $g' \not\in C.$ We may assume that $C$ is an $s_1$-rainbow circuit. 
Let $M_1'' = M_1' + s_1$ and define a colouring $c_1''$ for $M_1''$ where $c_1''(s_1) := 0,$ $c_1''(g) := c(f)$ and for all $x\in E(M_1') - g,$ define $c_1''(x):= c_1'(x).$  Note that $s_1$ is the unique colour-singular element of $(M_1'', c_1'').$
We claim that $(M_1'', c_1'')$ is circuit-achromatic.  For suppose it contains a rainbow circuit, say $D$.  Then $s_1 \in D.$  If $\{ f, g \} \cap D = \emptyset,$ then $C\triangle D$ is seen to be a rainbow cycle of $(M,c).$  Thus $\{ f, g \} \cap D \ne \emptyset$.  We may assume that $f \in D.$  As mentioned above, we can choose $C$ so that $f' \not\in C.$  It now follows that $C\triangle D$ is a rainbow cycle of $(M,c).$  Thus $(M_1'', c_1'')$ is circuit-achromatic and hence no edge of $E(M_1')$ is parallel with $s_1$ and $M_1''$ is simple.
We can now argue as in Case 1.1 where it is seen that either a) or b) is true. 
Suppose that  a) is true. 
If $s_1 \not\in \bigcup_iD_i,$ then $\{ D_1, D_2, D_3 \}$ is seen to be a $T$-SRCT for $(M,c).$ 
Suppose $s_1 \in \bigcup_iD_i$  where we may assume $s_1 \in D_1.$  If $f \in D_1$ (resp. $g\in D_1$), then choose $C$ such that $f' \not\in C$ (resp. $g' \not\in C$).
Let $C_1 = C \triangle D_1.$ Then $\{ C_1, D_2, D_3 \}$ is seen to be a $T$-SRCT for $(M,c).$  Thus if a) is true, then Theorem \ref{the-NoRainbowCircuit} i.1) holds for $(M,c).$
Suppose b) is true.  Let $x\in T.$  Then $(M_1'' - s_1, c_1'')$ has a $T$-SRCP $\{ D_1, D_2 \}$ where $x \not\in D_1 \cup D_2$ and $\{ D_1, D_2 \}$ is also seen to be a $T$-SRCP in $(M-e, c).$  In this case, Theorem \ref{the-NoRainbowCircuit} i.2) holds for $(M,c).$

\ssms
\noindent{\it Case 1.2.2}  $e \in E(M_1').$
\ssms 

We have that $M_1'$ contains exactly two colour-singular elements, one which is $e$ and the other being say $f.$  Then $(M_2',c_2')$ has exactly one colour-singular element, namely $f'.$
As in Case 1.2.1, there exists $j\in [3]$ for which $(M_2',c_2')$ has an $s_j$-rainbow circuit $C$ where $|C| \le n_2.$  We may assume that this true for $j=1.$
Extend the colouring $c_1'$ to a colouring $c_1''$ of $M_1'' = M_1' + s_1$ where $c_1''(s_1) := c(f).$  If $(M_1'', c_1'')$ contains a rainbow circuit, say $D$, then $s_1 \in D$ and $C \triangle D$ is seen to be a rainbow cycle in $(M,c).$  Thus $(M_1'',c_1'')$ is circuit-achromatic.  Suppose that $M_1''$ is not simple.  Then there exists $x \in E(M_1')$ where $x$ is parallel with $s_1.$  Since $(M_1'', c_1'')$ is circuit-achromatic, $c_1''(x) = c_1''(s_1) = c(f).$  Thus $x=f.$  Now redefine $M_i\ i = 1,2$ so that $f\in E(M_2).$  In this case, $\varepsilon(M_2') = 2n_2-2$ and $|c_2'(M_2')| = n_2-1$ and we can revert back to Case 1.1 with the roles of $M_1$ and $M_2$ switched.  Thus we may assume that $M_1''$ is simple.
By assumption, Theorem \ref{the-NoRainbowCircuit} i) holds for $(M_1'',c_1'')$ and $T.$  Suppose Theorem \ref{the-NoRainbowCircuit} i.1) holds for $(M_1'', c_1'').$ Let $\{ D_1, D_2, D_3 \}$ be a $T$-SRCT for $(M_1'', c_1'').$
If $s_1 \not\in \bigcup_{i=1}^3 D_i$, then $\{ D_1, D_2, D_3 \}$ is seen to be a $T$-SRCT for $(M,c).$
Suppose $s_1 \in \bigcup_{i=1}^3D_i$ where we may assume $s_1 \in D_1.$ Letting $D_1' = D_1\triangle C$ it is seen that $\{ D_1', D_2,D_3 \}$ is a $T$-SRCyT for $(M,c).$  Thus we see that Theorem \ref{the-NoRainbowCircuit} i.1) holds for $(M,c).$
Suppose Theorem \ref{the-NoRainbowCircuit} i.2) holds for $(M_1'', c_1'').$
Note that $e$ is the unique colour-singular element in $(M_1'', c_1'')$.  Let $x\in T.$  Then there is a $T$-SRCP $\{ D_1, D_2\}$ for $(M_1''-e,c_1'').$  If $s_1 \not\in D_1 \cup D_2,$ then $\{ D_1, D_2 \}$ is seen to be a $T$-SRCP for $(M -e,c).$  If $s_1 \in D_1 \cup D_2$ and  say $s_1 \in D_1,$ then let $D_1' = D_1 \triangle C.$  Then it is seen that $\{ D_1', D_2 \}$ is a $T$-SRCyP for $(M - e,c)$ where $x\not\in D_1' \cup C_2.$  Thus Theorem \ref{the-NoRainbowCircuit} i.2) holds for $(M,c).$  

\sms
\noindent{\bf Case 2}:  Suppose $\varepsilon_2' = 2n_2 -2.$ 

\ssms
\noindent{\it Case 2.1} $(M_2', c_2')$ is type 1.
\ssms

Then $\varepsilon_1' = 2n_1-3$ and $|c(M_1)| = n_1 -1.$  Furthermore, since $|c_2'(M_2')| = n_2-1$ we have that $c_2'$ is $2$-uniform and $e \in E(M_1').$  By assumption, Theorem \ref{the-NoRainbowCircuit} iii) holds for $(M_2',c_2')$ and $S$.  Thus $(M_2',c_2')$ has an $S$-SRCP $\{ C_1,C_2\}$ where  we may assume that for $i = 1,2,$ $s_i \in C_i.$
Let $M_1'' = M_1' + s_1 + s_2$ and extend the colouring $c_1'$ to a colouring $c_1''$ of $M_1''$ where $c_1''(s_i):= 0,\ i = 1,2.$  If $(M_1'',c_1'')$ contains a rainbow circuit, say $C$, then $C' \cap \{ s_1, s_2 \} \ne \emptyset.$  Assuming $s_1 \in C,$ then $C \triangle C_1$ would be a rainbow cycle in $M$, a contradiction.  Thus $(M_1'',c_1'')$ is circuit-achromatic and hence $M_1''$ is simple.  Given that $\varepsilon(M_1'') = 2n_1 -1,$ it follows by assumption that Theorem \ref{the-NoRainbowCircuit} i) holds for $(M_1'',c_1'')$ and $T.$ 
Suppose Theorem \ref{the-NoRainbowCircuit} i.1) holds for $(M_1'',c_1'').$  Then $(M_1'',c_1'')$ has an $T$-SRCT $\{ D_1, D_2, D_3 \}$ where we may assume that for all $i\in [2],$ if $s_i \in \bigcup_{j=1}^3D_j$, then $s_i \in D_i$.  For $i = 1, 2,$ let $D_i' = \left\{ \begin{array}{lr} D_i & \mathrm{if}\ s_i \not\in D_i\\ C_i \triangle D_i & \mathrm{if}\ s_i \in D_i. \end{array} \right .$  Since $|C_1| + |C_2| \le r(M_2) + 2 = n_2 +2$ and for all $1\le i<j\le 3$, $|D_i| + |D_j| \le r(M_1) + 2 = n_1 +2$, it follows that
\begin{align*}
|D_1'| + |D_2'| &\le |D_1| + |D_2| + |C_1| + |C_2| -4\\
&\le n_1 + n_2 + 4 -4\\
&= n + 2.
\end{align*}
Furthermore, it is seen that for $i = 1,2,$ $|D_i'| + |D_3| \le n+2.$
Thus $\{ D_1', D_2', D_3 \}$ is seen to be a $T$-SRCyT for $(M,c)$ and Theorem \ref{the-NoRainbowCircuit} i.1) holds for $(M,c).$
If instead Theorem \ref{the-NoRainbowCircuit} i.2) holds for $(M_1'',c_1'')$, then one can use similar arguments to show that Theorem \ref{the-NoRainbowCircuit} i.2) holds for $(M,c)$ as well.

\ssms
\noindent{\it Case 2.2}  $(M_2', c_2')$ is type 2.
\ssms

We have $|c_2'(M_2')| = n_2$ and $(M_2',c_2')$ contains exactly two colour-singular elements, say $f$ and $g.$  We also have that $\varepsilon_1' = 2n_1-3$ and $(M_1',c_1')$ has one or three colour-singular elements.

\ssms
\noindent{\it Case 2.2.1} $e \in E(M_1').$ 
\ssms

We have that $(M_1',c_1')$ has exactly three colour-singular elements being $e, f',$ $g'$ and $|c_1'(M_1')| = n_1.$ Let $c_2''$ be a colouring of $M_2'$ where $c_2''(g) := c_2'(f)$ and $\forall x\in E(M_2') - g, \ c_2''(x) := c_2'(x).$
By assumption, Theorem \ref{the-NoRainbowCircuit} iii) holds for $(M_2',c_2'')$ and $S$ and hence there exists an $S$-SRCP $\{ C_1, C_2\}$ for $(M_2',c_2'')$ where we may assume that for $i = 1,2,$ $s_i \in C_i.$  
Extend $c_1'$ to a colouring $c_1''$ for $M_1'' = M_1' + s_1 + s_2$ where $c_1''(s_i),\ i = 1,2$ are defined as follows.  If $\{ f, g \} \cap (C_1 \cup C_2) = \emptyset,$ then define $(c_1''(s_1), c_1''(s_2)) : = (c(f), c(g)).$  If  $\{ f, g \} \cap (C_1 \cup C_2) \ne \emptyset$
then define
$$(c_1''(s_1), c_1''(s_2)) : = \left\{ \begin{array}{lr} (c(f), c(g)) & \mathrm{if}\ f\in C_1\  \mathrm{or} \ g \in C_2\\ (c(g), c(f)) & \mathrm{if}\ g \in C_1 \ \mathrm{or}\ f \in C_2\end{array} \right.$$
We claim that $(M_1'',c_1'')$ is circuit-achromatic.  For if it has a rainbow circuit, say $D$, then $D\cap \{ s_1, s_2 \} \ne \emptyset.$  If $|D\cap \{ s_1, s_2 \}| = 1,$ and say $s_1 \in D,$ then $D\triangle C_1$ is a rainbow cycle of $M$, a contradiction.  If $\{ s_1, s_2 \} \subset D,$ then let $D' = D \triangle S = D - s_1 - s_2 + s_3.$  Since $|c(M_2')| = n_2,$ $M_2' + s_3$ has an $s_3$-rainbow circuit $C$ where $s_3 \in C.$ Noting that $\{ f', g' \} \cap D' = \emptyset,$ it follows that $C\triangle D'$ is a rainbow cycle of $M$, a contradiction.   Thus $(M_1'',c_1'')$ is circuit-achromatic.

  By symmetry, we may assume that $(c_1''(s_1), c_1''(s_2)) = (c(f), c(g)).$  We may assume that $M_1''$ is simple.  For if not, there is an element $x\in E(M_1')$ which is parallel with $s_1$ or $s_2.$  In this case, redefine $M_1$ and $M_2$ so that $x\in E(M_2')$.  Now $\varepsilon_2' = 2n_2-1$ and we can revert back to Case 0.  Thus we may assume that $M_1$ is simple.
By assumption, Theorem \ref{the-NoRainbowCircuit} i) holds for $(M_1'',c_1'')$.  Suppose Theorem \ref{the-NoRainbowCircuit} i.1) holds for $(M_1'',c_1'')$ and $T$.   Then there exists a $T$-SRCT $\{ D_1, D_2, D_3 \}$ for $(M_1'', c_1'').$  Let $D = \bigcup_iD_i.$  If $\{ s_1, s_2 \} \cap D = \emptyset$, then $\{ D_1, D_2, D_3 \}$ is seen to be a $T$-SRCT for $(M,c)$.  
 Suppose $\{ s_1, s_2 \} \cap  D \ne \emptyset.$  If $|\{ s_1, s_2 \}  \cap D| =1$ and say, $s_1 \in D_1$, then let $D_1' = C_1 \triangle D_1.$  Now $\{ D_1', D_2, D_3 \}$ is seen to be a $T$-SRCyT for $(M,c).$  Suppose $| \{ s_1, s_2 \} \cap D| =2.$  Suppose two of the circuits $D_i,\ i=1,2,3$ intersect $\{ s_1, s_2 \}.$  Then we may assume that $s_i\in D_i,\ i = 1,2.$  Let $D_i' = C_i \triangle D_i,\ i = 1,2.$   Given that $\{ C_1,C_2\}$ is an $S$-SRCP for $(M_2', c_2')$, 
 we have
 \begin{align*}
 |D_1'| + |D_2'| &= |D_1| + |D_2| + |C_1| + |C_2| -4\\
&\le n_1 + 2 + n_2 + 2 - 4\\
&= n+2.
 \end{align*}
 Furthermore, it is seen that for $i= 1,2,$ $|D_i'| + |D_3| \le n+2.$
 Thus $\{ D_1', D_2', D_3 \}$ is a $T$-SRCyT for $(M,c).$  Suppose instead that for some $i,$ $\{ s_1, s_2 \} \subseteq D_i.$  We may assume that
 this is true for $i=1.$  
 Similar to before, there is an $s_3$-rainbow circuit $C$ for $(M_2', c_2'')$ where $|C| \le n_2.$  Let $D_1' = D_1 \triangle S \triangle C.$  We observe that $\{ f', g' \} \cap D_1' = \emptyset$ and thus $D_1'$ is a rainbow cycle in $(M,c).$
It is now seen that  $\{ D_1', D_2, D_3 \}$ is a $T$-SRCyT for $(M,c).$
 It follows from the above that if Theorem \ref{the-NoRainbowCircuit} i) holds for $(M_1'',c_1''),$ then Theorem \ref{the-NoRainbowCircuit} i.1) holds for $(M,c).$
 If instead Theorem \ref{the-NoRainbowCircuit} i.2) holds for $(M_1'',c_1''),$ then using the similar arguments as above, one can show that Theorem \ref{the-NoRainbowCircuit} i.2) holds for $(M,c)$ as well.  

\ssms
\noindent{\it Case 2.2.2} $e \in E(M_2').$ 
\ssms

 We have that $(M_2', c_2')$ has exactly two colour-singular elements, one which is $e$ and the other which we denote by $f.$  It follows that $|c_2'(M_2')| = n_2$ and $f'$ is the unique colour-singular element of $(M_1', c_1').$
 Let $c_2''$ be a colouring of $M_2'$ where $c_2''(f):=c(e)$ and for all $x \in E(M_2')-f,\ c_2''(x):=c(x).$
By assumption, Theorem \ref{the-NoRainbowCircuit} iii) holds for $(M_2', c_2'')$ and $S$ hence there is an $S$-SRCP $\{C_1, C_2\}$ for $(M_2', c_2'')$ where for $i= 1,2,$ $s_i \in C_i.$  
Extend the colouring $c_1'$ to a colouring $c_1''$ of $M_1'' = M_1 + s_1 + s_2$ where $(c_1''(s_1), c_1''(s_2))= (c(e), c(f)),$ if $\{ e,f \} \cap (C_1 \cup C_2) = \emptyset$; otherwise, $$(c_1''(s_1), c_1''(s_2)):= \left\{\begin{array}{ll} (c(e),c(f)) &\mathrm{if}\ e\in C_1\ \mathrm{or}\ f \in C_2\\ (c(f),c(e)) &\mathrm{if}\ f\in C_1\ \mathrm{or}\ e \in C_2 \end{array}\right.$$
%
%
Suppose $(M_1'', c_1'')$ has a rainbow circuit, say $D$. Then $D\cap \{ s_1, s_2 \} \ne \emptyset.$  If for some $i\in [2],$ $D\cap \{ s_1, s_2 \} = \{ s_i \},$ then $D\triangle C_i$ is seen to be a rainbow cycle in $M$.  If $\{ s_1, s_2 \} \subseteq D,$ then $D' = D \triangle S = D -  s_1 - s_2  + s_3$ is an $s_3$-rainbow circuit for $(M_1',c_1')$ where $f' \not\in D'.$  
Since $|c_2'(M_2')| = n_2,$ it follows that there is an $s_3$-rainbow circuit $C'$ for $(M_2',c_2')$ and one sees that $C' \triangle D'$ is a rainbow cycle in $M$, a contradiction.  We conclude that $(M_1'', c_1'')$ is circuit-achromatic.   
%
%
As before, we may assume that $M_1''$ is simple; stronger still, we may assume that no element of $E(M_1')$ is parallel with any element of $S.$ 
By assumption,Theorem \ref{the-NoRainbowCircuit} i) holds for $(M_1'',c_1'').$  
Suppose that Theorem \ref{the-NoRainbowCircuit} i.1) holds for $(M_1'', c_1'')$ and let $\{ D_1, D_2, D_3 \}$ be a $T$-SRCT for $(M_1'', c_1'').$ We may assume that $\{ s_1, s_2 \} \cap \bigcup_i D_i \ne \emptyset.$
Suppose that for some $i\in [3],\ \{ s_1, s_2 \} \subset D_i.$  We may assume that $\{ s_1, s_2 \} \subset D_1.$ 
Let $C$ be a $s_3$-rainbow circuit in $M_2' + f'$ where $|C| \le n_2.$
Let $D_1' = D_1 \triangle S \triangle C$.  Then $\{ D_1', D_2, D_3 \}$ is seen to be a $T$-SRCyT for $(M,c).$
Suppose that for all $i\in [3],$ $|D_i \cap \{ s_1, s_2 \}| \le 1.$ We may assume that for all $i\in [2],$ if $s_i \in \bigcup_iD_i,$ then $s_i \in D_i.$  For $i= 1,2$, define $D_i' = \left\{ \begin{array}{ll} D_i \triangle C_i &\mathrm{if} \ s_i \in D_i\\ D_i &\mathrm{if} s_i \not\in D_i.\end{array} \right.$  Then $\{ D_1', D_2', D_3 \}$ is seen to be a $T$-SRCyT for $(M,c).$

Suppose that Theorem \ref{the-NoRainbowCircuit} i.2) holds for $(M_1'', c_1'')$ and $T.$  For convenience, we may assume that $c_1''(s_1) = c(e)$ and $c_1''(s_2) = c(f).$  Then $s_1$ is the unique colour-singular element in $(M_1'',c_1'').$
 Let $x \in T$ and let $\{ D_1, D_2 \}$ be a $T$-SRCP in $(M_1''-s_1, c_1'')$ where $x \not\in D_1 \cup D_2.$  Now using the circuits $C_2$ and $D_i,\ i \in [2],$ it is straightforward to construct a $T$-SRCP $\{ D_1', D_2' \}$ for $(M-e,c)$ where $x\not\in D_1' \cup D_2'.$  Thus Theorem \ref{the-NoRainbowCircuit} i.2) holds for $(M,c).$

%

\section{The proof of Theorem \ref{the-NoRainbowCircuit} ii)}
As in the previous section, let $M$ be a simple regular matroid where $\varepsilon(M) = 2n-1$ and let $c$ be a $2$-bounded $n$-colouring of $M$ where $(M,c)$ is rainbow circuit-free.  
Let $e$ be the unique singular element of $(M,c).$  Let $N$ be a regular extension of $M$ by and element $x$ where $e$ and $x$ are non-parallel.   
As before, we may assume that $N$ is connected and is neither graphic nor cographic.  Note that $M \not\simeq M(K_{3,3})$ since any $2$-bounded $5$-colouring of $M(K_{3,3})$ must contain a rainbow circuit.  As such $N \not\simeq R_{10}$ since deleting an element from $R_{10}$ results in a matroid isomorphic to $M(K_{3,3}).$ Thus $N$ is either a $2$-sum or $3$-sum of regular matroids.

%
\subsection{The case where $N = N_1 \oplus_2 N_2$}

Assume that $N$ is a $2$-sum of regular matroids, $N = N_1 \oplus_2 N_2$ where $x \in E(N_1)$ and $E(N_1) \cap E(N_2) = \{ s \}.$
Let $M_1 = N_1\backslash x$ and $M_2 = N_2.$  For $i = 1,2,$ let $n_i = r(M_i),$ $M_i' = M_i \backslash s,$ $\varepsilon_i' = \varepsilon(M_i')$ and let $c_i' = c \big| E(M_i').$   We have that $n =  n_1 + n_2 -1.$  and $\varepsilon_1' + \varepsilon_2' = 2(n_1 + n_2 -1) -1.$  Thus for some $i,$ $\varepsilon_i' \ge 2n_i-1.$  On the other hand, since for $i=1,2,$ $M_i'$ is circuit-achromatic, it follows that $\varepsilon_i' \le 2n_i-1.$  Thus for some $i,$ $\varepsilon_i' = 2n_i-1$ and $\varepsilon_{3-i}' = 2n_{3-i} -2.$  

\sms
\noindent{\bf Case 1} $\varepsilon_1' = 2n_1 -1.$
\sms

We have that $\varepsilon_2' = n_2-2$ and $(M_1',c_1')$ contains a unique colour-singular element say $f.$
Suppose $x$ and $f$ are non-parallel.   By assumption, Theorem \ref{the-NoRainbowCircuit} ii) holds for $(M_1', c_1')$ and as such there exists an $x$-semi-SRCP $\{ D_1, D_2\}$ for $(M_1', c_1').$  Now $\{ D_1, D_2\}$ is also seen to be a $x$-semi-SRCP for $(M,c).$  Thus we may assume that $x$ and $f$ are parallel.  Then $f \ne e$ and hence $e \in E(M_2'),$
$|c_2'(M_2')| = n_2,$ and $e$ and $f'$ are the colour-singular elements of $(M_2', c_2').$  Let $C_1 = \{ x, f \}.$
There is a $s$-rainbow circuit, say $C,$ for $(M_2',c_2')$ where $|C| \le n_2.$  
Since $|c_1'(M_1')| = n_1,$ there is a rainbow-subset $B\subset E(M_1') - f$ where $|B| = n_1-1.$  Now $B + x + s$ contains a circuit $D$.  If $x\not\in D,$ then $C\triangle D$ is seen to be a rainbow cycle of $(M,c).$  Thus $x\in D.$  If $s\not\in D,$ then 
$\{ C_1, D \}$ is seen to be a  $x$-semi-SRCP for $(M,c).$   If instead $s\in D,$ then $\{ C_1, C\triangle D \}$ is seen to be a $x$-semi-SRCyP for $(M,c).$

\sms
\noindent{\bf Case 2} $\varepsilon_2' = 2n_2-1$
\sms

We may assume that $x$ is not parallel to $s$;  otherwise, we may assume that $x\in E(M_2'),$ in which case we can apply the previous case with the roles of $M_1'$ and $M_2'$ reversed.
We observe that $M_1$ is simple; for otherwise, if some element $y \in E(M_1')$ is parallel to $s$, then $M_2'+y$ would be a simple matroid with $2n_2$ elements, and hence it would contain a rainbow circuit, a contradiction.
We may also assume that $M_2$ is simple.  For if there is an element $y\in E(M_2')$ which is parallel with $s$, then we can assume that $y\in E(M_1')$ and revert back to Case 1.
Extend the colouring $c_1'$ to a colouring $c_1$ of $M_1$ where $c_1(s):=0.$  

\ssms
\noindent{\it Case 2.1} $e \in E(M_2').$
\ssms

We have that $e$ is the unique singular element of $(M_2', c_2')$ and thus $s$ is the unique colour-singular element of $(M_1,c_1)$ (which is not parallel with $x$).  Since $|c_2'(M_2')| = n_2,$ there is an $s$-rainbow circuit $C$ for $(M_2',c_2').$
Now $(M_1, c_1)$ is seen to be circuit-achromatic.  For if it contained a rainbow circuit, say $D$, then $s\in D$ and $C\triangle D$ is a rainbow circuit in $(M,c)$, a contradiction.  Thus $(M_1, c_1)$ is circuit-achromatic and Theorem \ref{the-NoRainbowCircuit} ii) holds for $(M_1, c_1).$  As such
 there exists an $x$-semi-SRCP $\{ D_1,D_2 \}$ for $(M_1, c_1).$  If $s \not\in D_1 \cup D_2,$ then $\{ D_1,D_2 \}$ is an $x$-semi-SRCP for $(M,c).$  Suppose $s\in D_1\cup D_2$, where we may assume that $s\in D_1.$  Let $D_1' = C\triangle D_1.$  Then it is seen that $\{ D_1',D_2\}$ is an $x$-semi-SRCP for $(M,c).$

\ssms
\noindent{\it Case 2.2} $e \in E(M_1').$
\ssms

Let $f$ be the unique colour-singular element of $(M_2',c_2')$.  Then $e$ and $f'$ are the colour-singular elements of $(M_1',c_1').$  Extend the colouring $c_1'$ to a colouring $c_1$ of $M_1$ where $c_1(s):= c(f).$  There is an $s$-rainbow circuit $C$ for $(M_2',c_2')$ where $|C| \le n_2.$  If $(M_1,c_1)$ contains a rainbow circuit, say $D$, then $s\in D$ (and hence $f' \not\in D$) and $C\triangle D$ is seen to be a rainbow circuit of $(M,c).$  Thus $(M_1, c_1)$ is circuit-achromatic.  By assumption, Theorem \ref{the-NoRainbowCircuit} ii) holds for $(M_1, c_1)$ and hence there exists an $x$-semi-SRCP $\{ D_1,D_2 \}$ for $(M_1,c_1)$ where we may assume that $s\in D_1 \cup D_2$ (otherwise $\{ D_1,D_2 \}$ is an $x$-semi-SRCP for $(M,c)$).  In particular, we may assume that $s\in D_1.$  Let $D_1' = C \triangle D_1.$  Then $\{ D_1',D_2 \}$ is seen to be an $x$-semi-SRCP for $(M,c).$   

\subsection{The case where $N = N_1 \oplus_3 N_2$} 

Assume that $N$ is a $3$-sum of regular matroids, $N = N_1 \oplus_3 N_2$ where $E(N_1) \cap E(N_2) = S = \{ s_1, s_2, s_3 \}$ and $x\in E(N_1).$   Let $M_1 = N_1\backslash x$ and $M_2 = N_2.$  For $i = 1,2,$ let $n_i = r(M_i),$ $M_i' = M_i \backslash S,$ $\varepsilon_i' = \varepsilon(M_i'),$ and let $c_i' = c \big| E(M_i').$   We have that $n =  n_1 + n_2 -2$ and $\varepsilon(M) = 2n -1 = 2(n_1 + n_2 -2) -1$ and thus for some $i\in \{ 1,2 \},$ $\varepsilon_i' \ge 2n_i -2.$  As before, we have that for $i=1,2,$ $\varepsilon_i' \le 2n_i-1.$
We shall examine a number of cases.

\sms
\noindent{\bf Case 1}:  For some $i\in [2],$ $\varepsilon_i' = 2n_i -1.$
\sms

\ssms
\noindent{\it Case 1.1} $\varepsilon_1' = 2n_1 -1$
\ssms

We have $\varepsilon_2' = 2n_2 -4.$ 
Let $f$ be the unique singular element of $(M_1',c_1').$  Suppose $x$ and $f$ are parallel.  Then $f \ne e$ and hence $e$ and $f'$ are the colour-singular elements of $(M_2',c_2').$
Then $|c_2'(M_2')| = n_2-1$ and it follows by Lemma \ref{cl1} that for some $j\in [3]$, there exists an $s_j$-rainbow circuit $C$ for $(M_2', c_2')$ where $|C| \le n_2.$  We may assume this is true for $j=1.$  Let $C_1 = \{ x,f \}.$
Since $|c_1'(M_1')| = n_1,$ there is a rainbow subset $B \subset E(M_1')-f$ where $|B| = n_1-1.$  Then $B+ x + s_1$ contains a circuit $D.$  Clearly $x\in D$ for otherwise, $s_1 \in D$ and $C\triangle D$ is a rainbow cycle.  If $s_1 \not\in D,$ then $\{ C_1, D \}$ is seen to be a
$x$-semi-SRCP for $(M,c).$  Otherwise, if $s_1 \in D,$ then $\{ C_1, C\triangle D \}$ is a $x$-semi-SRCyP for $(M,c).$
%

Suppose that $x$ and $f$ are non-parallel.  Then by assumption, Theorem \ref{the-NoRainbowCircuit} ii) holds for $(M_1',c_1')$ and hence there is an $x$-semi-SRCP for
$(M_1',c_1')$ which is also seen to be an $x$-semi-SRCP for $(M,c).$  

\ssms
\noindent{\it Case 1.2}: $\varepsilon_2' = 2n_2 -1.$
\ssms

We have $\varepsilon_1' = 2n_1-4.$  We may assume that $x$ is not parallel with any element of $S$, for otherwise, we may just assume that $x \in E(N_2)$ and revert back to Case 1.1, with the roles of $M_1$ and $M_2$ interchanged. Let $f$ be the unique colour-singular element of $(M_2',c_2').$  By assumption, Theorem \ref{the-NoRainbowCircuit} i) holds for $(M_2',c_2')$ and thus for all $1\le i < j \le 3$, there exists an $S$-SRCP $\{D_1, D_2\}$ for $(M_2', c_2')$ where
$s_i \in D_1$ and $s_j \in D_2.$  We also note that $M_1$ must be simple.  For if not, then there is an element $y \in E(M_1')$ which is parallel to an element of $S$.  As such, $M_2' + y$ is a simple matroid having $2n_2$ elements and thus has a rainbow circuit,
a contradiction.  Thus $M_1$ is simple.  

\ssms
\noindent{\it Case 1.2.1}:  $f = e.$
\ssms 

Assume that if $e$ is parallel with an element of $S$, then it is parallel with $s_3.$  Extend the colouring $c_1'$ to a colouring $c_1$ of $M_1$ where for $i=1,2,$ $c_1(s_i):= a$ and $c_1(s_3):= b$ (and $a,b \not\in c(M)$).  Suppose that $(M_1,c_1)$  contains a rainbow circuit $C$.  Then 
$C\cap S \ne \emptyset.$  Suppose for some $i$, $C\cap S = \{ s_i \}.$  Then given that there is an $s_i$-rainbow circuit $D$ for $(M_2',c_2'),$ $C\triangle D$ is a rainbow cycle of $(M,c)$, a contradiction.  Suppose $|C \cap S| =2.$ We may assume that
$C\cap S = \{ s_1, s_3 \}.$  Given that there is an $s_2$-rainbow circuit $D$ for $(M_2',c_2'),$ it is seen that $C \triangle S \triangle D$ is a rainbow cycle for $(M,c).$  It follows that $(M_1,c_1)$ is circuit-achromatic.

We observe that $s_3$ is the unique colour-singular element of $(M_1,c_1).$  By assumption, Theorem \ref{the-NoRainbowCircuit} ii) holds for $(M_1, c_1)$.  Since $x$ and $s_3$ are non-parallel, there is an $x$-semi-SRCP $\{ C_1',C_2' \}$ for $(M_1,c_1).$  If $S \cap (C_1' \cup C_2') = \emptyset,$ then $\{C_1',C_2'\}$ is an $x$-semi-SRCP for
$(M,c).$  Suppose that for some $i$, $C_i' \cap S \ne \emptyset$ and $C_{3-i}' \cap S = \emptyset.$  We may assume this is true for $i=1$ and $s_1 \in C_1'.$  Suppose $C_1' \cap S = \{ s_1 \}.$  Let $D$ be an $s_1$-rainbow circuit for $(M_2',c_2')$ where $|D| \le n_2.$
Letting $C_1 = C_1' \triangle D,$ we see that $\{C_1, C_2'\}$ is an $x$ - semi-SRCyP for $(M,c).$   Suppose instead that $|C_1' \cap S| =2.$  Then $C_1'\cap S = \{ s_1, s_3 \}.$    Let $D$ be an $s_2$-rainbow circuit for $(M_2',c_2')$ where $|D| \le n_2.$  Letting $C_1 = C_1' \triangle S \triangle D$ we see that $\{C_1, C_2'\}$ is an $x$ - semi-SRCyP for $(M,c).$ 

Finally, suppose that for $i= 1,2$, $C_i' \cap S \ne \emptyset.$  Assume first that for $i=1,2,$ $|C_i' \cap S| =1$.  We may assume that $C_i' \cap S = \{ s_i \}, \ i = 1,2.$   Let $\{ D_1, D_2\}$ be a $S$-SRCP for $(M_2',c_2')$ where $s_i \in D_i,\ i = 1,2.$
For $i=1,2,$ let $C_i = C_i' \triangle D_i.$  Then $\{ C_1,C_2\}$ is seen to be an $x$-semi-SRCP for $(M,c).$  Assume now that for some $i$, $|C_i'\cap S| = 2.$  We may assume that $\{ s_1, s_3 \} \subset C_1'$ and $s_2 \in C_2'.$
By assumption, Theorem \ref{the-NoRainbowCircuit} ii) holds for $(M_2', c_2')$ and (noting that $e$ and $s_2$ are non-parallel) there is an $s_2$-semi-SRCP $\{ D_1,D_2\}$ for $(M_2',c_2').$  Letting $C_1 = C_1' \triangle S \triangle D_1$ and $C_2 = C_2' \triangle D_2$, we see that $\{ C_1,C_2 \}$ is an $x$-semi-SRCP for $(M,c).$

\ssms
\noindent{\it Case 1.2.2}:  $f \ne e.$
\ssms 

We have that $e$ and $f'$ are the colour-singular elements of $(M_1',c_1').$  We largely repeat our arguments from the previous case, with some modifications.  As noted above, Theorem \ref{the-NoRainbowCircuit} i.1) or i.2)  holds for $(M_2',c_2')$ and $S$.
If i.1) holds, then there is an $S$-SRCT $\{ D_1,D_2,D_3\}$ for $(M_2',c_2').$  In this case, we may assume that if for some $i$, $f \in D_i$, then $s_1 \in D_i.$  Notably, regardless of whether Theorem \ref{the-NoRainbowCircuit} i.1) or i.2)  holds for $(M_2',c_2')$, for $i=2,3$, there is a $S$-SRCP $\{ D_1,D_2 \}$ for $(M_2',c_2')$ where $s_2 \in D_1,\ s_3\in D_2$ and $f \not\in D_1 \cup D_2.$  In particular, this implies that for all $i\in \{ 2,3 \},$ there is an
$s_i$-rainbow circuit $D$ for $(M_2',c_2')$ where $f \not\in D$ and $|D| \le n_2.$

Extend the colouring $c_1'$ to a colouring $c_1$ of $M_1$ where $c_1(s_1):= c(f)$ and $c_1(s_i):= 0,\ i = 2,3.$  Then $e$ is the unique colour-singular element of $(M_1,c_1)$.  Using similar arguments as in the previous case, one can show that $(M_1,c_1)$ is circuit-achromatic.  Given that $M_1$ is simple, it follows
by assumption that Theorem \ref{the-NoRainbowCircuit} ii) holds for $(M_1,c_1)$ and consequently there is an $x$-semi-SRCP $\{C_1',C_2'\}$ for $(M_1,c_1).$  We may assume that $S\cap (C_1' \cup C_2') \ne \emptyset$, for otherwise, $\{ C_1',C_2'\}$ is an $x$-semi-SRCP for $(M,c).$

\ssms
\noindent{\it Case 1.2.2.1} For some $i$, $C_i' \cap S \ne \emptyset$ and $C_{3-i}' \cap S = \emptyset.$  
\ssms

We may assume that this is true for $i=1.$  Suppose first that $|C_1' \cap S| =1.$  If $C_1' \cap S = \{ s_1 \},$ the let $D$ be an $s_1$-rainbow circuit for $(M_2', c_2')$ where $|D| \le n_2.$
Then one sees that for $C_1 = C_1' \triangle D,$ $\{ C_1,C_2' \}$ is an $x$-semi-SRCP for $(M,c).$  Suppose instead that for some $j\in \{ 2,3 \}, \ C_1' \cap S = \{ s_j \}.$  We may assume $s_2 \in C_1'.$ 
Let $D$ be an $s_2$-rainbow circuit for $(M_2', c_2')$ where $f \not\in D$ and $|D| \le n_2.$  Letting $C_1 = C_1'\triangle D$ we see that $\{C_1,C_2' \}$ is an $x$-semi-SRCyP for $(M,c).$
Finally, suppose $|C_1' \cap S| = 2.$  We may assume that $C_1' \cap S = \{ s_1, s_2 \}.$  Let $D$ be an $s_3$-rainbow circuit for $(M_2',c_2')$ for which $f \not\in D$ and $|D| \le n_2.$  Letting $C_1 = C_1' \triangle S \triangle D$, we see that
$\{ C_1,C_2' \}$ is an $x$-semi-SRCP for $(M,c).$

\ssms
\noindent{\it Case 1.2.2.2} For $i= 1,2,$ $C_i' \cap S \ne \emptyset.$  
\ssms

Suppose first that $s_1 \not\in C_1'\cup C_2'.$  We may assume that $s_2 \in C_1'$ and $s_3\in C_2'.$  As previously noted, there is an $S$-SRCP $\{ D_1, D_2 \}$ for $(M_2',c_2')$ where $s_2\in D_1, s_3\in D_2,$ and $f \not\in D_1 \cup D_2.$
Letting $C_i = C_i' \triangle D_{i+1},\ i = 1,2,$ we see that $\{ C_1,C_2\}$ is an $x$-semi-SRCyP for $(M,c).$ 
Thus we may assume that $s_1 \in C_1'$ and $s_2 \in C_2'.$  Suppose that $s_3 \not\in C_1'.$  By assumption, there is an $S$-SRCP $\{D_1,D_2\}$ for $(M_2',c_2')$ where $s_i\in D_i,\ i = 1,2$ and if $f \in D_1 \cup D_2$, then $f \in D_1.$
Letting $C_i = C_i' \triangle D_i,\ i = 1,2$, we see that $\{ C_1,C_2 \}$ is an $x$-semi-SRCyP for $(M,c).$  Suppose instead that $s_3\in C_1'.$   If $f$ is parallel with $s_2$, then given that there is an $s_2$-rainbow circuit $D$ for $(M_2',c_2')$ where $f \not\in D,$ it would follow that $D-s_2+f$ is a rainbow circuit.  Thus $f$ and $s_2$ are non-parallel. 
By assumption, there exists an $s_2$-semi-SRCP $\{ D_1, D_2 \}$ for $(M_2',c_2')$ where we may assume that $f \not\in D_2.$  Letting $C_1 = C_1' \triangle S \triangle D_1$ and $C_2 = C_2' \triangle D_2$, we see that $\{ C_1,C_2 \}$ is an $x$-semi-SRCyP in $(M,c).$  
 
\sms
\noindent{\bf Case 2}:  For some $i\in [2],$ $\varepsilon_i' = 2n_i -2.$
\sms

\ssms
\noindent{\it Case 2.1} $\varepsilon_1' = 2n_1 -2$
\ssms

We have that $\varepsilon_2' = 2n_2 -3.$  By Claim \ref{cl1}, there is a $j\in [3]$ for which there is an $s_j$-rainbow circuit $C$ for $(M_2',c_2')$ where $|C| \le n_2.$  We may assume this is true for $j=1$; that is, $s_1 \in C.$

\ssms
\noindent{\it Case 2.1.1} $(M_1',c_1')$ is type 1.
\ssms

We have that $(M_1',c_1')$ has no colour-singular elements.  Extend the colouring $c_1'$ to a colouring $c_1''$ of $M_1'' = M_1' + s_1$ where $c_1''(s_1):= 0.$  If $(M_1'', c_1'')$ contains a rainbow circuit, say $D$,
then $s_1 \in D$ and hence $C\triangle D$ is a rainbow circuit of $(M,c)$.  Thus $(M_1'',c_1'')$ is circuit-achromatic and hence $M_1''$ is also simple.  We see that $s_1$ is the unique colour-singular element of $(M_1'', c_1'').$  By assumption, Theorem \ref{the-NoRainbowCircuit} ii) holds for $(M_1'', c_1'')$.  Suppose first that $x$ and $s_1$ are non-parallel.   Then there exists an $x$-semi-SRCP $\{ D_1,D_2 \}$ for $(M_1'', c_1'').$  If $s_1 \not\in D_1 \cup D_2,$ then $\{ D_1,D_2\}$ is seen to be an $x$-semi-SRCP  for $(M,c).$  Thus we may assume that $s_1 \in D_1 \cup D_2,$ and in particular, $s_1\in D_1.$  Let $D_1' = C\triangle D_1.$  Then $\{ D_1',D_2\}$ is seen to be an $x$-semi-SRCyP for $(M,c)$.
Suppose now that $x$ and $s_1$ are parallel. 
We have that $e\in E(M_2')$ and $e$ and $s_1$ are non-parallel.  If for some $j\in \{ 2,3 \}$ there is an $s_j$-rainbow circuit, then we could use this in place of $C$ in the previous argument, noting that $x$ is not parallel to $s_j.$
Thus we may assume that for $j\in \{ 2,3 \},$ there is no $s_j$-rainbow circuit in $(M_2', c_2').$  Because of this, $M_2'' = M_2' + s_2 + s_3$ is simple.  Extend the colouring $c_2'$ to a colouring $c_2''$ of $M_2''$ where for $j=2,3,$ $c_2''(s_j) := 0.$
Now $(M_2'', c_2'')$ is a simple, coloured matroid with $2n_2-1$ elements which is also circuit-achromatic. 
 Given that $x$ is parallel with $s_1$, $M_2'' + x$ is a regular extension of $M_2''$ and hence Theorem \ref{the-NoRainbowCircuit} ii) holds for $(M_2'', c_2'').$  Thus there
exists a $x$-semi-SRCP $\{ C_1', C_2' \}$ for $(M_2'', c_2'').$  If $\{ s_2, s_3 \} \cap (C_1' \cup C_2') = \emptyset,$ then $\{ C_1', C_2' \}$ is seen to be a $x$-semi-SRCP for $(M,c).$  Suppose $s_2 \in C_1' \cup C_2'$, where we may assume that $s_2 \in C_1'.$  Then
$C_1'' = C_1' \triangle \{ x , s_1 \} \triangle S = C_1'' - x - s_2 + s_3$ is seen to be a $s_3$-rainbow circuit in $(M_2'', c_2'')$, contradicting our assumptions.  A similar argument can be used if $s_3 \in C_1' \cup C_2'.$
It follows from the above that $(M,c)$ has an $x$-semi-SRCP. 

\ssms
\noindent{\it Case 2.1.2} $(M_1',c_1')$ is type 2.
\ssms

We have that $(M_1',c_1')$ contains two singular elements, say $f$ and $g.$

\ssms
\noindent{\it Case 2.1.2.1}: $e \in E(M_1').$
\ssms
 
 We may assume that $g = e.$   Extend the colouring $c_1'$ to a colouring $c_1''$ of $M_1'' = M_1' + s_1$ where $c_1''(s_1) := c(f).$  As in Case 1.2.1, $(M_1'', c_1'')$ is circuit-achromatic.  Note that $e$ is the unique colour-singular element of $(M_1'',c_1'').$
 Suppose first that $M_1''$ is simple.  Then by assumption, Theorem \ref{the-NoRainbowCircuit} ii) holds for $(M_1'', c_1'')$ and thus there is an $x$-semi-SRCP $\{ D_1,D_2 \}$ for $(M_1'', c_1'').$  We may assume $s_1 \in D_1 \cup D_2,$ (for otherwise, $\{ D_1,D_2\}$ is an $x$-semi-SRCP for $(M,c)$) and in particular, $s_1 \in D_1.$  Let $D_1' = C\triangle D_1.$ Then it is seen that $\{ D_1',D_2 \}$ is an $x$-semi-SRCyP for $(M,c)$.  
 
 Suppose instead that $M_1''$ is not simple.  Then there is an element $y \in E(M_1')$ which is parallel with $s_1$.  If $y \ne f,$ then $C' = C - s_1 +y$ is seen to be a rainbow circuit of $(M,c)$, a contradiction.  Thus $y = f$ (and $f' \in C$).  Let $M_2'' = M_2' + f$ and extend the colouring $c_2'$ to a colouring $c_2''$ for 
 $M_2''$ where $c_2''(f) := c(f).$  Then $M_2''$ is simple, $\varepsilon(M_2'') = 2n_2 -2$, and $(M_2'',c_2'')$ is circuit-achromatic.  By assumption, Theorem \ref{the-NoRainbowCircuit} iii) holds for $(M_2'', c_2'')$ and thus there exists an $S$-SRCP $\{ C',C''\}.$
 We may assume that $s_2 \in C'$  If $f \not\in C'$, then we can use $C'$ in place of $C$ in the previous arguments (noting that $f$ is not parallel with $s_2$).  Thus we may assume $f\in C'$ and hence $f' \not\in C'.$ 
 However, $D' = C' \triangle \{ f, s_1 \} \triangle S = C' - f - s_2+ s_3$ is seen to be an $s_3$-rainbow circuit for $(M_2',c_2')$ and we can use $D'$ in place of $C$ in the previous arguments.  It follows from the above that $(M,c)$ has an $x$-semi-SRCP.
 

\ssms
\noindent{\it Case 2.1.2.2}: $e \in E(M_2').$
\ssms

We see that $e$, $f'$ and $g'$ are the colour-singular elements of $(M_2',c_2')$ and $|c_2'(M_2')| = n_2.$  At least one of $f$ or $g$ is not parallel with $x$, and we may assume that $g$ is such.
By Lemma \ref{cl1}, there exists $j \in [3]$ such that there is an $s_j$-rainbow circuit $C$ for $(M_2',c_2')$ where 
$g' \not\in C$ and $|C| \le n_2.$  We may assume that this is true $j=1;$ that is, $s_1 \in C.$  Let $M_1'' = M_1' + s_1$ and extend $c_1'$ to a colouring $c_1''$ of $M_1''$ where $c_1''(s_1):=c(f).$
If $(M_1'',c_1'')$ has a rainbow circuit, say $D$, then $s_1 \in D$ and $C\triangle D$ is a rainbow circuit of $(M,c),$ a contradiction (here we use the fact that $g' \not\in C$).  Thus $(M_1'',c_1'')$ is circuit-achromatic.  Suppose $M_1''$ is simple.  Then by assumption, Theorem \ref{the-NoRainbowCircuit} ii) holds for $(M_1'',c_1'')$ where we note that $g$ is the unique
colour-singular element of $(M_1'', c_1'')$ (and $x$ and $g$ are non-parallel).  Thus there an $x$-semi-SRCP $\{ D_1,D_2\}$ for $(M_1'',c_1'')$.  We may assume $s_1 \in D_1;$ for if $s_1 \not\in D_1\cup D_2$, then $\{ D_1,D_2 \}$ is an $x$-semi-SRCP for $(M,c).$  Let $D_1' = C\triangle D_1.$  Then $\{ D_1',D_2 \}$ is seen to be an $x$-semi-SRCP for $(M,c).$  

Suppose instead that $M_1''$ is not simple.  Then there is an element $y\in E(M_1')$ which is parallel with $s_1.$
Since $(M_1'',c_1'')$ is circuit-achromatic,  it must be that $y= f.$ 
Suppose $x$ and $f$ are parallel.  Let $C_1 = \{ x, f \}$ and let\\ $C_2 = C \triangle \{ f, s_1\} \triangle \{ x, f \} = C - s_1 + x.$  Then $\{ C_1,C_2 \}$ is seen to be a $x$-semi-SRCP for $(M,c).$ 
Suppose that $x$ is not parallel with $f.$ Then one could argue with the roles of $f$ and $g$ interchanged, in which case we can assume $g$ is parallel with an element of $S$ as well.  In this case, we can assume that $\{ f, g \} \subset E(M_2')$ and $\varepsilon_2' = 2n_2-1.$
We can now revert back to the proof in Case 1.

\ssms
\noindent{\it Case 2.2}: $\varepsilon_2' = 2n_2 -2.$
\ssms 

We have that $\varepsilon_1' = 2n_1 - 3.$  We may assume that $M_1$ is simple, for if not, there is an element $y\in E(M_1')$ which is parallel to an element of $S.$  In this case, we may just assume that $y\in E(M_2'),$ in which case $\varepsilon_2' = 2n_2-1$ and we may revert back to the proof of Case 1.  

\ssms
\noindent{\it Case 2.2.1}: $(M_2',c_2')$ is type 1.
\ssms

We have that $(M_2',c_2')$ has no colour-singular elements and thus $e$ is the unique colour-singular element for $(M_1',c_1').$  By assumption, Theorem \ref{the-NoRainbowCircuit} iii) holds for $(M_2',c_2')$ and $S.$  Thus there an $S$-SRCP $\{ D_1,D_2 \}$ for
$(M_2',c_2')$.  We may assume that for $i=1,2,$ $s_i \in D_i.$  Extend the colouring $c_1'$ to a colouring $c_1''$ of $M_1'' = M_1' + s_1 + s_2$ where $c_1''(s_i):=0,\ i = 1,2.$  
Seeing that $(M_1'',c_1'')$ is circuit-achromatic, Theorem \ref{the-NoRainbowCircuit} ii)
holds for $(M_1'',c_1'')$ and $x.$  Thus there is an $x$-semi-SRCP $\{ C_1',C_2' \}$ for $(M_1'', c_1'').$   Following similar arguments as before using the circuits $C_i',D_i, \ i =1,2$, one can construct an $x$-semi-SRCP for $(M,c).$

\ssms
\noindent{\it Case 2.2.2}: $(M_2',c_2')$ is type 2.
\ssms 

We have that $(M_2',c_2')$ has two colour-singular elements which we denote by $f$ and $g.$

\ssms
\noindent{\it Case 2.2.2.1}: $e \in E(M_2').$
\ssms 

We may assume that $g = e.$  Then $f'$ is the unique colour-singular element of $(M_1',c_1').$  Let $c_2''$ be the colouring obtained from $c_2'$ by recolouring $f$ so that $c_2''(f) := c(e).$   Then Theorem \ref{the-NoRainbowCircuit} iii) holds for $(M_2',c_2'')$ and $S$ and there is an $S$-SRCP $\{ D_1,D_2\}$ for $(M_2',c_2'')$.
We may assume that for $i=1,2,$ $s_i \in D_i$.  Furthermore, we may assume that if $f\in D_1 \cup D_2,$ then $f\in D_1.$  Extend the colouring $c_1'$ to a colouring $c_1''$ of  $M_1'' = M_1' + s_1 + s_2$ where $c_1''(s_1) := c(f),$ $c_1''(s_2):= 0,$ noting that $s_2$ is the unique colour-singular element of $(M_1'', c_1'').$
It is straightforward to show that $(M_1'',c_1'')$ is circuit-achromatic.  Suppose that $x$ and $s_2$ are non-parallel.  Then applying Theorem \ref{the-NoRainbowCircuit} ii) to $(M_1'', c_1''),$ there is an $x$-semi-SRCP $\{ C_1',C_2' \}$ for $(M_1'', c_1'').$  Following similar arguments as before using
the circuits $C_i',D_i,\ i = 1,2$, one can construct an $x$-semi-SRCP for $(M,c).$  If $x$ and $s_2$ are parallel, then we may assume $x\in E(M_2')$ and we can argue as in Case 1.2.2.1 with the roles of $M_1'$ and $M_2'$ switched.

\ssms
\noindent{\it Case 2.2.2.2}: $e \in E(M_1').$
\ssms 

We have that $e, f', g'$ are the colour-singular elements of $(M_1', c_1').$  Let $c_2''$ be the colouring obtained from $c_2'$ by recolouring $g$ so that $c_2''(g) := c(f).$
As in Case 2.2.2.1, there is an $S$-SRCP $\{ D_1,D_2 \}$ for $(M_2',c_2'')$ where we may assume that for $i=1,2,$ $s_i \in D_i.$
Extend the colouring $c_1'$ to a colouring $c_1''$ of $M_1'' = M_1' + s_1 + s_2$ where 
$$(c_1''(s_1), c_1''(s_2)) = \left\{ \begin{array}{ll} (c(f), c(g)) & \mathrm{if} \ f \in D_1\ \mathrm{or}\ g \in D_2\\ (c(g), c(f)) & \mathrm{otherwise} \end{array} \right.$$
Then $e$ is the unique colour-singular element of $(M_1'', c_1'')$ and it is straightforward to show that $(M_1'', c_1'')$ is circuit-achromatic.  
 By assumption, there is an $x$-semi-SRCP $\{ C_1', C_2'\}$ for $(M_1'', c_1'')$.  Following similar arguments as before using the circuits $C_i', D_i,\ i = 1,2,$ one can construct 
an $x$-semi-SRCP for $(M,c).$

\section{The proof of Theorem \ref{the-NoRainbowCircuit} iii)}\label{sec-ProofNoRainbowCircuitRegulariii)}

To begin with, we shall need the following lemma:

\begin{lemma}
Let $N$ be a simple, regular matroid where $\varepsilon(N) = 2r(N) -3$ and let $N'$ be a regular extension of $N$ by a set $S = \{ s_1, s_2, s_3 \}$ which is a $3$-circuit of $N'.$ Let $c$ be a $2$-bounded colouring of $N$ where $(N,c)$ is circuit-achromatic.  Then $(N,c)$ has one or three colour-singular elements and 
\begin{itemize}
\item[i)] for some $j\in [3]$ there exists an $s_j$-rainbow circuit for $(N,c),$ and  
\item[ii)]  if $(N,c)$ has three colour-singular elements
$f_i,\ i=1,2,3$ and $N' \big| S \cup \{ f_1, f_2, f_3 \} \not\simeq M(K_4),$  
then for some $j\in [3]$ there is an $s_j$- rainbow circuit $C$ for $(N,c)$
 where $|C \cap \{ f_1, f_2 , f_3 \}| \le 1.$  Moreover, $C$ can be chosen so that $|C| \le r(N) -1.$
\end{itemize}
\label{lem-helplemma} 
\end{lemma}

\begin{proof}
Since $(N,c)$ is circuit-achromatic, it follows that $|c(N)| \le r(N).$  Thus if $(N,c)$ has $\xi$ colour-singular elements then $\xi$ is odd (since $\varepsilon(N) = 2r(N) -3$ is odd) and $|c(N)| = \frac {2r(N) - 3 - \xi}2 + \xi \le r(N)$ and hence $\xi \le 3.$
The proof of i) follows from Lemma \ref{cl1}. 
%
%
We shall prove ii) by induction on $r(N).$  When $r(N) =3,$ ii) is seen to be true.  Suppose $r(N) = n > 3,$ and ii) is true for all regular matroids of rank at most $n-1.$  Assume that $(N,c)$ has three colour-singular elements $f_i,\ i = 1,2,3$ and let 
$F = \{ f_1, f_2, f_3 \}.$   Assume that $N' \big| F \cup S \not\simeq M(K_4).$  We may also assume that for all $j\in [3]$ and for all $e \in E(N)$ there is no parallel class $\{ e, s_j \}.$  For if such existed, then $C = \{ e, s_j \}$ would be the desired circuit.
%
By Theorem \ref{the1}, $c$ is stratified and we let $X_1, \dots ,X_{n}$ be a stratification.  We may assume that for $i=1,2,3,$ $X_i = \{ f_i \}.$  Observe that $X = X_n$ is a cocircuit of $N$ and $|X| =2.$
Let $X = \{ e, f \}.$   Let $P = N/e\backslash f,$ $P' = N'/e\backslash f,$ and let $c' = c\big| P.$  Then $(P,c')$ is simple, circuit-achromatic and $S$ is a $3$-circuit of $P'$ and $P'$ is a regular extension of $P.$  If $P' \big| F \cup S \not\simeq M(K_4),$ then by assumption, for some $j\in [3],$
there is an $s_j$-rainbow circuit $C'$ for $(P,c')$ for which 
$|C' \cap F| \le 1$ and $|C'| \le r(P) -1 = r(N)-2.$  Now it seen that either $C'$ or $C = C' + e$ is a circuit in $M'.$  If the former holds, then $C = C'$ is the desired circuit;  otherwise, $C = C'+e$ is the desired circuit.  As such, we may assume that  $P' \big| F \cup S \simeq M(K_4)$.  The same applies if we switch the roles of $e$ and $f$; that is, we may also assume that for $Q' = N'/f\backslash e$, $Q' \big| F \cup S \simeq M(K_4).$  Let $N'' = M' \big| F \cup S \cup X.$  Since $N''$ is regular and $\varepsilon(N'') = 8$,  $N''$ is graphic.  Let $G$ be a graph where $N'' = M(G).$  By assumption $G/e\backslash f \simeq K_4 \simeq G/f\backslash e.$  Given that $S$ corresponds to a triangle and the subgraph of $G$ induced by the edges of $F \cup S$ is not isomorphic to $K_4$, it is straightforward to show that the edges $e$ and $f$ must be parallel in $G$.  This contradicts our assumption that $M$ is simple.  Thus for some $s_j \in S,$ there is a $s_j$-rainbow circuit $C$ where $|C\cap F| \le 1$ and $|C| \le r(N) -1.$
%
\end{proof}

To prove Theorem \ref{the-NoRainbowCircuit} iii), we shall use induction on $r(M)$ where we assume Theorem \ref{the-NoRainbowCircuit} is true when $r(M) < n.$  
Let $M$ be a simple regular matroid where $r(M) = n$ and $\varepsilon(M) = 2n-2.$  Let $c$ be a $2$-bounded colouring of $M$ for which $(M,c)$ is circuit-achromatic.   We may assume that $|c(M)| = n-1$ (i.e. $c$ is $2$-uniform).  For if there are two colour-singular elements in $(M,c),$ then these can be lumped into one colour class.  Let $N = M + T$ be a regular extension of $M$ where $T$ is a $3$-circuit of $N$ which is co-simple. 
By the results of Sections \ref{sec-BigTheoremiii)Graphic} and \ref{sec-BigTheoremiii)Cographic}, Theorem \ref{the-NoRainbowCircuit} iii) is true for both graphic and cographic matroids. Because of this
we may assume (as before) that $N$ is a $2$- or $3$- sum of two regular matroids.  We shall only deal with the case where $N$ is a $3$-sum as the $2$-sum case is simpler and uses similar arguments.

Let $N = N_1 \oplus_3 N_2$ where $E(N_1) \cap E(N_2) = S = \{ s_1, s_2, s_3 \}.$  For $i = 1,2,$ let $(M_i', c_i'),\ n_i$ and  $\varepsilon_i'$ be defined as before.  We may assume that $T \subseteq E(N_1).$  We have $$2r(M) -2 = 2(n_1+n_2 -2) -2 = 2n-6 = \varepsilon(M) = \varepsilon_1' + \varepsilon_2'.$$   
We observe that if for some $i\in [2],\ \varepsilon_i' \ge 2n_i -3,$ then $n_i-1 \le |c_i'(M_i')| \le n_i.$  In this case, we say (as before) that $M_i'$ has type $1$ if $|c_i'(M_i')| = n_i -1$ and has type $2$ if $|c_i'(M_i)| = n_1.$
Lastly, we may assume that $\varepsilon_1' \le 2n_1-3$, for otherwise it follows by our assumptions that $(M_1',c_1')$ has a $T$-SRCP.   

\sms
\noindent{\bf Case 1} $\varepsilon_1' = 2n_1-3.$
\sms

We have that $\varepsilon_2' = 2n_2 -3$.

\ssms
\noindent{\it Case 1.1} $M_1'$ is type 1.
\ssms
 
We have that $(M_1',c_1')$ has a unique colour-singular element which is denoted by $e.$  Consequently, $e'$ is the unique colour-singular element of $(M_2',c_2')$ and hence $|c_2'(M_2')| = n_2 -1.$
By Lemma \ref{cl1}, there exists $j\in [3]$ for which there is an $s_j$-rainbow circuit $C$ for
$(M_2',c_2')$ where $|C| \le n_2.$  We claim that we may choose $C$ so that $M_1' + s_j$ is simple.  For suppose $C$ is an $s_1$-rainbow circuit where $M_1' + s_1$ is not simple.  Then there is an element $x\in E(M_1')$ which is parallel to $s_1.$  Clearly $x = e$, for otherwise, $C - s_1 +x$ is a rainbow circuit of $(M,c).$  We have for $i=2,3,$ $e$ and $s_i$ are non-parallel and hence $M_1' + s_i$ is simple.  Let $M_2'' = M_2' + e$ and extend the colouring $c_2'$ to a colouring
$c_2''$ of $M_2''$ where $c_2''(e) := c(e).$ Then $(M_2'',c_2'')$ is simple, circuit-achromatic and $\varepsilon(M_2'') = 2n_2-2.$  By assumption, Theorem \ref{the-NoRainbowCircuit} iii) holds for $(M_2'', c_2'')$ and $S$.
Thus there is an $S$-SRCP $\{ D_1,D_2 \}$ for $(M_2'', c_2'').$  For some $i$, $s_1 \not\in D_i$, and we may assume that $s_1\not\in D_1$ and $s_2\in D_1.$  If $e \not\in D_1,$ then $D_1$ is the desired circuit, since $M_1'+s_2$ is simple.
If $e \in D_1$, then $D_1 \triangle \{ e, s_1 \} \triangle S = D_1-e - s_2 +s_3$ is seen to be an $s_3$-rainbow circuit for $(M_2',c_2')$ (where $|D_1| \le n_2$) and $M_1' + s_3$ is simple.  This proves our claim.  Thus we may assume that $C$ is an $s_1$-rainbow circuit for $(M_2',c_2')$ where  $M_1' + s_1$ is simple. 

Extend the colouring $c_1'$ to a colouring $c_1''$ for $M_1'' = M_1' + s_1$ where $c_1''(s_1):= c(e).$  Now $M_1''$ is simple and $(M_1'',c_1'')$ is seen to be circuit-achromatic.
By assumption, Theorem \ref{the-NoRainbowCircuit} iii) holds for $(M_1'',c_1'')$ and $T$.  Thus there is a $T$-SRCP $\{ C_1', C_2' \}$ for $(M_1'',c_1'').$  We may assume that $s_1 \in C_1'$, for if $s_1 \not\in C_1' \cup C_2',$ then 
$\{ C_1',C_2'\}$ is a $T$-SRCP for $(M,c).$  Letting $C_1 = C_1' \triangle C$ we see that $\{ C_1,C_2' \}$ is a $T$-SRCyP for $(M,c).$

\ssms
\noindent{\it Case 1.2} $M_1'$ has type 2.
\ssms

We have that $(M_1',c_1')$ has three colour-singular elements which we denote by $f_i,\  i =1,2,3,$  Correspondingly, $|c_2'(M_2')| = n_2$ and $f_i',\ i=1,2,3$ are the colour-singular elements of $(M_2',c_2').$  Let $F = \{ f_1, f_2, f_3 \}$ and $F' = \{ f_1', f_2', f_3' \}.$
We observe that for all $i \in [3]$ and $j\in [2],$ there exists a $s_i$-rainbow circuit for $(M_j',c_j').$
Thus if for some $i\in [3],$ $C$ is an $s_i$-rainbow circuit for $(M_1', c_1')$ and $D$ is an $s_i$-rainbow circuit for $(M_2',c_2'),$ then $C\cap F \ne \emptyset$ and $D\cap F' \ne \emptyset$.  For otherwise, $C\triangle D$ is a rainbow cycle in $(M,c).$ 

\ssms
\noindent{\it Case 1.2.1}  Suppose $M_2 \big| F' \cup S \not\simeq M(K_4).$ 
\ssms

 By Lemma \ref{lem-helplemma}, there exists $j\in [3]$ for which there is an $s_j$-rainbow circuit $C$ for $(M_2',c_2')$ where $|C\cap F'| \le 1$ and $|C| \le n_2-1.$  We may assume that $s_1 \in C.$  Since $C\cap F' \ne \emptyset,$ we may assume that $C\cap F' = \{ f_1 \}.$  Suppose first that
$M_1'' = M_1' + s_1$ is simple.  Extend the colouring $c_1'$ to a colouring $c_1''$ of $M_1''$ where $c_1''(s_1) := c(f_1).$   Then $(M_1'',c_1'')$ is seen to be circuit-achromatic and by assumption Theorem \ref{the-NoRainbowCircuit} iii)
holds for $(M_1'',c_1'').$  Thus there is a $T$-SRCP $\{ C_1',C_2'\}$ for $(M_1'', c_1'')$.  We may assume that $s_1 \in C_1'$, for if $s_1 \not\in C_1'\cup C_2',$ then $\{ C_1',C_2' \}$ is a $T$-SRCP for $(M,c).$
Letting $C_1 = C_1' \triangle C$ we see that $\{ C_1,C_2'\}$ is a $T$-SRCyP for $(M,c).$  Suppose instead that $M_1' + s_1$ is not simple.  Then $f_1$ and $s_1$ are parallel.  Let $M_2'' = M_2' + f_1$ and let $c_2''$ be a colouring
of $M_2''$ where $c_2''(f_1) := c(f_1),$ $c_2''(f_3') := c(f_2'),$ and $c_2'' \big| E(M_2')- f_3 = c_2'.$   Then $M_2''$ is simple, $(M_2'',c_2'')$ is circuit-achromatic and $\varepsilon(M_2'') = 2n_2 -2.$  Thus by assumption, Theorem \ref{the-NoRainbowCircuit} iii) holds for $(M_2'',c_2'')$ and $S$ and
there is an $S$-SRCP $\{ D_1,D_2\}$ for $(M_2'', c_2'').$  Note that for $i=1,2,$ $D_i \cap \{ f_1, f_1',f_2', f_3' \} \ne \emptyset.$  We may assume that $s_2 \in D_2.$

\ssms
\noindent{\it Case 1.2.1.1}  $f_1$ is the only element of $F$ which is parallel with an element of $S.$
\ssms

Suppose $f_1 \in D_2.$  Then $D_2' = D_2 \triangle S \triangle \{ s_1, f_1 \} = D_2 -f_1 - s_2 + s_3$ is an $s_3$-rainbow circuit of $(M_2',c_2').$ Thus $D_2'$ contains exactly one of $f_2'$ or $f_3'$ and we may assume $f_2' \in D_2'.$   We can use $C = D_2'$ in the previous argument, noting that by assumption,
$f_2$ is not parallel with $s_3.$  Thus we may assume that $f_1 \not\in D_2.$  If $D_2$ contains exactly one of the elements $f_1', f_2', f_3',$ then we could use $C = D_2$ in the previous argument (as no element of $F$ is parallel with $s_2$).
Thus we may assume that $\{ f_1', f_2' \} \subset D_2.$  Let $M_1'' = M_1' + s_2$ and let $c_1''$ be a colouring of $M_1''$ where $c_1''(s_2):= c(f_2)$, $c_1''(f_3):= c(f_1)$ and $c_1'' \big| E(M_1')- f_3 = c_1'.$  Then $M_1''$ is simple. Suppose that  $(M_1'',c_1'')$ contains a rainbow circuit $D$.  Then $s_2 \in D$.  If $f_1 \not\in D,$ then $D\triangle D_2$ is a rainbow cycle of $(M,c).$  Thus $f_1 \in D$ and hence $f_3 \not\in D.$  Now $D' = D\triangle S \triangle \{ s_1, f_1 \} = D - f_1 - s_1 +s_3$ is
an $s_3$-rainbow circuit for $(M_1', c_1')$.  However, $D' \cap F = \emptyset,$ contradicting a previous observation. Thus $(M_1'',c_1'')$ is circuit-achromatic. By assumption, there is a $T$-SRCP $\{ C_1',C_2'\}$ for $(M_1'', c_1'').$  We may assume that $s_2 \in C_1',$ for 
otherwise, $\{ C_1', C_2'\}$ is a $T$-SRCP for $(M,c).$  If $f_1 \not\in C_1',$ then letting $C_1 = C_1' \triangle D_2$ we see that $\{ C_1,C_2'\}$ is a $T$-SRCyP for $(M,c).$  On the other hand, if $f_1 \in C_1',$
then $C_1'' = C_1' \triangle S \triangle \{ s_1, f_1 \} = C_1' - f_1 - s_2 + s_3$ is a circuit where $C_1'' \cap F = \emptyset.$  Let $D$ be an $s_3$-rainbow circuit for $(M_2',c_2').$  Then letting $C_1 = C_1'' \triangle D$ we see that $(C_1,C_2')$ is a $T$-SRCyP
for $(M,c).$

\ssms
\noindent{\it Case 1.2.1.2} $f_2$ or $f_3$ is parallel with an element of $S.$
\ssms

We may assume that $f_2$ and $s_2$ are parallel.  If $f_3$ is parallel with $s_3$, then $M_2 + F$ is simple, has $2n_2$ elements and hence has a rainbow circuit.  Thus $f_3$ is not parallel with $s_3.$
Consequently, $M_1'' = M_1' + s_3$ is simple.  Define a colouring $c_1''$ for $M_1''$ where
$c_1''(s_3) = c(f_3),$ $c_1''(f_2):=c(f_1)$ and $c_1'' \big| E(M_1') - f_2 = c_1'.$  Suppose $(M_1'', c_1'')$ contains a rainbow circuit $D$.  Then $s_3 \in D$ and hence $f_3 \not\in D.$
Since $D$ is an $s_3$-rainbow circuit for $(M_1',c_1')$ we have that $D \cap F \ne \emptyset.$  We may assume $f_1 \in D$ (and $f_2 \not\in D$).  Then $D' = D \triangle S \triangle \{ s_1, f_1 \} = D - s_3 - f_1 + s_2$ is seen to be an $s_2$-rainbow circuit for $(M_1',c_1')$ for which 
$D' \cap F = \emptyset,$ a contradiction.  Thus $(M_1'',c_1'')$ is circuit-achromatic.  By assumption, there is a $T$-SRCP $\{ C_1',C_2' \}$ for $(M_1'', c_1'')$.  We may assume that $s_3 \in C_1'$; otherwise,  $\{ C_1', C_2' \}$ is a $T$-SRCP for $(M,c).$    Suppose that $C_1' \cap F \ne \emptyset.$  Let $C_1= C_1' \triangle D$ where $D$ be an $s_3$-rainbow circuit for $(M_2',c_2')$ for which $|D| \le n_2.$ Then we see that $\{ C_1, C_2'\}$ is a $T$-SRCyP for $(M,c).$  Suppose instead that $C_1' \cap F \ne \emptyset.$  We may assume that $C_1' \cap F = \{ f_1 \}.$  Let $C_1'' = C_1' \triangle S \triangle \{ s_1,f_1 \} = C_1' - s_3 -f_1 + s_2$ and let $D$ be an $s_2$-rainbow circuit for $(M_2',c_2')$ where $|D| \le n_2.$  Letting $C_1 = C_1'' \triangle D$ we seen that $\{ C_1, C_2' \}$ is a $T$-SRCyP for $(M,c).$

%

\ssms
\noindent{\it Case 1.2.2}  Suppose $M_2 \big| F' \cup S \simeq M(K_4.).$ 
\ssms

In this case, let $M_1'' = M_1' + F'.$  Then $r(M_1'') = n_1 +1$ and $\varepsilon(M_1'') = \varepsilon_1' + 3 = 2n_1 = 2r(M_1'') -2.$  We note that $\varepsilon(M_1'') < \varepsilon(M)$ since $\varepsilon_2' > 3.$  Extend the colouring $c_1'$ to a colouring $c_1''$ of $M_1''$ where for $i= 1,2,3$ $c_1''(f_i') := c(f_i).$  Then $M_1''$ is simple and $(M_1'', c_1'')$ is circuit-achromatic.  By assumption there is a $T$-SRCP for $(M_1'', c_1'')$ and such is seen to be a $T$-SRCP for $(M,c).$

\sms
\noindent{\bf Case 2} $\varepsilon_1' = 2n_1-4.$
\sms

We have that $\varepsilon_2' = 2n_2-2$ and $|c_2'(M_2')| \ge n_2-1.$
By assumption, Theorem \ref{the-NoRainbowCircuit} iii) holds for $(M_2', c_2')$ and $S$.  Thus there is an $S$-SRCP $\{ D_1, D_2\}$ for $(M_2', c_2')$ where we may assume that for $i=1,2,$ $s_i \in D_i.$
Let $M_1'' = M_1' + s_1 + s_2.$

\ssms
\noindent{\it Case 2.1} $|c_2'(M_2')| = n_2 -1.$
\ssms

We have that $(M_2',c_2')$ has no colour-singular elements and the same holds for $(M_1', c_1').$  Extend the colouring $c_1'$ to a colouring $c_1''$ of $M_1'' $ where for $i= 1,2,$ $c_1''(s_i):=0.$
It is straightforward to show that $(M_1'',c_1'')$ is circuit-achromatic and hence $M_1''$ is also simple.   Since $\varepsilon(M_1'') = 2n_1-2,$ it follows by assumption that there is a $T$-SRCP $\{ C_1',C_2' \}$ for $(M_1'', c_1'').$  Now using similar arguments as before, one can construct a $T$-SRCP $\{ C_1,C_2 \}$ for $(M,c)$ using the circuits $C_i', D_i,\ i = 1,2.$

\ssms
\noindent{\it Case 2.2} $|c_2'(M_2')| = n_2 .$
\ssms

We have that $(M_2',c_2')$ contains two colour-singular elements which we denote by $f_i,\ i = 1,2$.  Then $f_i',\ i = 1,2$ are the colour-singular elements of $(M_1', c_1').$  We extend the colouring $c_1'$ to a colouring $c_1''$ to $M_1''$ where 
$$(c_1''(s_1), c_1''(s_2)):= \left\{ \begin{array}{ll} (c(f_1), c(f_2)) & \mathrm{if}\ f_1 \in D_1 \ \mathrm{or} \ f_2 \in D_2\\  (c(f_2), c(f_1)) & \mathrm{otherwise} \end{array}. \right.$$ 
Suppose first that $M_1''$ is simple.  Then by assumption, there is a $T$-SRCP $\{ C_1',C_2'\}$ for $(M_1'', c_1'').$  If for $i=1,2,$ $\{ s_1, s_2 \} \not\subset C_i',$ then similar to before, one can construct a $T$-SRCP for $(M,c)$ using the circuits $C_i',D_i,\ i = 1,2.$   We may assume that $\{ s_1, s_2 \} \subset C_1'.$  Let $D$ be an $s_3$-rainbow circuit for $(M_2',c_2')$ where $|D| \le n_2.$  Letting $C_1 = C_1' \triangle S \triangle D$ we see that $\{ C_1,C_2' \}$ is a $T$-SRCyP for $(M,c).$ 

Suppose instead that $M_1''$ is not simple.  Then for some $x \in E(M_1')$ and some $i\in \{ 1,2 \},$ $x$ is parallel with $s_i.$  We may assume that $x$ is parallel with $s_1$.  Clearly $x \in \{ f_1', f_2' \}$, for otherwise, $D_1-s_1+x$ is a rainbow circuit of $(M,c).$
We may assume that $x = f_1'.$  Now $f_2'$ is not parallel with $s_2$ or $s_3$, for otherwise $M_2' + f_1' +f_2'$ would be a simple matroid with $2n_2$ elements and hence would contain a rainbow circuit.  Let $M_2'' = M_2' + f_1'$ and extend the colouring $c_2'$ to a colouring $c_2''$ of $M_2''$ where $c_2''(f_1') := c(f_1').$  Then $\varepsilon(M_2'') = 2n_2-1$ and $f_2$ is the unique colour-singular element for $(M_2'', c_2'').$
By assumption, Theorem \ref{the-NoRainbowCircuit} i.1) or i.2) holds for $(M_2'', c_2'').$  Thus there is an $S$-SRCP $\{ D_1',D_2'\}$ for $(M_2'',c_2'')$ where $s_2 \in D_1'$ and $s_3 \in D_2'.$  Redefine $(M_1'',c_1'')$ where $M_1'' = M_1' + s_2 + s_3$ and extend the colouring $c_1'$ to a colouring $c_1''$ of $M_1''$ as follows:  
Suppose that some $i$, $f_2 \in D_i'.$  Then define $c_1''(s_{i+1}) := c(f_2)$ and $c_1''(s_{4-i}) = c(f_1).$ If $f_2 \not\in D_1'\cup D_2'$ and for some $i$, $f_1\in D_i'$, then define
$c_1''(s_{i+1}) = c(f_1)$ and $c_1''(s_{4-i}) = c(f_2).$ Otherwise, if for $i = 1,2,$ $f_i \not\in D_1'\cup D_2',$ then define $c_1''(s_2) := c(f_1)$ and $c_1''(s_3):= c(f_2).$  We claim that $(M_1'', c_1'')$ is circuit-achromatic.
For suppose $(M_1'', c_1'')$ has a rainbow circuit $C$.  Then $C\cap \{ s_2, s_3 \} \ne \emptyset$ and we may assume that $s_2 \in C.$  If $s_3 \in C$, then $C' = C \triangle S = C - s_2 - s_3 + s_1$ is seen to be an $s_1$-rainbow circuit where $C' \cap \{ f_1', f_2' \} = \emptyset.$  Thus for an $s_1$-rainbow circuit $D$ for $(M_2',c_2'),$ it is seen that $C' \triangle D$ is a rainbow cycle of $(M,c).$  Thus $s_3 \not\in C.$  We also have that $C \cap \{ f_1', f_2' \} \ne \emptyset$ for otherwise $(M,c)$ has a rainbow circuit.  Suppose $f_2' \in C.$  Then $c_1''(s_2) = c(f_1)$
and thus $f_1' \not\in C.$   By the definition of $c_1'',$ we have that $f_2 \not\in D_1'$ and thus it is seen that $C \triangle D_1'$ is a rainbow cycle of $(M,c).$  Thus $f_2' \not\in C$ and hence $f_1' \in C.$  Then $C'= C \triangle \{ f_1', s_1 \} \triangle S = C - s_2 - f_1' + s_3$ is an $s_3$-rainbow circuit.  Let $D$ be an $s_3$-rainbow circuit for $(M_2',c_2').$  Then $C' \triangle D$ is seen to be rainbow circuit of $(M,c).$  This proves our claim.  Thus $(M_1'', c_1'')$ is simple and circuit-achromatic and $\varepsilon(M_1'') = 2n_1 -2.$  By assumption, there is a $T$-SRCP $\{ C_1',C_2'\}$ for $(M_1'', c_1'').$  We may assume that $\{s_2, s_3 \} \cap (C_1' \cup C_2') \ne \emptyset$
for otherwise $\{ C_1',C_2'\}$ is a $T$-SRCP for $(M,c).$  We may assume that $s_2 \in C_1'.$

\ssms
\noindent{\it Case 2.1.1.1} $s_3 \not\in C_2'.$
\ssms

Suppose first that $s_3 \not\in C_1'.$  If $f_1' \not\in C_1'$,  then letting $C_1 = C_1' \triangle D_1'$ it is seen that $\{C_1,C_2' \}$ is a $T$-SRCyP for $(M,c).$  Suppose $f_1' \in  C_1'.$  Then $c_1''(s_2) = c(f_2)$ and $f_2' \not\in C_1'.$
We see that $C_1'' = C_1' \triangle \{ f_1', s_1 \} \triangle S = C_1' - s_2 - f_1' +s_3$ is a circuit.  Let $D$ be an $s_3$-rainbow circuit for $(M_2',c_2')$ and let $C_1 = C_1'' \triangle D.$  Then $\{ C_1,C_2' \}$ is seen to be a $T$-SRCyP for $(M,c).$
Suppose $s_3 \in C_1'.$  Then for $i= 1,2,$ $f_i' \not\in C_1'.$  Let $C_1'' = C_1' \triangle S = C_1' - s_2 - s_3 + s_1$ and let $D$ be an $s_1$-rainbow circuit for $(M_2',c_2').$  Letting $C_1 = C_1'' \triangle D$, we see that $\{ C_1,C_2' \}$ is a
$T$-SRCyP for $(M,c).$

\ssms
\noindent{\it Case 2.1.1.2} $s_3 \in C_2'.$
\ssms

We may assume that $f_1' \in C_1'$, for if $f_1' \not\in C_1' \cup C_2',$ then one can construct (using similar arguments as before) a $T$-SRCyP for $(M,c)$ using the circuits $C_i',D_i',\ i = 1,2.$  
Thus we have that $c_1''(s_2) = c(f_2)$ and $f_2' \not\in C_1'.$  Furthermore, by the definition of $c_1'',$ it follows that $f_1 \not\in D_1'$ and $f_2 \not\in D_2'.$  Letting $C_i = C_i' \triangle D_i',\ i = 1,2$, we see that $\{ C_1, C_2 \}$ is a $T$-SRCyP for $(M,c).$

%

\sms
\noindent{\bf Case 3} $\varepsilon_1' = 2n_1-5.$
\sms

We have that $\varepsilon_2' = 2n_2-1$ and $(M_2', c_2')$ contains a unique colour-singular element which we denote by $f.$  We also have that $f'$ is the unique colour-singular element for $(M_1', c_1').$  We observe that there is no element of $M_1'$ which is parallel to an element of $S$, for if there was such an element $x$, then $M_2' + x$ would have $2n_2$ elements and hence would contain a rainbow circuit. Thus  $M_1$ is simple.  We shall extend the colouring $c_1'$ to a colouring $c_1$ of $M_1$ in the following manner. By assumption, Theorem \ref{the-NoRainbowCircuit} i.1) or i.2) holds
for $(M_2', c_2')$ and $S$.  If i.1) holds and there exists an $S$-SRCT $\{ D_1,D_2,D_3 \}$ for $(M_2',c_2')$ where for $i= 1,2,3,$ $s_i \in D_i,$  then we may assume that $f \not\in D_2\cup D_3.$
Define $c_1''(s_1):= c(f)$ and for $i=2,3,$ define $c_1''(s_i):=0.$    

We claim that $(M_1,c_1)$ is circuit-achromatic.  For suppose it contains a rainbow circuit $C$.  Then $C\cap S \ne \emptyset.$  Suppose $|C\cap S| =2.$  We may assume that
$\{ s_1, s_2 \} \subset C.$  We have that $f' \not\in C$ and $C' = C \triangle S = C-s_1-s_2+s_3$ is an $s_3$-rainbow circuit.  Let $D$ be an $s_3$-rainbow circuit for $(M_2',c_2').$  Then $C'\triangle D$ is seen to be a rainbow cycle of $(M,c),$ a contradiction.  Suppose $|C\cap S| =1.$  Suppose $s_1\in C.$  Then $f'\not\in C$ and for an $s_1$-rainbow circuit $D$ of $(M_2',c_2')$, we have that $C\triangle D$ is a rainbow cycle in $(M,c).$  Thus we may assume that $s_2\in C.$  Now regardless of whether Theorem \ref{the-NoRainbowCircuit} i.1) or i.2) holds for $(M_2',c_2')$ and $S$, there is an $S$-SRCP $\{ D_2,D_3 \}$ for which
$s_i \in D_i,\ i = 2,3$ and $f \not\in D_2 \cup D_3.$  Now it is seen that $C \triangle D_2$ is a rainbow cycle in $(M,c),$ a contradiction. Thus $(M_1,c_1)$ is circuit-achromatic.

Since $\varepsilon(M_1) = 2n_1-2,$ it follows by assumption that Theorem \ref{the-NoRainbowCircuit} iii) holds for $(M_1,c_1)$ and $T.$  Thus there is a $T$-SRCP $\{C_1',C_2' \}$ for $(M_1,c_1).$  Again, we may assume that for some $i$, $C_i' \cap S \ne \emptyset,$ and this we may assume is true for $i=1.$
If $C_2' \cap S = \emptyset,$ then one can use similar arguments as in the previous cases to construct a $T$-SRCP for $(M,c).$  Thus we may assume that $C_2' \cap S \ne \emptyset.$ 

\ssms
\noindent{\it Case 3.1} $s_1 \not\in C_1' \cup C_2'.$
\ssms

We may assume that $s_2\in C_1'$ and $s_3 \in C_2'.$  As before, there is an $S$-SRCP $\{D_2,D_3\}$ for $(M_2', c_2')$ where $s_i \in D_i',\ i = 2,3$ and $f \not\in D_2 \cup D_3.$  Letting $C_1 = C_1' \triangle D_2$ and $C_2 = C_2' \triangle D_3,$ we see that
$\{ C_1,C_2 \}$ is a $T$-SRCyP for $(M,c).$

\ssms
\noindent{\it Case 3.2} $s_1 \in C_1' \cup C_2'.$
\ssms

We may assume that $s_1 \in C_1'.$  Suppose first that $|C_1' \cap S| =2.$  We may assume that $s_2 \in C_1'$ and $s_3 \in C_2'.$  Then $f' \not\in C_1'.$  Let $C_1'' = C_1' \triangle S = C_1' - s_1 -s_2 +s_3.$  By assumption, Theorem \ref{the-NoRainbowCircuit} ii) holds
for $(M_2', c_2')$. We may assume that $f$ and $s_3$ are not parallel;  if they are, then we can assume that $f\in E(M_1')$ and revert back to Case 2.  It follows that there is an $s_3$-semi-SRCP $\{ D_1', D_2' \}$ for $(M_2',c_2')$ where
we may assume that $f \not\in D_2'.$  Let $C_1 = C_1'' \triangle D_1'$ and let $C_2 = C_2' \triangle D_2'.$  Then $C_i,\ i = 1,2$ are rainbow cycles in $(M,c)$ and 
\begin{align*}
|C_1| + |C_2| &= |C_1''| + |C_2'| + |D_1'| + |D_2'| -4\\
&=  |C_1'| -1 +|C_2'| + |D_1'| + |D_2'| -4\\
&\le n_1 + 2 -1 + n_2+3 -4 = n+2.\\
\end{align*}
Thus $\{ C_1, C_2 \}$ is a $T$-SRCyP for $(M,c).$  We may assume that $C_1' \cap S = \{ s_1 \}$ and $C_2' \cap S = \{ s_2 \}.$  Regardless of whether Theorem \ref{the-NoRainbowCircuit} i.1) or i.2) holds for $(M_2',c_2'),$ there is a $S$-SRCP $\{ D_1, D_2 \}$ where
for $i=1,2,$ $s_i \in D_i$ and $f\not\in D_2.$  Letting $C_i = C_i' \triangle D_i,\ i = 1,2$ we see that $\{ C_1, C_2\}$ is a $T$-SRCyP for $(M,c).$ 

\section{Proof of main theorem}\label{sec-Proofofmaintheorem}
 
 In this section, we shall complete the proof of Theorem \ref{the-main}.  
To begin with, we shall use the following simple observations:
%

\begin{observation}
Let $(M,c)$ be a coloured binary matroid where $\varepsilon(M) = 2(r(M)+1)$, $r(M) \ge 3$ and $c$ is $2$-uniform.  Suppose there are at least three colour classes where the elements are non-parallel.
If $M$ contains a rainbow $2$-circuit, then it contains a SRCP.\label{obs-rainbowdigon}
\end{observation}

\begin{proof}
Suppose $M$ contains a rainbow $2$-circuit $C = \{ e, f \}.$  Let $M' = M - C.$  We shall show that $M'$ has a rainbow circuit of size at most $r(M).$  We have $|c(M')| = |c(M)| = r(M) + 1 \ge r(M') +1$ and $\varepsilon(M') = 2r(M) \ge 2r(M') \ge r(M') + 2.$  Thus $M'$ contains a rainbow circuit $D$. If $|D| \le r(M),$ then $D$ is the desired circuit.  Suppose instead that $|D| >r(M).$  Then $|D| = r(M) +1$ and $c(D) = c(M).$  
Since $M$ contains at least $3$ colour classes which are non-parallel, there exists $x\in D$ which belongs to a colour class $X= \{ x,x'\}$ for which $x$ and $x'$ are non-parallel.  Then $x'$ is a chord of $D$ and $D+x'$ contains a $D,x'$-chordal circuit $D'$ where $x \not\in D',$ $D'$ is a rainbow circuit, and $|D'| < |D| = r(M) +1.$ Thus $D'$ is the desired circuit.  

From the above, there is a rainbow circuit $D$ in $M'$ for which $|D'| \le r(M).$  Now $\{ C,D \}$ is seen to be a SRCP.  
\end{proof}

\begin{observation}
Let $(M,c)$ be a coloured binary matroid where $\varepsilon(M) = 2r(M)$ and $c$ is $2$-bounded, with the exception of one colour class which is a $3$-circuit.  Then $M$ has a rainbow circuit.
\label{obs-find1rainbowc}
\end{observation}

\begin{proof}
We have that $|c(M)| \ge r(M).$  If $|c(M)| \ge r(M)+1,$ then clearly $(M,c)$ has a rainbow circuit.  As such, we may assume that $|c(M)| = r(M).$   Let $B$ be a rainbow subset where $|B| = r(M).$  If $B$ is not a basis, then it contains a rainbow circuit.  Thus we may assume that $B$ is a basis.  Let $\{ e,f,g \}$ be the colour class which is a $3$-circuit where
we may assume that $e \in B.$ Now $B+f$ contains a circuit $C$.  If $e\not\in C,$ then $C$ is a rainbow circuit.  If $e\in C,$ then $C' = C \triangle \{ e,f,g \}  = C - e -f +g$ is rainbow circuit.  In either case, $M$ contains a rainbow circuit.
\end{proof}

\begin{observation}
Suppose $(M,c)$ is a simple, regular coloured matroid where $M = M_1 \oplus_3 M_2$ and $c$ is a $2$-uniform $(r(M)+1)$-colouring.  Then either $M$ has a SRCP or for exactly one $i\in [2]$, 
$M_i \backslash E(M_{3-i})$ contains a rainbow circuit.
\label{obs-1rainbowcircorpair}
\end{observation}

\begin{proof}
Assume that $(M,c)$ has no SRCP.  Let $E(M_1) \cap E(M_2) = T = \{ t_1, t_2, t_3 \}.$  For $i=1,2,$ let $n_i = r(M_i),$ $M_i' = M_i -T$ and $\varepsilon_i' = \varepsilon(M_i').$  For $i=1,2,$ let $c_i' = c \big| E(M_i').$
We have $\varepsilon(M) = 2(r(M) +1) = 2(n_1 + n_2 -1) = 2n_1 + 2n_2 -2 = \varepsilon_1' + \varepsilon_2'.$
Suppose that for $i=1,2,$ $(M_i',c_i')$ is circuit-achromatic.  Then for $i=1,2,$ $\varepsilon_i' \le 2n_i -1$ and hence $\varepsilon_i' = 2n_i-1.$
Let $e$ be the unique colour-singular element for $(M_1',c_1').$  Then $e'$ is the unique colour-singular element for $(M_2',c_2').$  By Theorem \ref{the-NoRainbowCircuit} i), there is a $T$-SRCP $\{ C_1',C_2' \}$ for $(M_1',c_1')$ where $e \not\in C_1' \cup C_2'.$
We may assume that for $i=1,2,$ $t_i \in C_i'.$  By Theorem \ref{the-NoRainbowCircuit} i) $(M_2',c_2')$ has a $T$-SRCP $\{ D_1,D_2 \}$ where for $i=1,2,$ $t_i \in D_i.$  Letting $C_i = C_i' \triangle D_i$ is seen that $\{ C_1,C_2\} $ is a SRCyP
for $(M,c)$ (implying that $(M,c)$ has a SRCP), contradicting our assumption. We may assume that $(M_1',c_1')$ contains a rainbow circuit, say $C_1$  where we may choose $C_1$ so that $|C_1| \le n_1.$  Suppose $(M_2',c_2')$ also contains a rainbow circuit, say $C_2$.  We may choose $C_2$ 
so that $|C_2| \ne n_2.$  Then $|C_1| + |C_2| \le n_1 + n_2 \le n+2$ and hence $\{ C_1,C_2 \}$ is a SRCP for $(M,c),$ contradicting our assumption.  Thus only $(M_1',c_1')$ has a rainbow circuit.  
\end{proof}

Suppose $(M,c)$ is a simple, regular coloured matroid where $c$ is a $2$-uniform $(r(M)+1)$-colouring.  We shall assume that Theorem \ref{the-main} is true for all regular matroids with rank less than $r(M).$  By Theorem \ref{the-GraphicCographicMain}, we may assume that $M$ is neither graphic nor cographic. 
It follows that $M$ is a either a $2$-sum or $3$-sum of regular matroids.

\begin{nonamenoname}
We may assume that there is no proper subset $A \subset E(M)$ where $M' = M \big| A$ is such that $\varepsilon(M') = 2(r(M') +1).$
\label{finishnonanona0}
\end{nonamenoname}

\begin{proof}
Suppose there is a subset $A \subset E(M)$ where $M' = M \big| A$ is such that $\varepsilon(M') = 2(r(M') +1).$  By assumption, $M'$ has a SRC $4$-tuple and such is seen to be a SRC $4$-tuple for $M.$
Thus we may assume that there is no such subset $A.$
\end{proof}

\begin{nonamenoname}
We may assume that $M$ is a $3$-sum of regular matroids.
\label{finishnonanona1}
\end{nonamenoname}

\begin{proof}
Suppose that $M$ is a $2$-sum of regular matroids $M = M_1 \oplus_2 M_2$ where $E(M_1) \cap E(M_2) = \{ s \}.$  For $i = 1,2,$ let $M_i' = M_i \backslash E(M_{3-i}),$ $c_i' =c \big| E(M_i'),$ $n_i = r(M_i)$ and $\varepsilon_i' = \varepsilon(M_i').$  
We have that $\varepsilon(M) = 2(r(M)+1) = 2(n_1 + n_2) = \varepsilon_1' + \varepsilon_2'.$
It follows by Lemma \ref{lem-shortcircuitin23sum} that we may assume that $(M_1', c_1')$ is circuit-achromatic. Thus $\varepsilon_1' \le 2n_1 -1$ and $\varepsilon_2' \ge 2n_2 +1.$  By ({\bf \ref{finishnonanona0}}), $\varepsilon_2' \le 2n_2+1$ and hence it follows that
$\varepsilon_2' = 2n_2+1$ and $\varepsilon_1' = 2n_1-1.$  By Observation \ref{obs-onesingular}, $(M_1',c_1')$ has a unique colour-singular element, say $f$ and $f'$ is seen to be the unique colour-singular element of $(M_2', c_2').$
If there is an element $x \in E(M_2')$ which is parallel with $s$, then we can add $x$ to $M_1'$, in which case $M_1'$ has $2n_1$ elements and $M_2'$ has $2n_2$ elements.  Then for $i=1,2,$ $M_i'$ contains a rainbow circuit
$C_i$ where $|C_i| \le n_i$ and $\{ C_1, C_2 \}$ is seen to be a SRCP for $(M,c).$  Thus we may assume that no element of $E(M_2')$ is parallel with $s$ and hence $M_2$ is simple.  Extend the colouring $c_2'$ to a colouring $c_2$ of $M_2$ where $c_2(s) = c(f).$  Then $c_2$ is a $2$-uniform $(n_2+1)$-colouring of $M_2$.
By assumption, there is a SRC $4$-tuple $\{ C_1', C_2', C_3', C_4' \}$ for $(M_2,c_2).$  By Observation \ref{obs-erainbowcircuit}, there is an $s$-rainbow circuit $D$ for $(M_1',c_1')$ where $|D|\le n_1.$  For all $i\in [4],$ let 
$C_i = \left\{ \begin{array}{lr} C_i' \triangle D & \mathrm{if}\  s \in C_i'\\ C_i' &\mathrm{if} \ s \not\in C_i'. \end{array} \right.$  Then it is seen that $\sum_{i=1}^4 |C_i| \le 2n_2+ 4 + 2(|D| -2) \le 2(n_1+n_2) = 2r(M) +2.$  Thus $\{ C_1,C_2,C_3,C_4 \}$ is a SRC $4$-tuple for $(M,c).$   
It follows from the above that we may assume that $M$ is a $3$-sum of two regular matroids.
\end{proof}

By ({\bf \ref{finishnonanona1}}), we may assume that $M$ is a $3$-sum $M = M_1 \oplus_3 M_2$ of regular matroids where $S = \{ s_1, s_2, s_3 \} = E(M_1) \cap E(M_2)$ is co-simple.  Furthermore, we may assume that for $i=1,2,$ $r(M_i) \ge 4.$   For $i = 1,2,$ let $M_i' = M_i \backslash E(M_{3-i}),$ $\varepsilon_i' = |E(M_i')|,$ and let $c_i' = c \big| E(M_i').$
By Lemma \ref{lem-shortcircuitin23sum}, we may assume that $(M_1', c_1')$ is circuit-achromatic.

 Let $F$ (resp. $F'$)  be the set of singular elements of $(M_1',c_1')$ (resp. $(M_2', c_2')$)  If $F \ne \emptyset$ (resp. $F' \ne \emptyset$), let $F = \{ f_1, \dots, f_\alpha \}$ (resp. $F' = \{ f_1', \dots ,f_{\alpha}' \}$).
 
\begin{nonamenoname}
$2r(M_1) -3 \le \varepsilon_1' \le 2r(M_1) -1.$
\label{finishnonanona2}
\end{nonamenoname}

\begin{proof}
Since $M_1'$ is simple and $(M_1', c_1')$ is circuit-achromatic, it follows that $\varepsilon_1' \le 2r(M_1) -1.$  Also, by our assumptions, $\varepsilon_2' \le 2r(M_2) +1.$ 
Since $$\varepsilon_1' + \varepsilon_2' = \varepsilon(M) = 2(r(M) +1) = 2(r(M_1) + r(M_2) -1),$$ it follows that
\begin{align*} 
\varepsilon_1' &= 2(r(M_1) + r(M_2) -1) - \varepsilon_2'\\ 
&\ge 2(r(M_1) + r(M_2) -1) - 2r(M_2) -1 = 2r(M_1)  -3.
\end{align*}
\end{proof}

\begin{nonamenoname}
Suppose $\varepsilon_1' = 2r(M_1) -1.$  Then $F= \{ f_1 \}$ and there is an $S$-SRCP $\{ D_{1}', D_{2}' \}$ for $(M_1', c_1')$ where $f_1 \not\in D_{1}' \cup D_{2}'.$  Furthermore, we may assume that if $(M_1',c_1')$ has a $S$-SRCT, then there is circuit $D_3'$ such that $\{ D_1', D_2', D_3' \}$ is an $S$-SRCT for $(M_1',c_1').$   

\label{finishnonanona3}
\end{nonamenoname} 

\begin{proof}
By Observation \ref{obs-onesingular}, we have that $|F| =1$ and hence $F = \{ f_1 \}.$
We have that Theorem \ref{the-NoRainbowCircuit} i.1) or i.2) holds for $(M_1',c_1').$  If i.1) holds, then there is an $S$-SRCT $\{ D_{1}', D_{2}', D_{3}' \}$ and hence for some $i <j,$ $f_1 \not\in D_{i}' \cup D_{j}'.$  If Theorem \ref{the-NoRainbowCircuit} i.2) holds for $(M_1',c_1'),$ then there is a $S$-SRCP $\{ D_1', D_2' \}$ for $(M_1', c_1')$ where $f_1 \not\in D_1' \cup D_2'.$  Thus regardless of whether i.1) or i.2) holds, there is an $S$-SRCP $\{ D_{1}', D_{2}'\}$ for $(M_1', c_1')$ where $f_1 \not\in D_{1}' \cup D_{2}'.$  Furthermore, in the case where $(M_1',c_1')$ has a $S$-SRCT, we may assume that there is a circuit $D_3'$ such that $\{ D_1', D_2', D_3' \}$ is an $S$-SRCT for $(M_1',c_1').$ 
\end{proof}

\begin{nonamenoname}
Suppose that $\varepsilon_1' = 2r(M_1) -1.$ Then $(M,c)$ has a SRC $4$-tuple.
\label{finishnonanona4}
\end{nonamenoname}

\begin{proof}
By (\ref{finishnonanona3}), $F = \{ f_1 \}$ 
and there is a $S$-SRCP $\{ D_1', D_2' \}$ for $(M_1',c_1')$ where $f_1 \not\in D_1' \cup D_2'.$  Furthermore, we may assume that if $(M_1',c_1')$ has a $S$-SRCT, then there is circuit $D_3'$ such that $\{ D_1', D_2', D_3' \}$ is an $S$-SRCT for $(M_1',c_1').$ 
Without loss of generality, we may assume that for $i= 1,2$ $s_i \in D_i'.$  Note that $F' = \{ f_1' \}$ and $f_1$ is not parallel with $s_1$ or $s_2.$  
Extend the colouring $c_2'$ to a colouring $c_2$ of $M_2$ where for $i = 1,2$, $c_2(s_i) := 0$ and $c_2(s_3):= c(f_1).$ 
Then $c_2$ is a $2$-uniform $(r(M_2) +1)$-colouring of $M_2.$   Suppose $M_2$ is not simple.  Then there exists an element $x \in E(M_2')$ and $j\in [3]$ for which $x$ is parallel with $s_j.$ If $\{ x, s_j \}$ is a rainbow $2$-circuit, then it follows by Observation \ref{obs-rainbowdigon} that $M_2' + s_j$ has a SRCP (and hence also a SRC $4$-tuple).  Suppose $c_2(x) = c_2(s_j)$.  Then $j=3,$ $c_2(x) = c_2(s_3) = c(f_1)$ and hence $x = f_1'.$  Now we may assume that $f_1' \in E(M_1')$ and $\varepsilon_1' = 2n_1,$ $\varepsilon_2' = 2n_2-2$ and $F=\emptyset.$  We may now argue as in the proof of (\ref{finishnonanona5}), where $F = \emptyset,$ with the roles of $M_1$ and $M_2$ reversed.  Thus we may assume that $M_2$ is simple.

By assumption, $(M_2,c_2)$ has a SRC $4$-tuple $\{C_1',C_2',C_3', C_4'\}$.  We may assume that $\left(\bigcup_i C_i'\right)\cap S \ne \emptyset,$ for otherwise $\{ C_1',C_2',C_3',C_4'\}$ a SRC $4$-tuple for $(M,c).$  

\sms
\noindent{\bf Case 1} For exactly one $i\in [4],$ $C_i' \cap S \ne \emptyset.$ 
\sms

We may assume that $C_1' \cap S \ne \emptyset$ and for all $i= 2,3,4,$ $C_i'\cap S =\emptyset.$   Furthermore, we may assume that $|C_1' \cap S|=1,$ since if $|C_1' \cap S|=2,$ then $C_1'' = C_1' \triangle S$ is a rainbow circuit where $|C_1'' \cap S| =1.$
Let $C_1' \cap S = \{ s_i \}$.  If $i\in \{ 1,2 \},$ then $C_1 = C_1' \triangle D_i'$ is seen to be a rainbow cycle of $M$ and $|C_1| = |C_1'| + |D_i'| - 2.$ Then 
\begin{align*} |C_1| + |C_2'| + |C_3'| + |C_4'| &= |C_1'| + |D_i'| - 2 + |C_2'| + |C_3'| + |C_4'|\\ &\le 2r(M_2) + 4 + r(M_1) -2\\ &= r(M_2) + (r(M_1) + r(M_2)) + 2\\ &= r(M_2) + r(M) + 4 \le 2r(M) + 4.\end{align*}
Thus $\{ C_1, C_2', C_3', C_4' \}$ is a SRCy $4$-tuple in $(M,c).$  Suppose instead that $C_1' \cap S = \{ s_3 \}.$  Since $|c_1'(M_1')| = r(M_1)$, there is an $s_3$-rainbow circuit $D$ for $(M_1',c_1')$ where $|D| \le r(M_1).$
Given that $c_2(s_3) = c(f_1),$ it follows that $C_1 = C_1' \triangle D$ is a rainbow cycle. Arguing as above $D$ in place of $D_i'$, one sees that $\{ C_1, C_2', C_3',C_4' \}$ is a SRCy $4$-tuple for $(M,c).$

 \sms
\noindent{\bf Case 2}  For exactly two integers $i\in [4],$ $C_i' \cap S \ne \emptyset.$ 
\sms

We may assume that for $i=1,2,$ $C_i' \cap S \ne \emptyset$ and for $i=3,4,$ $C_i' \cap S = \emptyset.$  We may also assume that for $i= 1,2,$ $|C_i' \cap S| =1$, for if for some $i$, $|C_i' \cap S| = 2,$ then we can use $C_i' \triangle S$ in place of $C_i'.$
Suppose first that $C_1' \cap S \ne C_2'\cap S.$  Then we may assume that $C_1' \cap S = \{ s_1 \}.$  Suppose that $C_2' \cap S = \{ s_2 \}.$  For $i = 1,2,$ let $C_i = C_i' \triangle D_i'$.  Then 
$C_i',\ i = 1,2$ are seen to be rainbow cycles of $(M,c)$ and $|C_1| + |C_2| = |C_1'| + |C_2'| + |D_1'| + |D_2'| - 4$.  Thus we have
\begin{align*}
|C_1| + |C_2| + |C_3'| + |C_4'| &= |C_1'| + |C_2'| + |D_1'| + |D_2'| - 4 + |C_3'| + |C_4'|\\ &\le 2r(M_2) + 4 - 4 +  r(M_1) + 2 \le 2r(M) + 4. 
\end{align*} 
Suppose that $C_2' \cap S = \{ s_3 \}.$  If $(M_1',c_1')$ has an $S$-SRCT, then by assumption there is an $s_3$-rainbow circuit $D_3'$ for $(M_1',c_1'),$ such that $\{D_1',D_2',D_3'\}$ is an $S$-SRCT for $(M_1',c_1').$
Letting $C_1 = C_1' \triangle D_1'$ and $C_2 = C_2' \triangle D_3',$ we see that $\{ C_1,C_2,C_3',C_4'\}$ is a SRCy $4$-tuple for $(M,c).$  Suppose that $(M_1',c_1')$ has no $S$-SRCT.  Then there exists an $S$-SRCP
$\{D_1'', D_2'' \}$ for $(M_1',c_1')$ where $s_1\in D_1'',$ $s_3 \in D_2'',$ and $f_1 \not\in D_1'' \cup D_2''.$  Letting $C_1 = C_1' \triangle D_1''$ and $C_2 = C_2' \triangle D_2'',$ we see that $\{ C_1,C_2,C_3',C_4'\}$ is a SRCy $4$-tuple for $(M,c).$

Suppose now that $C_1' \cap S = C_2' \cap S.$  Assume first that $C_1' \cap C_2' = \{ s_1 \}.$  For $i =1,2,$ let $C_i = C_i' \triangle D_1'.$  Then $|C_1| + |C_2| = |C_1'| +|C_2'| + 2|D_1'| - 4$ and thus
\begin{align*} |C_1| + |C_2| + |C_3'| + |C_4'| &=  |C_1'| +|C_2'| + 2|D_1'| - 4 + |C_3'| + |C_4'|\\ &\le 2r(M_2) + 4 - 4 + 2r(M_1)\\ &= 2(r(M_1) + r(M_2)) = 2r(M) +4.\end{align*} 
Thus $\{ C_1, C_2, C_3', C_4'\}$ is a SRCy $4$-tuple for $(M,c).$  If $C_1' \cap C_2'  = \{ s_2 \},$ then similar reasoning can be used with $D_2'$ in place of $D_1'.$
Assume that $C_1' \cap C_2' = \{ s_3 \}.$  Since $|c_1'(M_1')| = r(M_1),$ there is an $s_3$-rainbow circuit $D$ for $(M_1',c_1')$ where $|D| \le r(M_1).$  Noting that $C_i = C_i' \triangle D,\ i = 1,2$ are rainbow cycles, we see that as before that
$\{ C_1, C_2, C_3', C_4'\}$ is a SRCy $4$-tuple for $(M,c).$

\sms
\noindent{\bf Case 3} For exactly three integers $i\in [4],$ $C_i' \cap S \ne \emptyset.$ 
\sms

We may assume that for all $i \in [3],$ $C_i' \cap S \ne \emptyset$ and $C_4' \cap S = \emptyset.$
We shall look at two sub-cases, depending on whether $(M_1', c_1')$ has a $S$-SRCT or not.

\sms
\noindent{\it Case 3.1} Suppose that $(M_1',c_1')$ has an $S$-SRCT.
\sms

By assumption, there is a rainbow circuit $D_3'$ where $\{ D_1', D_2', D_3'\}$ is a $S$-SRCT for $(M_1', c_1').$
There are two sub-cases, depending on whether the sets $C_i' \cap S, i\in [3]$ are distinct or not.

\sms
\noindent{\it Case 3.1.1} Suppose that the sets $C_i' \cap S, i\in [3]$ are not distinct.
\sms

We may assume that $C_1' \cap S = C_2' \cap S.$  Suppose first that $|C_1' \cap S| =1$ and assume that for some $i\in [3],$ $C_1' \cap C_2' = \{ s_i \}.$  Suppose $i\in \{ 1, 2 \}.$  Then given that $f_1$ is not parallel with $s_1$ or $s_2,$ there is a $s_i$-semi-SRCP $\{ D_1'', D_2''\}$ for $(M_1', c_1')$ where $f_1 \not\in D_1''.$
Define $C_i,\ i = 1,2,3$ as follows:
$$(C_1,C_2, C_3) = \left\{ \begin{array}{ll} (C_1' \triangle D_i', C_2' \triangle D_i', C_3' \triangle D_j') &\mathrm{if}\ C_3' \cap S = \{ s_j \}\\
 (C_1' \triangle D_i', C_2' \triangle D_1'', C_3' \triangle S \triangle D_2'') &\mathrm{if}\ C_3' \cap S = S-s_i \end{array} \right.$$   
Then $\{ C_1, C_2, C_3, C_4'\}$ is seen to be a SRCy $4$-tuple for $(M,c).$  Lastly, if $C_1' \cap C_2' = \{ s_3 \},$  then $f_1$ is not parallel with $s_3$ (otherwise, $C_1' - s_3 + f_1$ is a rainbow circuit) and we can argue as before.

Suppose that $|C_1' \cap S| =2.$ Then for some $i\in [2],$ $C_1' \cap S = \{ s_i, s_3 \}.$  Let $S - \{ s_i, s_3 \} = \{ s_j \}.$   Then $C_3' \cap S = \{ s_j \}.$ Given that $f_1$ is not parallel with $s_j,$ there is a $s_j$-semi-SRCP $\{ D_1'', D_2''\}$ for $(M_1',c_1')$ where $f_1 \not\in D_1''.$  Now define $(C_1, C_2, C_3) = (C_1' \triangle S\triangle  D_2'', C_2' \triangle S \triangle  D_2'', C_3' \triangle D_1'').$  Then $\{ C_1, C_2, C_3, C_4' \}$ is seen to be a SRCy $4$-tuple.

\sms
\noindent{\it Case 3.1.2} Suppose that the sets $C_i' \cap S,\ i \in [3]$ are distinct.
\sms

For $i= 1,2,3$ let $C_i = \left\{ \begin{array}{ll} C_i' \triangle D_j' &\mathrm{if}\ C_i' \cap S = \{ s_j \}\\ C_i' \triangle S \triangle D_j' &\mathrm{if}\ S - C_i' = \{ s_j \}  .\end{array} \right.$  We see that $\{ C_1,C_2,C_3,C_4'\}$ is a SRCy $4$-tuple for $(M,c).$

\sms
\noindent{\it Case 3.2} Suppose that $(M_1',c_1')$ has no $S$-SRCT.
\sms

We have that Theorem \ref{the-NoRainbowCircuit} i.2) holds for $(M_1', c_1')$ and $S.$  Thus for all $i,j,\ i < j,$ there exists an $S$-SRCP $\{D_i^{ij}, D_j^{ij}\}$ for $(M_1',c_1')$ where $s_i\in D_i^{ij},$ $s_j \in D_j^{ij},$ and $f_1 \not\in D_i^{ij} \cup D_j^{ij}.$
Again, we have two sub-cases, depending on whether the sets $C_i \cap S,\ i\in [3]$ are distinct or not. 

\sms
\noindent{\it Case 3.2.1} Suppose that the sets $C_i \cap S,\ i\in [3]$ are not distinct.
\sms

  We may assume that $C_1' \cap S = C_2' \cap S.$  Suppose first that for some $i\in [3],$ $C_1' \cap C_2' = \{ s_i \}.$  If $i\in \{ 1, 2 \},$ then $f_1$ is not parallel with $s_i$ and hence there is a $s_i$- semi-SRCP $\{ D_1'', D_2'' \}$ for $(M_1',c_1')$ where $f_1 \not\in D_1''.$  Define $C_i,\ i = 1,2,3$ as follows:
$$(C_1,C_2, C_3) = \left\{ \begin{array}{ll} (C_1' \triangle D_i^{ij}, C_2' \triangle D_i^{ij}, C_3' \triangle D_j^{ij}) &\mathrm{if}\ C_3' \cap S = \{ s_j \}\\
(C_1' \triangle D_i', C_2' \triangle D_1'', C_3' \triangle S \triangle D_2'') &\mathrm{if}\ C_3' \cap S = S-s_i. \end{array} \right.$$
We see that $\{ C_1,C_2,C_3,C_4'\}$ is a SRCy $4$-tuple. 

Suppose instead that for some $i\in [2],$ $C_1' \cap S = C_2'\cap S = S - s_i.$  Then\\  $C_3' \cap S = \{ s_i \}.$  As before, there exists a $s_i$-semi-SRCP $\{ D_1'', D_2'' \}$ for $(M_1',c_1').$  Define $C_i,\ i = 1,2,3$ such that
$(C_1,C_2,C_3) = (C_1' \triangle S \triangle D_1'', C_2 \triangle S \triangle D_2'', C_3 \triangle D_1'')$.  Then $\{ C_1, C_2, C_3, C_4' \}$ is seen to be a SRCy $4$-tuple for $(M,c).$

\sms
\noindent{\it Case 3.2.2} Suppose that the sets $C_i' \cap S, \ i \in [3]$ are distinct.
\sms

Assume first that at least two of the circuits $C_i',\ i \in [3]$ contain only one element of $S.$  We may assume that $C_1' \cap S = \{ s_i \}$ and $C_2' \cap S = \{ s_j \}.$  Let $\{ k \} = [3] - \{ i,j \}.$  Define
$C_1 = C_1' \triangle D_i^{i,j}$ and $C_2 = C_2' \triangle D_j^{ij}$ and let $$C_3 = \left\{ \begin{array}{ll} C_3' \triangle D_k^{kj} &\mathrm{if}\ C_3' \cap S = \{ s_k \}\\ C_3' \triangle S \triangle D_k^{kj}&\mathrm{if}\ S - C_3' = \{ s_k \} \\
 C_3' \triangle S \triangle D_i^{ij}&\mathrm{if}\ S - C_3' = \{ s_i \}\\  C_3' \triangle S \triangle D_j^{ij}&\mathrm{if}\ S - C_3' = \{ s_j \}\
 \end{array} \right.$$  We see that $\{ C_1,C_2,C_3,C_4'\}$ is a SRCy $4$-tuple for $(M,c).$
 
 Assume that at least two (and hence exactly two) of the circuits $C_i',\ i \in [3]$ contain two elements of $S.$  We may assume that for $i = 1,2,$ $C_i' \cap S = \{ s_i, s_3 \}$ and $C_3' \cap S = \{ s_k \}.$  For $i=1,2,$ define $C_i = C_i' \triangle S \triangle D_{3-i}^{12}$ and define $$C_3 = \left\{ \begin{array}{ll} C_3' \triangle D_k^{12} &\mathrm{if} \ k\in [2]\\ C_3' \triangle D_{3}^{13} &\mathrm{if}\ k=3. \end{array} \right.$$  Then $\{ C_1,C_2,C_3,C_4'\}$ is seen to be a SRCy $4$-tuple for $(M,c).$

\sms
\noindent{\bf Case 4} Suppose that for all $i\in [4],$ $C_i' \cap S \ne \emptyset.$
\sms

We shall examine two sub-cases depending whether $(M_1',c_1')$ has a $S$-SRCT or not.

\ssms
\noindent{\it Case 4.1} Suppose that $(M_1',c_1')$ has an $S$-SRCT.
\ssms
 
By assumption, there is a rainbow circuit $D_3'$ where $\{D_1', D_2', D_3'\}$ is an $S$-SRCT for $(M_1', c_1').$  

\ssms
\noindent{\it Case 4.1.1}  Suppose that for all $i\in [4]$, $|C_i' \cap S| =1.$  
\ssms

For all $i\in [4]$, let $C_i' \cap S = \{ s_{i'} \}$ and let $C_i = C_i' \triangle D_{i'}'$.
Then 
\begin{align*}
\sum_i|C_i| &= \sum_i |C_i'| + \sum_i |D_{i'}| - 8\\
&\le 2(r(M_2)) + 4 + 2(r(M_1) +2) - 8\\
&= 2(r(M_1) + r(M_2)) = 2r(M) + 4.
\end{align*}
Thus $\{ C_1,C_2,C_3,C_4 \}$ is seen to be an SRCy $4$-tuple for $(M,c).$

\sms
\noindent{\it Case 4.1.2}  Suppose that for some $i\in [4]$, $|C_i' \cap S| =2.$  
\sms

We may assume that $|C_1' \cap S| =2$ and $C_1' \cap S = \{ s_1, s_3 \}.$  Since each element $s_i$ belongs to at most two of the circuits $C_j',\ j\in [4],$ it follows that for some $j\in \{ 2,3,4 \},$ $s_2 \in C_j'.$  We may assume that $s_2 \in C_2'.$  Since $f_1$ is not parallel with $s_2,$ there is a $s_2$-semi-SRCP $\{ D_1'', D_2'' \}$ for $(M_1', c_1')$ where $f_1 \not\in D_2''.$

Suppose that for $i = 2,3,4$, $|C_i'\ \cap S| = 1$ and for $i = 2,3,4,$ let $C_i' \cap S = \{ s_{i'} \}.$  Since $S \subset C_1' \cup C_2',$ we have $3' \ne 4'.$  Let $C_1 = C_1'  \triangle S \triangle D_1'',$ $C_2 = C_2' \triangle D_2''$, and for $i= 3,4,$ let $C_i = C_i' \triangle D_{i'}'.$  Then we see that
$\{ C_1, C_2, C_3, C_4 \}$ is a SRCy $4$-tuple for $(M,c).$  By the above, we may assume that for some $i\in \{ 2,3,4 \}$, $|C_i' \cap S| =2.$  Suppose first that $C_i' \cap S = C_1' \cap S = \{ s_1, s_3 \}.$  We may assume this is true for $i= 3.$
Then for $i= 2,4,$ $C_i' \cap S = \{ s_2 \}.$  For $i= 1,3,$ let $C_i = C_i' \triangle S \triangle D_1''$ and for $i = 2,4,$ let $C_i = C_i' \triangle D_2''.$  Then
\begin{align*}
\sum_i |C_i| &= \sum_i |C_i'| + 2(|D_1''| + |D_2''|) - 10\\
&\le 2r(M_2) + 4 + 2(r(M_1) +3) -10\\
&= 2(r(M_1) + r(M_2)) = 2r(M) + 4.
\end{align*}
Thus $\{ C_1,C_2,C_3,C_4\}$ is seen to be a SRCy $4$-tuple for $(M,c).$  Suppose instead that there is a circuit $C_i'$ where $|C_i' \cap S| =2$ and $C_i' \cap S \ne C_1'\cap S.$  We may assume this is true for $i=2$ and moreover, $C_2' \cap S = \{ s_2, s_3 \} ,$ $C_3'\cap S = \{ s_1 \}$ and $C_4' \cap S = \{ s_2 \}.$
Let $C_1 = C_1' \triangle S \triangle D_2',\ C_2 = C_2' \triangle S \triangle D_1'$ and $C_3 = C_1' \triangle D_1'\ C_4 = C_4' \triangle D_2'.$  Then $\{ C_1,C_2,C_3,C_4\}$ is seen to be a SRCy $4$-tuple for $(M,c).$

\ssms
\noindent{\it Case 4.2} Suppose that $(M_1',c_1')$ has no $S$-SRCT.
\ssms

For all $i,j,\ i < j,$ there exists an $S$-SRCP $\{ D_i^{ij}, D_j^{ij} \}$ for $(M_1', c_1')$ where $s_i\in D_i^{ij},$ $s_j \in D_j^{ij},$ and $f_1 \not\in D_i^{ij} \cup D_j^{ij}.$  

\sms
\noindent{\it Case 4.2.1} Suppose that for all $i\in [4]$, $|C_i' \cap S| =1.$ 
\sms

For all $i\in [4]$, let $C_i' \cap S = \{ s_{i'} \}.$
Then for some $i<j$, we have that $s_{i'} = s_{j'}.$  We may assume that $s_{1'} = s_{2'} = s_1.$  If $s_{3'} = s_{4'} = s_j,$ then define $(C_1,C_2,C_3,C_4) = (C_1' \triangle D_1^{1j}, C_2' \triangle D_1^{1j}, C_3' \triangle D_j^{1j}, C_4'\triangle D_j^{1j}).$
Otherwise, if $s_{3'} \ne s_{4'},$ then define   $(C_1,C_2,C_3,C_4) = (C_1' \triangle D_1^{13'}, C_2' \triangle D_1^{14'}, C_3' \triangle D_{3'}^{13'}, C_4' \triangle D_{4'}^{14'}).$  In either case, we see that $\{ C_1,C_2,C_3,C_4 \}$ is a SRCy $4$-tuple for $(M,c).$

\sms
\noindent{\it Case 4.2.2}  Suppose that for some $i\in [4]$, $|C_i' \cap S| = 2.$ 
\sms

 We may assume that $C_1' \cap S = \{ s_1, s_3 \}.$   Then for some $i\in \{ 2,3,4 \}$  $s_2 \in C_i'$ and we may assume this is true for $i=2.$   Suppose that for $i= 2,3,4,$ $|C_i' \cap S| = 1.$
Then $C_2' \cap S = \{ s_2 \}$ and for some $i\in \{ 3,4 \}$, $C_i' \cap S = \{ s_1\}$ or $\{ s_3 \}.$  We may assume that $C_3' \cap S = \{ s_1 \}.$  Then $C_4' \cap S = \{ s_2 \}$ or $\{ s_3 \}.$  Given that $f_1$ and $s_2$ are non-parallel, there is a $s_2$-semi-SRCP $\{ D_1'', D_2''\}$ for $(M_1',c_1')$ where $f_1 \not\in D_2''.$
Define $(C_1,C_2,C_3,C_4)$ as follows:
If $C_4' \cap S = \{ s_2 \},$ then let $$(C_1, C_2,C_3,C_4) = (C_1' \triangle S \triangle D_1'', C_2' \triangle D_2'', C_3' \triangle D_1^{12}, C_4' \triangle D_2^{12}).$$  Otherwise, if $C_4' \cap S = \{ s_3 \},$ then define
$$(C_1, C_2,C_3,C_4) = (C_1' \triangle S \triangle D_2^{12}, C_2' \triangle D_2^{23}, C_3' \triangle D_1^{12}, C_4' \triangle D_3^{23}).$$  In either case, $\{ C_1,C_2,C_3,C_4\}$ is seen to be a SRCy $4$-tuple for $(M,c).$

Lastly, suppose that for some $i\in \{ 2,3,4 \},$ $|C_i' \cap S| =2.$   Suppose that $C_i' \cap S = C_1' \cap S = \{ s_1, s_3 \}.$  We may assume that $C_3' \cap S = \{ s_1, s_3 \}$ and $C_2' \cap S = C_4' \cap S = \{ s_2 \}.$
As before, there is a $s_2$-semi-SRCP $\{ D_1'', D_2''\}$ for $(M_1',c_1')$ where $f_1 \not\in D_2''.$  Let $$(C_1, C_2, C_3, C_4) = (C_1' \triangle S \triangle D_1'', C_2' \triangle D_2'', C_3'\triangle S \triangle D_1'', C_4' \triangle D_2'').$$
Then $\{ C_1, C_2, C_3, C_4 \}$ is seen to be a SRCy $4$-tuple for $(M,c).$  Suppose now that $C_i' \cap S \ne C_1' \cap S.$  Then we may assume that $C_2' \cap S = \{ s_2, s_3 \}$ and $C_3' \cap S = \{ s_1 \}$ and $C_4' \cap S = \{ s_2 \}.$
Let $$(C_1,C_2,C_3,C_4) = (C_1' \triangle S \triangle D_2^{12}, C_2' \triangle S \triangle D_1^{12}, C_3' \triangle D_{1}^{12}, C_4' \triangle D_2^{12}).$$  Then $\{ C_1, C_2, C_3, C_4 \}$ is seen to be a SRCy $4$-tuple for $(M,c).$
\end{proof}

\begin{nonamenoname}
Suppose that $\varepsilon(M_1) = 2r(M_1) -2.$  Then $(M,c)$ has a SRC $4$-tuple.
\label{finishnonanona5}
\end{nonamenoname}


\begin{proof}
Since $\varepsilon(M_1') = 2r(M_1) -2,$ we have $\varepsilon(M_2') = 2r(M_2)$ and $r(M_1') -1 \le |c_1'(M_1')| \le r(M_1').$  If $|c_1'(M_1')| = r(M_1') -1$, then $|F| =0;$ otherwise if $|c_1'(M_1')| = r(M_1'),$ then $|F| =2.$
By Theorem \ref{the-NoRainbowCircuit} ii) there is a $S$-SRCP $\{ D_1', D_2' \}$ for $(M_1', c_1').$   We may assume that for $i= 1,2$, let $D_i' \cap S =  \{ s_{i} \}.$  Observe that if $|F| =2,$ then $F = \{ f_1, f_2 \}$ and for $i= 1,2,$  $F \cap D_i' \ne \emptyset.$  In this case, we may assume that for $i= 1,2,$
$f_i \in D_{i}'.$
Extend $c_2'$ to a colouring $c_2''$ of $M_2'' = M_2' + s_{1} + s_{2}$ where for $i=1,2$, 
$$c_2''(s_{i}) = \left\{ \begin{array}{lr}  0 &\mathrm{if} \ F = \emptyset \\ c(f_i) &\mathrm{if} \ |F| =2. \end{array}\right.$$
Then $c_2''$ is a $2$-uniform $(r(M) +1)$-colouring of $M_2''.$ 
Suppose that $(M_2'', c_2'')$ contains a rainbow circuit $C.$  Then $C\cap \{ s_{1}, s_{2} \} \ne \emptyset.$  If for some $i\in [2],$ $s_{i} \in C$ and $s_{3-i} \not\in C,$ then $C\triangle D_i'$ is a rainbow cycle in $(M,c).$  Suppose $\{ s_{1}, s_{2} \} \subset C.$  Then $|F| =2,$ $|c_1'(M_1')| = r(M_1)$ and hence there is a $s_{3}$-rainbow circuit $D$ for $(M_1',c_1').$  Now $C \triangle S \triangle D$ is seen to be a rainbow cycle in $(M,c)$, a contradiction. It follows from the above that $(M_2'', c_2'')$ is circuit-achromatic.

\begin{noname}
We may assume that $(M_2'',c_2'')$ has a SRC $4$-tuple.
\label{finishnonanona5claim1}
\end{noname}

\begin{proof}
If $M_2''$ is simple, then it follows by our assumptions that $(M_2'', c_2'')$ has a SRC $4$-tuple.  Suppose that $M_2''$ is not simple.  Then there exists $x\in E(M_2')$ and $i\in \{ 1,2 \}$ such that $x$ is parallel with $s_{i}.$  Without loss of generality, we may assume that $x$ is parallel with $s_{1}.$  Suppose that $c_2''(x) \ne c_2''(s_{1}).$  Then $\{ x,s_{1} \}$ is a rainbow 2-circuit and hence it follows by Observation \ref{obs-rainbowdigon} that $(M_2'',c_2'')$ has a SRCP (which implies that $(M_2'',c_2'')$ has a SRC $4$-tuple).
Thus we may assume that $c_2''(x) = c_2''(s_{1})$ and thus $x = f_1'.$  Now $F \ne \emptyset,$ and hence $|F| = |F'| =2.$  If $f_2'$ is parallel with $s_{2}$, then we can assume that for $i=1,2,\ f_i' \in E(M_1')$, in which case $\varepsilon_1' = 2r(M_1)$ and $\varepsilon_2' = 2r(M_1) -2.$  Now we can switch the roles of $M_1$ and $M_2$ and argue as before where $F = \emptyset.$ Thus we may assume that $f_2'$ and $s_{2}$ are non-parallel and hence $M_2' - f_1'$ is simple.

Suppose that $M_1' + f_1'$ has a rainbow circuit $D$. If $M_2' - f_1'$ has a rainbow circuit, then it follows by Observation \ref{obs-1rainbowcircorpair}, that $(M,c)$ has a SRCP.   Thus we may assume that $(M_2'-f_1',c_2')$ is circuit-achromatic.  Seeing as $M_2' - f_1'$ is simple and $\varepsilon(M_2' - f_1') = 2n_2-1,$ Theorem \ref{the-NoRainbowCircuit} i) holds for $(M_2' - f_1',c_2').$ Suppose that $(M_2' - f_1', c_2')$ has an $S$-SRCT $(C_1', C_2', C_3')$ where for $i = 1,2,$ $s_{i} \in C_i'.$ 
If $f_2' \not\in C_2',$ then letting $C_i = C_i' \triangle D_i',\ i = 1,2$, one sees that $\{ C_1,C_2\}$ is a SRCyP for $(M,c)$.  Suppose $f_2' \in C_2'.$   Since $|c_1'(M_1')| = r(M_1),$ there is an $s_{3}$-rainbow circuit $D_3'$ for $(M_1', c_1')$ where $|D_3'| \le r(M_1).$  
 Let $C_1 = C_1' - s_{1} + f_1'$ and $C_2 = C_3' \triangle D_3'.$  Then $\{ C_1,C_2\}$ is seen to be a SRCyP for $(M,c)$.  Suppose that $(M_2' - f_1', c_2')$ has no $S$-SRCT.  Then there is an $S$-SRCP $\{C_1', C_2'\}$ where $s_{i} \in C_i',\ i = 1,2$ and $f_2' \not\in C_1' \cup C_2'.$  Letting $C_i = C_i' \triangle D_i',\ i = 1,2,$ one sees that $\{ C_1,C_2\}$ is SRCyP for $(M,c).$  
 
 From the above, we may assume that  that $M_1' + f_1'$ is circuit-achromatic.  However, we can redefine $M_i,\ i = 1,2$ so that $f_1' \in E(M_1)$ and $f_1' \not\in E(M_2)$, in which case we have that $\varepsilon(M_1') = 2r(M_1) -1.$  It now follows by ({\bf \ref{finishnonanona4}}) that $(M,c)$ has a SRC $4$-tuple. 
 \end{proof}
By (\ref{finishnonanona5claim1}), we may assume that $\{ C_1', C_2',C_3',C_4'\}$ is SRC $4$-tuple for $(M_2'',c_2'').$  We may also assume that $\left(\bigcup_iC_i' \right) \cap \{ s_1, s_2 \} \ne \emptyset.$ 

\sms
\noindent{\bf Case 1} Suppose there exist distinct $i,j\in [4]$ for which $C_i' \cap \{ s_1, s_2 \} = C_j' \cap \{ s_1, s_2 \} \ne \emptyset.$  
\sms

We may assume that $C_1' \cap \{ s_1, s_2 \} = C_2' \cap \{ s_1, s_2 \} \ne \emptyset.$  Suppose that $\{ s_1, s_2 \} \subset C_1'.$  Then for $i=1,2,$ $c_2'' (s_{i}) = c(f_i)$ and $F = \{ f_1, f_2 \}.$  As before, there exists an $s_{3}$-rainbow circuit $D$ for $(M_1', c_1').$
Letting $C_i = C_i' \triangle S \triangle D,\ i = 1,2,$ we see that
$\{ C_1,C_2,C_3',C_4'\} $ is a SRCy $4$-tuple for $(M,c).$  Thus we may assume that $C_1' \cap \{ s_{1}, s_{2} \} = \{ s_1 \}.$  For $i = 1,2,$ let $C_i = C_i' \triangle D_1'$ and for $i= 3,4$ let $C_i = \left\{ \begin{array}{ll} C_i' \triangle D_2' &\mathrm{if} \ s_{2} \in C_i'\\ C_i' &\mathrm{if}\ s_{2} \not\in C_i'. \end{array} \right.$
Then $\{C_1,C_2,C_3,C_4\}$ is seen to be a SRCy $4$-tuple for $(M,c).$

\sms
\noindent{\bf Case 2}  Suppose that the nonempty sets among $C_i' \cap \{ s_1, s_2 \},\ i \in [4],$ are distinct. 
\sms

 Suppose that for some $i\in [4],$ $\{ s_{1}, s_{2} \} \subset C_i'.$  We may assume that $\{ s_{1}, s_{2} \} \subset C_1'.$  Then for $i= 2,3,4$, $|C_i' \cap \{ s_{1}, s_{2} \} | \le 1$ and it follows that for some $i \in [4],$
$C_i' \cap \{ s_1, s_2 \} = \emptyset.$  We may assume that $C_4' \cap \{ s_1, s_2 \} = \emptyset.$  Furthermore, we may assume that for $i= 2,3$, if $C_i' \cap \{ s_{1}, s_{2} \} \ne \emptyset,$ then $s_{i-1} \in C_i'.$
As in Case 1, there exists an $s_{3}$-rainbow circuit $D$ for $(M_1', c_1').$
Letting $C_1 = C_1' \triangle S \triangle D$ and for $i= 2,3,$ $$C_i = \left\{ \begin{array}{ll} C_i' \triangle D_{i-1}' &\mathrm{if} \ C_i' \cap \{ s_1, s_2 \} \ne \emptyset\\ C_i' &\mathrm{if} \ C_i' \cap \{ s_1, s_2 \} = \emptyset \end{array}\right.$$
one sees that $\{ C_1,C_2,C_3,C_4' \}$ is a SRCy $4$-tuple for $(M,c).$
Suppose instead that for all $i\in [4],$ $|C_i' \cap \{ s_{1}, s_{2} \}| \le 1.$  For all $i\in [4],$ if $C_i' \cap \{ s_{1}, s_{2} \} \ne \emptyset$, then let $C_i' \cap \{ s_{1}, s_{2} \} = \{ s_{i'} \}.$ For all $i\in [4],$ let
$$C_i = \left\{ \begin{array}{ll} C_i' \triangle D_{i'}' &\mathrm{if}\ C_i' \cap \{ s_{1}, s_{2} \} \ne \emptyset\\ C_i' &\mathrm{if}\ C_i' \cap \{ s_{1}, s_{2} \} = \emptyset .\end{array} \right.$$
Then $\{ C_1, C_2, C_3, C_4 \}$ is a SRCy $4$-tuple for $(M,c).$
\end{proof}

\begin{nonamenoname}
Suppose that $\varepsilon(M_1) = 2r(M_1) -3.$  Then $(M,c)$ has a SRC $4$-tuple.
\label{finishnonanona6}
\end{nonamenoname}

\begin{proof}
Since $(M_1',c_1')$ is circuit-achromatic, we have $|c_1'(M_1')| \le r(M_1)$ and thus $|F| = 1$ or $|F| = 3.$  We also have that $\varepsilon_2' = 2r(M_2) +1.$

\begin{noname}
For some $j\in [3],$ there is an $s_j$-rainbow circuit $D$ for $(M_1', c_1')$ where $|D\cap F| \le 1$ and $|D| \le r(M_1) -1.$  
\label{finishnonanona6claim1}
\end{noname}

\begin{proof}
 Suppose that $M_1 \big| F \cup S \simeq M(K_4).$  Let $M_2'' = M_2' + F.$  We see that $r(M_2'') = r(M_2) + 1$ and $$\varepsilon(M_2'') = \varepsilon_2' + 3 = 2r(M_2) + 4 = 2(r(M_2'') +1).$$ This violates ({\bf \ref{finishnonanona0}}).
Thus $M_1 \big| F \cup S\not\simeq M(K_4)$ and it follows by Lemma \ref{lem-helplemma} that for some $j\in [3]$, $(M_1', c_1')$ has an $s_j$-rainbow circuit $D$ for which $|D \cap F| \le 1$ and $|D| \le r(M_1) -1.$  
\end{proof}

By (\ref{finishnonanona6claim1}), for some $i\in [3]$ there is an $s_i$-rainbow circuit $D$ for $(M_1',c_1')$ where $|D\cap F| \le 1$ and $|D| \le r(M_1) -1.$  Here we may assume this is true for $i=1.$
Extend the colouring $c_2'$ to a colouring $c_2''$ of $M_2'' = M_2' + s_1,$ where $$c_2''(s_1) = \left\{ \begin{array}{ll} c(f_i) &\mathrm{if}\ D\cap F = \{ f_i\}\\ c(f_1) &\mathrm{if} \ D\cap F = \emptyset \end{array} \right.$$
We see that $c_2''$ is a $2$-uniform $(r(M_2) +1)$-colouring of $M_2''.$  Moreover, $(M_2'',c_2'')$ is seen to be circuit-achromatic.

\begin{noname}
We may assume that there is a SRC $4$-tuple for $(M_2'', c_2'').$
\label{finishnonanona6claim2}
\end{noname}

\begin{proof}
If $M_2''$ is simple, then it follows by assumption that $(M_2'',c_2'')$ has SRC $4$-tuple.  Suppose $M_2''$ is not simple.  Then there is an element $x \in E(M_2')$ which is parallel with $s_1.$
If $c_2''(x) \ne c_2''(s_1),$ then $\{ x, s_1 \}$ is a rainbow $2$-circuit of $(M_2'', c_2'').$  It follows by Observation \ref{obs-rainbowdigon} that $(M_2'', c_2'')$ has a SRCP.  Thus we may assume that $c_2''(x) = c_2''(s_1)$ and thus $x \in F'.$
Suppose that $M_1' + x$ contains a rainbow circuit $D'.$  Since $\varepsilon(M_2' - x) = 2r(M_2)$ and $M_2' - x$ is simple, it follows by Corollary \ref{cor-cor2} that $(M_2'-x,c_2')$ has a rainbow circuit.  It now follows by
Observation \ref{obs-1rainbowcircorpair} that $(M,c)$ has a SRCP.  Thus we may assume that $M_1' + x$ is circuit-achromatic. However, replacing $M_1$ by $M_1 + x$ and $M_2$ by $M_2 -x,$ it follows by ({\bf \ref{finishnonanona5}}) that $(M,c)$ has a SRC $4$-tuple. 
\end{proof}

By (\ref{finishnonanona6claim2}), we may assume that $(M_2'', c_2'')$ has a SRC $4$-tuple $\{ C_1',C_2',C_3',C_4' \}$ and we may also assume that $s_1 \in C_1'.$ 
If $s_1 \in C_2' \cup C_3' \cup C_4',$ then we may assume that $s_1 \in C_2'.$  Let $C_1 = C_1' \triangle D$ and let $C_2 = \left\{ \begin{array}{ll} C_2' \triangle D&\mathrm{if}\ s_1 \in C_2'\\ C_2' &\mathrm{otherwise.} \end{array} \right.$
Then $(C_1,C_2,C_3',C_4')$ is seen to be a SRCy $4$-tuple for $(M,c).$
\end{proof}

\end{document}